\documentclass[12pt,xcolor=dvipsnames]{amsart}

\usepackage[T1]{fontenc}
\usepackage[utf8]{inputenc}

\usepackage{amsmath,amsthm,amssymb,url,mathdots,verbatim,mathtools,nccmath}
\usepackage[margin = 2.5cm]{geometry}
\usepackage{tikz}
\usepackage{ytableau}
\usepackage{xcolor}
\usepackage{booktabs}
\usepackage{MnSymbol}
\usepackage{graphicx}
\usepackage{scalerel}
\usepackage{enumitem}
\usepackage{diagbox}
\usepackage[normalem]{ulem}
\numberwithin{equation}{section}

\usepackage[colorlinks=true]{hyperref}

\newtheorem{theorem}{Theorem}

\newtheorem{prop}[theorem]{Proposition}
\newtheorem{cor}[theorem]{Corollary}
\newtheorem{lemma}[theorem]{Lemma}
\newtheorem{definition}[theorem]{Definition}
\newtheorem{conj}{Conjecture}

\theoremstyle{definition}

\newtheorem{example}[theorem]{Example}
\newtheorem{remark}[theorem]{Remark}

\DeclareMathOperator{\gt}{gt}
\DeclareMathOperator{\id}{id}

\DeclareMathOperator{\sgt}{sgt}
\DeclareMathOperator{\TGT}{TGT}
\DeclareMathOperator{\T}{T}
\DeclareMathOperator{\RGT}{RGT}
\DeclareMathOperator{\GT}{GT}
\DeclareMathOperator{\MT}{MT}

\DeclareMathOperator{\st}{Strict}

\DeclareMathOperator{\sgn}{sgn}

\DeclareMathOperator*{\pf}{Pf}

\usepackage{enumitem}

\newcommand{\e}{{\operatorname{E}}}

\newcommand{\asym}{\mathbf{ASym}}
\newcommand{\sym}{\mathbf{Sym}}

\newcommand{\cart}{\bigtimes}
\newcommand{\const}[2]{{\mathbf{#1}}^{#2}}

\newcommand{\Dfn}[1]{\emph{\color{violet}#1}} 

\title[Monotone trapezoids]{The number of monotone trapezoids with prescribed bottom row}

\author{Ilse Fischer and Hans H\"ongesberg}
\thanks{The authors acknowledge support from the Austrian Science Fund (FWF) grants 
\href{https://dx.doi.org/10.55776/F1002}{10.55776/F1002} (IF),
\href{https://dx.doi.org/10.55776/P34931}{10.55776/P34931} (HH), 
\href{https://dx.doi.org/10.55776/J4810}{10.55776/J4810} (HH)}

\begin{document} 

\begin{abstract} 
 We establish an operator formula for the number of monotone trapezoids with prescribed bottom row, generalizing alternating sign matrices. The special case of the formula for monotone triangles previously provided an alternative proof for the enumeration of alternating sign matrices and led to several results on alternating sign triangles and alternating sign trapezoids. The generalization presented in this paper reveals an additional ``hidden operator'' that is annihilated in the special case of monotone triangles, whose discovery was a major challenge.

The enumeration formula is conceptually simple: it applies, in addition to the newly discovered hidden operator, an operator ensuring row strictness to the formula for the number of Gelfand--Tsetlin trapezoids. Notably, the top row of the monotone trapezoid is not prescribed. Thus, our result involves a free boundary, which is a novel situation in this area. We also uncover an unexpected relation to the coinvariant algebra and propose a conjecture on its generalization.
\end{abstract}

\maketitle 

\section{Introduction} 
\Dfn{Gelfand--Tsetlin patterns} were introduced in \cite{gelfandUni}  in connection with the representation theory of the Lie algebra $\mathfrak{sl}(n,\mathbb{C})$. They are triangular arrays of integers that satisfy certain monotonicity conditions, see Definition~\ref{GTT} for details. Using their representation theoretical meaning  
as well as the {Weyl dimension formula}, it can be shown that the number of Gelfand--Tsetlin patterns with fixed bottom row $(k_1,\ldots,k_n)$ is given by 
\begin{equation}
\label{weyl} 
\prod_{1 \le i < j \le n} \frac{k_j-k_i+j-i}{j-i}.
\end{equation} 
However, there are also several possibilities to prove this formula that are purely based on their combinatorial definition. For instance, one can easily establish a bijection with certain families of {non-intersecting lattice paths} and then the {lemma of Lindstr\"om--Gessel--Viennot}  \cite{Lin73,GesVie85,GesVie89} reduces the problem to the calculation of a determinant, which can in turn be evaluated using the {Vandermonde determinant evaluation}, see for instance \cite[Appendix~A]{Fis12}. The most combinatorial way to enumerate Gelfand--Tsetlin patterns with prescribed bottom row is probably to use the bijection with {semistandard Young tableaux} of shape $(k_n,\ldots,k_1)$ (see \cite[p. 314]{Sta99}) and then the enumeration formula follows by rewriting Stanley's {hook-content formula} \cite[Theorem 7.21.2]{Sta99}; the latter formula has several bijective proofs, see for instance \cite{remmel1983,kratthook1,kratthook2}. A proof by the first author that explains the zeros of the formula essentially by providing a combinatorial interpretation for all $(k_1,\ldots,k_n) \in \mathbb{Z}^n$ is demonstrated in \cite{Fis05}.

\Dfn{Monotone triangles} are Gelfand--Tsetlin patterns where all rows are \Dfn{strictly increasing}. They have been introduced in \cite{MilRobRum86} in connection with \Dfn{alternating sign matrices} as monotone triangles with bottom row $(1,\ldots,n)$ are in easy bijective correspondence with $n \times n$ alternating sign matrices. Whether there is also a connection of 
monotone triangles to representation theory is not clear at this point, however, some indications are given in \cite{Oka06,bounded} as well as in forthcoming works of the authors and Gangl and the first author on generalizations of the Littlewood identity as well as the Cauchy identity that involve alternating sign matrices. The interest in alternating sign matrices is currently also stimulated through their appearance in other fields such as probability, see \cite{chhita} for one example\footnote{Other examples are formulated in terms of six-vertex configurations which are alternating sign matrices in disguise.}, and, in disguise as bumpless pipe dreams, in algebraic geometry and algebraic combinatorics, see \cite{weigandt}.

From the point of view of enumerative combinatorics an interesting question is whether the number of monotone triangles is also expressible by some enumeration formula. For most enumeration problems this is certainly not the case. Clearly, the answer also depends on what we consider to be an acceptable formula. Compared to Gelfand--Tsetlin patterns, it is significantly harder to establish an enumeration formula for monotone triangles and it also seems that one has to admit a wider class of enumeration formulas in so far that we also allow the involvement of shift operators.  In \cite{Fis06}, it was shown that the number of monotone triangles with bottom row $(k_1,\ldots,k_n)$ is given by the following operator formula 
\begin{equation}
\label{operatorweyl} 
\prod_{1 \le i < j \le n} (\e^{-1}_{k_i}+\e_{k_j} - \e^{-1}_{k_i} \e_{k_j})^{-1} \prod_{1 \le i < j \le n} \frac{k_j-k_i+j-i}{j-i}, 
\end{equation} 
where $\e_x$ denotes the shift operator, defined as $\e_x p(x) \coloneq p(x+1)$ for any polynomial~$p$ in $x$. This formula was key to an alternative proof for enumerating $n \times n$ alternating sign matrices, see \cite[Theorem~2]{Fis07}. It is worth mentioning that
the factors $ (\e^{-1}_{k_i}+\e_{k_j} - \e^{-1}_{k_i} \e_{k_j})^{-1}$ in \eqref{operatorweyl} can be replaced by $\e^{-1}_{k_j}+\e_{k_i} - \e^{-1}_{k_j} \e_{k_i}$, see \cite{Fis06}.

There is an intuition behind the operator $\e^{-1}_{k_i}+\e_{k_j} - \e^{-1}_{k_i} \e_{k_j}$ that is related to the strict increase along rows in a monotone triangle. This is explained  in the next section. The formula in \eqref{operatorweyl} is equivalent to a constant term formula as well as to a formula that reduces the counting problem  to the evaluation of the symmetrizer of a rational function in $x_1,\ldots,x_n$ at $(x_1,\ldots,x_n)=(1,\ldots,1)$; the complication in the evaluation of the latter stems from the fact that the rational function is singular at  $(x_1,\ldots,x_n)=(1,\ldots,1)$, see 
\cite[Cor. 3 and Rem. 2.1.]{fischergog}. A multivariate version of this formula, which includes \Dfn{Schur polynomials} as well as \Dfn{Hall--Littlewood polynomials} and \Dfn{symmetric Grothendieck polynomials} as special cases, has recently been provided in \cite{nASMDPP}.

As for Gelfand--Tsetlin patterns, a next level of complication is reached when considering \Dfn{Gelfand--Tsetlin trapezoids}, which are defined as Gelfand--Tsetlin patterns with a chopped-off triangle at the top (Definition~\ref{GTT}). The number of Gelfand--Tsetlin trapezoids with prescribed bottom row of given height can be expressed using a \Dfn{Pfaffian}, see Theorem~\ref{GT}. A proof of the result is presented in Section~\ref{GTproof}. Another proof that generalizes the one for \eqref{weyl}  using the Lindstr\"om--Gessel--Viennot Lemma can be accomplished by invoking an extension of Stembridge \cite{ste90} of that lemma. Let us emphasize that the top row of the Gelfand--Tsetlin trapezoid is not prescribed and thus we have an instance of a \Dfn{free boundary} here, which is a new situation in this area.

Now the question is whether we can also find a formula for the number of \Dfn{monotone trapezoids} with prescribed bottom row of given height, where a monotone trapezoid is a Gelfand--Tsetlin trapezoid with strictly increasing rows. Comparing \eqref{weyl} and \eqref{operatorweyl}, it is tempting to guess that the number of monotone trapezoids with bottom row $k_1,\ldots,k_n$ of height $h+1$ is 
\begin{equation}
\label{wrong}  
\prod_{1 \le i < j \le n} (\e^{-1}_{k_i}+\e_{k_j} - \e^{-1}_{k_i} \e_{k_j})^{-1} \GT_h(k_1,\ldots,k_n), 
\end{equation} 
where $\GT_h(k_1,\ldots,k_n)$ denotes the number of Gelfand--Tsetlin trapezoids with bottom row $k_1,\ldots,k_n$ of height $h+1$. 
As a matter of fact, this guess is wrong, but it is not far from the truth as the main result (Theorem~\ref{main}) of this paper will show. It is worth noting that the main result also includes a possibility of replacing $(\e^{-1}_{k_i}+\e_{k_j} - \e^{-1}_{k_i} \e_{k_j})^{-1}$ with $\e^{-1}_{k_j}+\e_{k_i} - \e^{-1}_{k_j} \e_{k_i}$ as it is the case for 
monotone triangles.

The reader may have noted that the enumeration formulas we have discussed so far are all polynomials in $k_1,\ldots,k_n$. There is a simple reason for that, which is discussed in Section~\ref{additive}. The degree grows linearly for any fixed $k_i$  with the height for all these enumeration formulas. In the case of Gelfand--Tsetlin patterns and monotone triangles, this degree is $n-1$ if the array has $n$ rows; in the case of Gelfand--Tsetlin trapezoids the degree is $2h$ for arrays with $h+1$ rows. This is surprising as a straightforward degree estimation would suggest in all cases an exponential growth of the degree with $h$, see also Section~\ref{additive}. Now all of this applies also to monotone trapezoids, in particular already \cite[Section~6, (3)]{Fis06} implies that 
the degree also grows linearly in this case, without knowing the formula. Since then, the first author believed that there exists an operator formula 
for the number of monotone trapezoids that extends \eqref{operatorweyl}. Finding the formula ended up in a long journey and an involved proof, where the key was the discovery of a certain additional operator that is annihilated in the case of monotone triangles. This operator is not unique, but one possibility is 
$\prod_{1 \le i < j \le n} A(\Delta_{k_i},\Delta_{k_j})$ where 
$$
A(x,y) = (1+x+y) \frac{P(x,y) P(y,x)}{P(\iota(x),y) P(\iota(y),x)}
$$
with $\iota(x)=- \frac{x}{1+x}$, $P(x,y)=1- (x+1) (y+1) - e^{\frac{i \pi}{3}} (x+2)$ and $\Delta_x \coloneq \e_x - \id$.

\medskip

We conclude the introduction by indicating some unexpected relation to the \Dfn{coinvariant algebra} (see \cite{bergeron}) that is related to the fact that the operator which is needed to correct \eqref{wrong} is not unique. The classical coinvariant algebra is known to be isomorphic to
$$
\left\{\left. R(\Delta_{k_1},\ldots,\Delta_{k_n}) \prod_{1 \le i < j \le n} \frac{k_j-k_i+j-i}{j-i} \right| \left. R(X_1,\ldots,X_n) \in \mathbb{Q}[X_1,\ldots,X_n] \right. \right\}, 
$$
which follows from the fact that the ideal of $R(X_1,\ldots,X_n) \in \mathbb{Q}[X_1,\ldots,X_n]$ with 
$$R(\Delta_{k_1},\ldots,\Delta_{k_n}) \prod_{1 \le i < j \le n} \frac{k_j-k_i+j-i}{j-i}=0$$
is generated by the symmetric polynomials without constant term. Now  
\begin{equation}
\label{coinvar} 
\left\{\left. R(\Delta_{k_1},\ldots,\Delta_{k_n}) \GT_h(k_1,\ldots,k_n) \right| \left. R(X_1,\ldots,X_n) \in \mathbb{Q}[X_1,\ldots,X_n] \right. \right\}
\end{equation} 
certainly generalizes the coninvariant algebra, since the latter is obtained by specializing $h=n-1$ in \eqref{coinvar}. In Conjecture~\ref{anni}, we offer a conjectural set of generators of the annihilator ideal for this generalization.

\section{Main result} 
\label{sec:main} 

The purpose of this paper is to enumerate the objects defined next.

\begin{definition} 
\label{GTT}
Let $h,n$ be non-negative integers with $n \ge h$.  
An \Dfn{$(h,n)$-Gelfand--Tsetlin trapezoid} (or short \Dfn{$(h,n)$-GT trapezoid}) is an array of integers of the following form 
$$
\begin{array}{ccccccccccccc}
& &&  a_{0,1} & & a_{0,2} & & \dots & & a_{0,n-h} &&& \\
&&a_{1,1} & & a_{1,2} & & \dots & & \dots & & a_{1,n-h+1}&& \\ 
& \iddots &   & \iddots & & \dots & & \dots & & \dots & &  \reflectbox{$\iddots$}& \\
a_{h,1}  & & a_{h,2} & & \dots & & \dots & & \dots & & \dots & & a_{h,n} 
\end{array} 
$$
with weak increase along $\nearrow$-diagonals and along $\searrow$-diagonals. An \Dfn{$(h,n)$-monotone trapezoid} is an 
$(h,n)$-GT trapezoid with strictly increasing rows.

Note that $(n-1,n)$-GT trapezoids are also known as \Dfn{Gelfand--Tsetlin patterns} and that 
$(n-1,n)$-monotone trapezoids are also known as \Dfn{monotone triangles}. Also the somewhat degenerate case of $(n,n)$-GT trapezoids are Gelfand--Tsetlin patterns, and, similarly, $(n,n)$-monotone trapezoids are monotone triangles.

\end{definition} 

\begin{example}\label{ex:GTtrapezoid}
A $(3,7)$-GT trapezoid is displayed next.
 $$
\begin{array}{ccccccccccccc}
& &&  8 & & 12 & & 15 & & 18 &&& \\
&& 7 & & 10 & & 15 & & 17 & &  19 && \\ 
& 6 &   & 8 & & 14 & & 15 & & 17 & &  19& \\
3  & & 8 & & 12 & & 15 & & 16 & & 18 & & 20 
\end{array} 
$$
It is also a $(3,7)$-monotone trapezoid as the rows increase strictly.
\end{example} 

\medskip

Extending the standard bijection between monotone triangles and alternating sign matrices, it can be shown that $(h,n)$-monotone triangles with bottom row $(k_1,\ldots,k_n)$ correspond to $h \times k_n$ matrices with entries in $\{0,\pm 1\}$ such that
\begin{itemize} 
	\item non-zero entries alternate in each row and column, 
	\item all row sums are $1$, and
	\item the bottom most non-zero entry of a column (if existent) is $1$ for those columns that are in positions $k_1,\ldots,k_n$ and it is $-1$ (if existent) for the others.
	\end{itemize}
In fact, mapping an $(h,n)$-monotone triangle to a $\{0,\pm 1\}$-matrix by extending the standard bijection yields an $(h+1) \times k_n$ matrix whose first row sums to $n-h$ and contains no entries equal to $-1$. We can, however, delete the first row since it is the unique $\{0,1\}$-row such that all columns at positions $k_1,\ldots,k_n$ sum to $1$ and to $0$ otherwise. In the standard case $h=n-1$, this first row contains a single $1$ and, thus, the non-zero entries alternate in sign.

\begin{example}
The $(3,7)$-monotone trapezoid from Example~\ref{ex:GTtrapezoid} is mapped to the following $3 \times 20$ matrix:
\begin{equation*}
	\begin{pmatrix}
	0 & 0 & 0 & 0 & 0 & 0 & 1 &-1 & 0 & 1 & 0 &-1 & 0 & 0 & 0 & 0 & 1 &-1 & 1 & 0\\
	0 & 0 & 0 & 0 & 0 & 1 &-1 & 1 & 0 &-1 & 0 & 0 & 0 & 1 & 0 & 0 & 0 & 0 & 0 & 0\\
	0 & 0 & 0 & 1 & 0 &-1 & 0 & 0 & 0 & 0 & 0 & 1 & 0 &-1 & 0 & 1 &-1 & 1 &-1 & 1
	\end{pmatrix}
\end{equation*}
\end{example}
 
Throughout the article, we let
$\mathbf{k}_n \coloneq(k_1,\ldots,k_n)$ where $\mathbf{k}$ may also be replaced by other  letters. We define
\begin{align*} 
\GT_h(\mathbf{k}_n) & \coloneq \# \text{ of $(h,n)$-GT trapezoids with bottom row $\mathbf{k}_n$,} \\ 
\MT_h(\mathbf{k}_n) &\coloneq \# \text{ of $(h,n)$-monotone trapezoids with bottom row $\mathbf{k}_n$.}
\end{align*} 

The goal of this paper is to provide an \Dfn{operator formula} for $\MT_h(\mathbf{k}_n)$. It will be expressed using 
$\GT_h(\mathbf{k}_n)$ for which a formula is (essentially) known in terms of a Pfaffian as stated in Theorem~\ref{GT}.

In order to state the latter, we recall the definition of the \Dfn{Pfaffian} of an upper triangular array of even order: a pairing 
$$\pi=\{(i_1,j_1),(i_2,j_2),\ldots,(i_n,j_{n})\}$$ of the integers in $[2n]=\{1,2,\ldots,2n\}$ is said to be a matching, since it can be 
interpreted as a perfect matching of the complete graph $K_{2n}$. Assuming $i_l < j_l$ for all $l$, we define the sign of $\pi$ as the sign of 
the permutation $i_1 j_1 i_2 j_2 \ldots i_n j_n$. Note that the sign is well-defined as the sign of the permutation does not depend on the order of the pairs $(i_l,j_l)$. Let $A=(A_{i,j})_{1 \le i < j \le 2n}$ be an upper triangular array of order $2n$. The Pfaffian $\pf_{1 \le i < j \le 2n} \left( A \right)$ 
of $A$ is the sum over all matchings of $\pi =\{(i_1,i_2),\ldots,(i_n,j_n)\}$ with summand 
\begin{equation}
\label{displayed} 
\sgn \pi A_{i_1,j_1} A_{i_2,j_2} \ldots A_{i_n,j_n}.
\end{equation}
Note that if we extend $A$ to a $2n \times 2n$ skew-symmetric matrix, then we do not need to insist on $i_l < j_l$ in \eqref{displayed} as sign changes in the permutation that occur in the context of entries $A_{i_l,j_l}$ such that $i_l>j_l$ are compensated by the sign of $A_{i_l,j_l}$. It is also well-known that we have 
$$
\left[ \pf_{1 \le i < j \le 2n} (A_{i,j}) \right]^2= \det_{1 \le i,j \le 2n} (A_{i,j})
$$
then.

We use the following definition of the binomial coefficient for non-negative integers $k$ and 
integers $n$
$$
\binom{n}{k} \coloneq \frac{n(n-1)\cdots (n-k+1)}{k!}
$$
as well as the \Dfn{shift operator} $\e_x$ and the \Dfn{(forward) difference operator} $\Delta_x$, which act on functions $f(x)$ as follows:
$$
\e_x f(x) \coloneq f(x+1) \quad \text{and} \quad \Delta_x f(x) \coloneq f(x+1)-f(x).
$$
We also extend this notion to vectors, that is, for $\mathbf{x}=(x_1,\ldots,x_n)$ and functions $f(\mathbf{x})$, we 
set 
\begin{equation}
\label{vectors}
\e_{\mathbf{x}} f(\mathbf{x}) \coloneq \e_{x_1} \ldots \e_{x_n} f(x_1,\ldots,x_n) \quad \text{and} \quad
\Delta_{\mathbf{x}} f(\mathbf{x}) \coloneq  \Delta_{x_1} \ldots \Delta_{x_n} f(x_1,\ldots,x_n).
\end{equation} 
The following bivariate polynomial depending on a non-negative integer $h$ is crucial for what follows. 
We set $\gt_0(x,y)=1$, and, for $h \ge 1$, we set 
$$
\gt_h(x,y) \coloneq \sum_{p=h}^{x-1} \left( \binom{x-p+h-1}{h} \binom{y-p+h-2}{h} - \binom{x-p+h-2}{h} \binom{y-p+h-1}{h}  \right).
$$
Moreover, the following triangular array of size $n$ as well as the following  $n \times h'$ matrix will be used throughout the article:
$$
\TGT_{h}(\mathbf{k}_n) \coloneq \left( \gt_h(k_i,k_j) \right)_{1 \le i<j \le n} \quad \text{and} \quad 
\RGT_{h,h'}(\mathbf{k}_n) \coloneq \left( \binom{k_i}{h-j} \right)_{1 \le i \le n, 1 \le j \le h'}.
$$
Here $\TGT$ indicates that we have a triangular array associated with GT-trapezoids, while $\RGT$ indicates that we have a rectangular array associated with GT-trapezoids.
We also set 
$$\mathbf{k\!\uparrow}_n \coloneq \mathbf{k}_n + (1,2,\ldots,n)
\quad \text{and} \quad \T(0)_h \coloneq (0)_{1 \le i < j \le h},$$
where addition of vectors is always component wise in this paper.

With this notation, we are now in the position to state the formula for the number of Gelfand--Tsetlin patterns with fixed bottom row.

\begin{theorem} 
\label{GT} 
Let $h$ and $n$ be non-negative integers with $n \ge h$.
\begin{enumerate} 
\item In the case that $h+n$ is even, we have 
$$
\GT_h(\mathbf{k}_n) =  \pf \left( 
\begin{array}{c|c}  
\TGT_{h}(\mathbf{k\!\uparrow}_n) & \RGT_{h,h}(\mathbf{k\!\uparrow}_n) \\ \hline 
                                                     & \T(0)_h 
\end{array} 
\right).
$$
\item In the case that $h+n$ is odd, we have 
$$
\GT_h(\mathbf{k}_n) = \Delta^h_{k_{n+1}} \GT_h(\mathbf{k}_{n+1}).
$$
\end{enumerate} 
\end{theorem} 

{\bf The strict operator.} In order to state the main result, we introduce the following crucial operator $\st_{x,y}$, which is defined through shift operators with respect to two variables $x,y$:
$$
\st_{x,y} \coloneq \e_x^{-1} + \e_y - \e_x^{-1} \e_y.
$$
There is the following intuition behind this operator, which also explains the naming. Let us first consider Gelfand--Tsetlin trapezoids with bottom row $k_1,\ldots,k_n$ and let 
$l_1,\ldots,l_{n-1}$ denote the preceding row, so that the bottom two rows are as follows:
$$
\begin{array}{ccccccccccccc}
& l_{1}  && \dots && l_{i-1} && l_i && \dots && l_{n-1} & \\
k_1  && \dots  && k_{i-1} && k_i &&  k_{i+1} && \dots && k_n
\end{array} 
$$
We concentrate on the situation locally around $k_i$: by the definition of GT trapezoids, we require 
$$
(l_{i-1},l_i)  \in  [k_{i-1},k_i] \times [k_i,k_{i+1}],
$$
where $[a,b] \coloneq \{a,a+1,\ldots,b\}$ for integers $a,b$. On the other hand, describing the situation locally around $k_i$ in a monotone trapezoid, this transforms into
$$
(l_{i-1},l_i) \in   \left. \st_{k^L_i,k^R_i} [k_{i-1},k^L_i] \times [k^R_i,k_{i+1}]   \right|_{k^L_i=k^R_i=k_i},
$$
where ``$+$'' is interpreted as the union and ``$-$'' is interpreted as the complement, and we have to work with multisets so that multiplicities are taken into account. Indeed, 
\begin{multline} 
\label{strictop}
 \left. \st_{k^L_i,k^R_i} [k_{i-1},k^L_i] \times [k^R_i,k_{i+1}]   \right|_{k^L_i=k^R_i=k_i} \\
 = [k_{i-1},k_i-1] \times [k_i,k_{i+1}] + [k_{i-1},k_i] \times [k_i+1,k_{i+1}] - [k_{i-1},k_i-1] \times [k_i+1,k_{i+1}] \\ 
 = [k_{i-1},k_i-1]  \times [k_i,k_{i+1}] + \{k_i\} \times [k_i+1,k_{i+1}],
 \end{multline}
 and now the latter set is obviously the set of all $(l_{i-1},l_i)$ such that $k_{i-1} \le l_{i-1} \le k_i \le l_{i} \le k_{i+1}$ and 
 $l_{i-1} < l_i$. 

It is also important to note that $\st_{x,y}$ is invertible when applying to polynomials in $x$, as 
$$
\st_{x,y} = \e_{x}^{-1} ( \id + \e_y \Delta_x) 
$$
and, by the geometric series expansion, 
$$
(\id + \e_y \Delta_x)^{-1} = \sum_{i \ge 0} (-1)^i \e_y^{i} \Delta^i_x.
$$
Since any polynomial in $x$ of degree $m$ vanishes when applying $\Delta^i_x$ for $i>m$, this sum is effectively finite.

\bigskip

{\bf The ``hidden'' operator.} We need a second bivariate operator. The naming of the operator stems from the fact that it is hidden in the special case of monotone triangles for which the formula was proven about $20$ years ago. Finding this operator was one major obstacle that we needed to overcome in this work, and it is the reason why it took so long to find the extension.

The operator is symmetric in $x,y$ and, to some extend, its purpose is to annihilate $\st_{x,y}$ at one important place in the computation. There is in fact a choice for this operator. In the following, we set
\begin{equation}
\label{iota}  
\iota(x) \coloneq - \frac{x}{1+x}.
\end{equation}
Note that $\iota(x)$ is an involution.  
A hidden  formal power series $A(x,y)$ is obtained 
by first choosing a formal power series 
$P(x,y)$ with non-vanishing constant term that is a solution of 
\begin{equation} 
\label{P}
\frac{P(x,\iota(x)) P(\iota(x),x)}{P(x,x) P(\iota(x),\iota(x))} = \frac{1+x}{1+x+x^2}
\end{equation} 
and then setting  
$$
A(x,y) \coloneq  (1+x+y) \frac{P(x,y) P(y,x)}{P(\iota(x),y) P(\iota(y),x)} 
\quad \text{or} \quad
A(x,y) \coloneq  (1- x y) \frac{P(x,y) P(y,x)}{P(\iota(x),y) P(\iota(y),x)} .
$$
Note that $A(x,y)$ is a formal power series with constant term $1$. The associated operator acting on polynomials in variables $x,y$ is then defined as 
\begin{equation}
\label{A} 
A_{x,y} \coloneq A(\Delta_x,\Delta_y).
\end{equation}

Possible choices for $P(x,y)$'s are 
$$
1- (x+1) (y+1) - \rho (x+2) \quad \text{and} \quad 1-(x+1)(y+1) + \rho (x+1) (x+2),
$$
where $\rho$ is a primitive sixth root of unity.
Also  
\begin{multline*} 
\sqrt{(1- (x+1) (y+1) - \rho (2+x))(1- (x+1) (y+1) - \rho^5(x+2))} \\ 
= \sqrt{4+ 6 x + 3 x^2+2 y +y^2 + 5 x y + 3 x^2 y  + 2 x y^2 +x^2 y^2}
\end{multline*} 
 and 
\begin{multline*} 
\sqrt{(1-(x+1) (y+1) + \rho (x+1) (x+2))(1-(x+1) (y+1) + \rho^5 (x+1) (x+2))} \\
=\sqrt{4+10 x + 11 x^2+5 x^3+x^4- 2 y+y^2-3 x y-2 x^2 y-x^3 y+2 x y^2+x^2 y^2}
\end{multline*} 
are solutions in the ring of formal power series; they have rational coefficients using 
$$
\sqrt{1 + q(x,y)} = \sum_{i \ge 0} \binom{1/2}{i} q(x,y)^i,
$$
where $q(x,y)$ is a formal power series with vanishing constant term. 

Two other interesting choices for $P(x,y)$ are $P(x,y)=(1+x + x y) Q(x,y)$ with
$$
Q(x,y)= x y + \rho x  + \rho^{-1} y - 2  \quad \text{or} \quad  Q(x,y)= x y + \left(1+ \frac{\rho}{2} \right) x  + \left(1+\frac{\rho^{-1}}{2} \right) y + 1. 
$$
One can deduce from \eqref{P} that $Q(x,y)$ can be replaced by any formal power series $Q(x,y)$ with non-vanishing constant term that is a solution of
\begin{equation}
\label{Q}
\frac{Q(x,\iota(x)) Q(\iota(x),x)}{Q(x,x) Q(\iota(x),\iota(x))} = \frac{1+x+x^2}{(1-x)(1+2x)}.
\end{equation} 

\medskip

Our main theorem is as follows.

\begin{theorem} 
\label{main} 
Let $h,n$ be non-negative integers with $n \ge h$ and $A(x,y)$ a hidden formal power series. Then the number of $(h,n)$-monotone trapezoids with bottom row $\mathbf{k}_n$ is 
$$
\MT_h(\mathbf{k}_n) = \prod_{1 \le i < j \le n} \st^{-1}_{k_i,k_j} A_{k_i,k_j} \GT_h(\mathbf{k}_n).
$$
\end{theorem} 

The formula has to be applied as follows: first note that $\GT_h(\mathbf{k}_n)$ is clearly a polynomial in $k_1,\ldots,k_n$. We need to apply the operator $\prod_{1 \le i < j \le n} \st^{-1}_{k_i,k_j} A_{k_i,k_j}$ to this polynomial and only then specialize $k_1,\ldots,k_n$ to specific integers values.

\begin{remark}\label{rem:Strict_from_denom_to_numer}
Note that when choosing $A(x,y) =  (1- x y) \frac{P(x,y) P(y,x)}{P(\iota(x),y) P(\iota(y),x)}$ with $P(x,y) = (1+x+ x y) Q(x,y)$ and $Q(x,y)$ as given in \eqref{Q} then this somewhat removes $\st_{k_i,k_j}$ from the denominator and replaces it by $\st_{k_j,k_i}$ in the numerator. Clearly, we have 
$
\st_{k_i,k_j} = \e_{k_i}^{-1} (\id + \Delta_{k_i} + \Delta_{k_i} \Delta_{k_j}), 
$
so that 
$$
\st^{-1}_{k_i,k_j} A_{k_i,k_j} = \e_{k_i}  \e_{k_j} \frac{\st_{k_j,k_i}}{1 + \delta_{k_i}+ \delta_{k_j}}  
\frac{Q(\Delta_{k_i},\Delta_{k_j}) Q(\Delta_{k_j},\Delta_{k_i})}{Q(\delta_{k_i},\Delta_{k_j}) Q(\delta_{k_j},\Delta_{k_i})}
$$
using $\iota(\Delta_x) = \delta_{x}$, where $\delta_x \coloneq \e^{-1}_x - \id$ is the \Dfn{backward difference operator}\footnote{Note that the backward difference is often also defined as $\id - \e^{-1}_x$.}.
\end{remark}

\begin{example}
\label{weareinbusiness}
We provide a few examples and discuss why it was difficult to find the formula in Theorem~\ref{main} by a guess-and-prove approach.
For instance, we have 
\begin{multline*}
\MT_1(\mathbf{k}_4) =
k_3 k_2^2-k_4 k_2^2-k_2^2-k_3^2
   k_2+k_1 k_2-k_1 k_3 k_2+k_3
   k_2+k_1 k_4 k_2+k_3 k_4
   k_2-k_2+k_1
   k_3^2-k_3^2+k_3 \\ -k_1 k_4-k_1
   k_3 k_4+k_3 k_4-1. 
\end{multline*}
It is interesting to note that this agrees with $\prod_{1 \le i < j \le 4} \st^{-1}_{k_i,k_j}  \GT_1(\mathbf{k}_4)$, so the hidden formal power series does not play a role here. This is not anymore true for $h=1, n=5$, in which case we have the following:
\begin{multline*}
 \MT_1(\mathbf{k}_5) = k_4^2 k_2^2+k_3 k_2^2-k_3 k_4
   k_2^2+k_3 k_5 k_2^2-k_4 k_5
   k_2^2-k_5 k_2^2-k_3^2
   k_2-k_1 k_4^2 k_2-k_3 k_4^2
   k_2  -k_1 k_3 k_2 \\ +k_3^2 k_4
   k_2+k_1 k_3 k_4 k_2+k_4
   k_2-k_3^2 k_5 k_2+k_1 k_5
   k_2-k_1 k_3 k_5 k_2+k_3 k_5
   k_2+k_1 k_4 k_5 k_2+k_3 k_4
   k_5 k_2-k_5 k_2 -k_2 \\ +k_1
   k_3^2-k_3^2+k_1 k_4^2+k_1
   k_3 k_4^2-k_3 k_4^2+k_1+k_1
   k_3-k_1 k_3^2 k_4+k_3^2
   k_4-k_1 k_4-k_1 k_3
   k_4+k_4+k_1 k_3^2 k_5-k_3^2
   k_5+k_3 k_5 \\ -k_1 k_4 k_5 -k_1
   k_3 k_4 k_5+k_3 k_4
   k_5-k_5-1.
\end{multline*} 
For $h=2, n=4$, we have 
\begin{multline*}
\MT_2(\mathbf{k}_4) =
\frac{1}{12} (-k_3 k_2^4+k_4
   k_2^4+k_2^4-4 k_1 k_2^3+4 k_1
   k_3 k_2^3-4 k_3 k_2^3-4 k_1
   k_4 k_2^3 +4 k_4 k_2^3+4
   k_2^3  +3 k_1^2 k_2^2  -3 k_1
   k_3^2 k_2^2 \\ +9 k_3^2 k_2^2+3
   k_1 k_4^2 k_2^2-3 k_3 k_4^2
   k_2^2  -6 k_4^2 k_2^2+3 k_1
   k_2^2-3 k_1^2 k_3 k_2^2-8 k_3
   k_2^2+3 k_1^2 k_4 k_2^2+3
   k_3^2 k_4 k_2^2  +6 k_1 k_4
   k_2^2  -6 k_3 k_4 k_2^2 \\ -7 k_4
   k_2^2+5 k_2^2+k_3^4 k_2-4
   k_3^3 k_2-9 k_1^2 k_2+3 k_1^2
   k_3^2 k_2-6 k_1 k_3^2 k_2+8
   k_3^2 k_2-3 k_1^2 k_4^2 k_2+3
   k_3^2 k_4^2 k_2  +6 k_1 k_4^2
   k_2  +6 k_3 k_4^2 k_2 \\ +7 k_1
   k_2+6 k_1^2 k_3 k_2-k_1 k_3
   k_2-25 k_3 k_2-4 k_3^3 k_4
   k_2-12 k_1^2 k_4 k_2+13 k_1
   k_4 k_2+k_3 k_4 k_2+8 k_4
   k_2  +14 k_2  -k_1 k_3^4 \\ +k_3^4+4
   k_1 k_3^3-4 k_3^3-6 k_1^2
   k_3^2+7 k_1 k_3^2+5 k_3^2+3
   k_1^2 k_4^2-3 k_1 k_3^2
   k_4^2+3 k_3^2 k_4^2  -9 k_1
   k_4^2+3 k_1^2 k_3 k_4^2-12
   k_1 k_3 k_4^2  +9 k_3 k_4^2  \\ +8
   k_1 k_3 -14 k_3+4 k_1 k_3^3
   k_4-4 k_3^3 k_4+9 k_1^2 k_4-3
   k_1^2 k_3^2 k_4+6 k_1 k_3^2
   k_4-3 k_3^2 k_4-15 k_1 k_4+6
   k_1^2 k_3 k_4 \\ -13 k_1 k_3
   k_4+7 k_3 k_4).
\end{multline*}
This is again equal to $\prod_{1 \le i < j \le 4} \st^{-1}_{k_i,k_j}  \GT_2(\mathbf{k}_4)$. The same behavior is true for $\MT_2(\mathbf{k}_5)$, which has
already $315$ terms. However, we have 
$$\prod_{1 \le i < j \le n} \st^{-1}_{k_i,k_j}  \GT_2(\mathbf{k}_n) \not= \prod_{1 \le i < j \le n} \st^{-1}_{k_i,k_j} A_{k_i,k_j} \GT_2(\mathbf{k}_n)$$
for $n \ge 6$.

This indicates why it was difficult to find the hidden formal power series $A(x,y)$ empirically, since it does not make a difference in small examples; the computations 
are costly beyond the small examples especially for $h \ge 2$. 

At an earlier point in our studies, we did find empirically that $A(x,y) = 1 - \frac{3}{4} x^2 y^2$ works for $h=1$ and all $n$'s for which we could compute $\MT_1(\mathbf{k}_n)$
explicitly, but then we had to discover that this choice for $A(x,y)$ does not work for $h \ge 2$.

We then had to work with an approach that envisions a possible proof of the formula at a point where we did not have the formula and combine this with computer experiments. An important breakthrough happened when we realized that the hidden formal power series $A(x,y)$ we have found for those examples that we could handle with a computer (which were examples with $h=1,2$ and relatively small $n$) satisfy the two properties that are presented at the beginning of Section~\ref{neutralizing}. This was surprising for us as these properties seemed to be so restrictive that it was hard to believe that they admit any solution in general. However, in the end there are solutions. Luckily, these two properties also simplify the computation in the proof at an important point significantly. 
\end{example}

\bigskip

We argue that Theorem~\ref{main} implies the operator formula \eqref{operatorweyl} for the number of monotone triangles with bottom row $(k_1,\ldots,k_n)$.
We need to consider the case $h=n$ of Theorem~\ref{main}.  It suffices to show the following two assertions. For any positive integer $n$, we have 
$$
\pf  \left( 
\begin{array}{c|c}  
\TGT_{n}(\mathbf{k\!\uparrow}_n) & \RGT_{n,n}(\mathbf{k\!\uparrow}_n) \\ \hline 
                                                     & \T(0)_n 
\end{array} \right)
= \prod_{1 \le i < j \le n} \frac{k_j-k_i + j - i}{j-i},
$$
and, for any symmetric power series $P(X_1,\ldots,X_n)$, we have  
$$
P(\Delta_{k_1},\ldots,\Delta_{k_n}) \prod_{1 \le i < j \le n} \frac{k_j-k_i + j - i}{j-i} = P(0,\ldots,0) \prod_{1 \le i < j \le n} \frac{k_j-k_i + j - i}{j-i}.
$$
The latter is then applied to $P(X_1,\ldots,X_n) = \prod_{1 \le i < j \le n} A(X_i,X_j)$, noting also that $A(0,0)=1$, which follows from 
Lemma~\ref{characterize} \eqref{characterize:1} and $\iota(0)=0$.  

The second assertion is Lemma~1 in \cite{Fis07}. For the first assertion note that due to the block form of the triangular array underlying the Pfaffian
only matchings where all $j$ with $j \in \{n+1,\ldots,2n\}$ are matched to an $i$ with $i \in [n]$ contribute a non-zero summand, which eventually implies that the Pfaffian is equal to 
$$
(-1)^{\binom{n}{2}} \det  \left( \RGT_{n,n}(\mathbf{k\!\uparrow}_n) \right) = 
(-1)^{\binom{n}{2}} \det_{1 \le i, j \le n} \left( \binom{k_i}{n-j} \right).
$$
Using the Vandermonde determinant evaluation,
the latter evaluates to $\prod_{1 \le i < j \le n} \frac{k_j-k_i + j - i}{j-i}$. 

\medskip

The second assertion is related to the coinvariant algebra as pointed out next. It follows from a result of Steinberg \cite{steinberg} that the ideal $\mathcal{I}_n$ of all polynomials $R(X_1,\ldots,X_n)$
with 
$$
R(\Delta_{k_1},\ldots,\Delta_{k_n}) \prod_{1 \le i < j \le n} \frac{k_j-k_i + j - i}{j-i} = 0
$$
is generated by the symmetric polynomials without constant term, and, consequently, it follows from the fundamental theorem of symmetric polynomials that it is generated by the elementary symmetric polynomials $e_i(X_1,\ldots,X_n)$ with $i \ge 1$\footnote{We are using the fact that $\Delta_x$ has the same effect on the binomial basis $\binom{x}{n}$ of polynomials in $x$ as $\partial_x$ on the basis $\frac{x^n}{n!}.$}.
Recalling that the \Dfn{coinvariant algebra} of the symmetric group is the quotient 
$$
\mathcal{R}_n=\mathbb{Q}[X_1,\ldots,X_n] / \mathcal{I}_n,
$$
this implies that $\mathcal{R}_n$
is isomorphic to 
$$
\left\{ R(\Delta_{k_1},\ldots,\Delta_{k_n}) \prod_{1 \le i < j \le n} \frac{k_j-k_i + j - i}{j-i} \,\middle\vert \, R(X_1,\ldots,X_n) \in \mathbb{Q}[X_1,\ldots,X_n] \right\}
$$
as $\mathcal{S}_n$-modules. The coinvariant algebra is a well-understood object, see for instance \cite{stan79,bergeron}. 

An obvious question is whether there is a similar story when replacing $\GT_{n}(\mathbf{k}_n)$ 
by $\GT_{h}(\mathbf{k}_n)$. In Section~\ref{sec:conjecture}, we provide a conjecture on a generating set
of the ideal 
$\mathcal{I}_{h,n}$ of polynomials $R(X_1,\allowbreak \ldots,\allowbreak X_n)$ satisfying 
$
R(\Delta_{k_1},\ldots,\Delta_{k_n}) \GT_{h}(\mathbf{k}_n)=0.
$

\bigskip

{\bf Applications and future directions.} The operator formula for monotone triangles \eqref{operatorweyl} has been the key ingredient for numerous results on alternating sign matrices and related objects. To name a few, it has led to another proof of the alternating sign matrix theorem \cite{Fis07} and it was also fundamental for the first (complicated) bijective proof of this theorem \cite{PartI,cube}. It has been used to establish the connection between alternating sign trapezoids and certain holey lozenge tilings that generalize descending plane partitions \cite{fourfold} as well as to establish the connection between alternating sign triangles and totally symmetric self-complementary plane partitions \cite{gangl}. A multivariate generalization of the formula \cite{nASMDPP} that includes Schur polynomials and Hall--Littlewood polynomials as special cases was used to give an $(n+3)$-parameter refinement for the equinumerosity of alternating sign matrices and descending plane partitions. Thus it is expected that the more general operator formula for monotone trapezoids will also have numerous applications, some of which generalize the results mentioned above. 

Of particular interest is the fact that our new formula involves alternating sign matrix objects with free boundary, which is a new feature for a formula in this area. That is likely to lead to new types of results. One concrete direction we have in mind, which was also our motivation to work out the formula, is related to Littlewood-type identities that involve alternating sign matrix objects.

\bigskip

{\bf Outline of the paper.} 
Section~\ref{GTproof} is devoted to the proof of Theorem~\ref{GT}. Although the result is a straightforward application of a result of \cite{ste90}, we present an alternative proof as a warm-up for the proof of Theorem~\ref{main}. 
The proof of the main result (Theorem~\ref{main}) is very involved, and thus we provide a sketch highlighting the basic steps of the proof in Section~\ref{sketch}. The details of the proof are then presented in Sections~\ref{neutralizing}--\ref{sec:second}. 
More concretely, in Section~\ref{neutralizing}, we provide at the beginning  the crucial properties of the hidden formal power series and then derive a nice characterization of the hidden formal power series, which implies in particular their existence.
In Section~\ref{sec:first}, we present the first part of the proof of Theorem~\ref{main}. This is the part that follows the structure of the proof of the simpler Theorem~\ref{GT} and reduces the result to two lemmas (Lemma~\ref{one} and Lemma~\ref{zero}). There is no component  in the proof of Theorem~\ref{GT} that is equivalent to these lemmas. The two lemmas are then proved in Section~\ref{sec:second}. This is by far the most complicated part of the proof. 
In Section~\ref{sec:conjecture}, we present the conjecture on the annihilator ideal of $\GT_h(\mathbf{k}_n)$ that has relations to the coinvariant algebra.

\section{Proof of Theorem~\texorpdfstring{\ref{GT}}{\ref*{GT}}}
\label{GTproof} 

We first consider the case $h+n$ is even. The proof is by induction with respect to $h$ and we first check the base case $h=0$ of the induction. 
Concerning the left-hand side, note that $\GT_0(\mathbf{k}_n)=1$. As for the right-hand side, 
the claimed Pfaffian in Theorem~\ref{GT} has all entries equal to $1=\gt_0(k_i+i,k_j+j)$. It suffices to show the following.

\begin{lemma}
We have 
$
\pf_{1 \le i < j \le 2m} \left( 1 \right) = 1.
$
\end{lemma} 

\begin{proof} 
It suffices to find a sign-reversing involution on the matchings $\pi$ of $[2m]$ different from $(1,2),(3,4),\ldots,(2m-1,2m)$. 
Let $2i-1$ be the minimal odd integer that is not matched with $2i$. Assume that $2i-1$ is matched with $p$ and $2i$ is matched with $q$. Construct $\pi'$ from $\pi$ by matching $2i-1$ with $q$ and matching $2i$ with $p$, and leaving $\pi$ unchanged otherwise. This is clearly a sign-reversing involution.
\end{proof} 

In order to perform the induction step, we use a recursion for $\GT_h(\mathbf{k}_n)$: if we delete the bottom row of 
an $(h,n)$-GT trapezoid with $h>0$ and bottom row $\mathbf{k}_n$, we obtain an $(h-1,n-1)$-GT trapezoid with, say, bottom row
$\mathbf{l}_{n-1}$ such that $\mathbf{l}_{n-1}$ \Dfn{interlaces} $\mathbf{k}_n$, that is, 
\begin{equation*} 
\label{interlaces}
k_1 \le l_1 \le k_2 \le l_2 \le k_3 \le \ldots \le l_{n-1} \le k_n,
\end{equation*} 
which we denote by $\mathbf{l}_{n-1} \preceq  \mathbf{k}_{n}$.
This implies the following recursion:
\begin{equation}
\label{GTrec1}  
\GT_h(\mathbf{k}_n) = \sum_{\mathbf{l}_{n-1} \preceq  \mathbf{k}_{n}}  \GT_{h-1}(\mathbf{l}_{n-1}).
\end{equation} 
In the special case $n=1$, the recursion acts like the identity.

Using the transformations $\mathbf{k}_n \mapsto \mathbf{k}_n - \mathbf{0\!\uparrow}_n$ and $\mathbf{l}_{n-1} \mapsto \mathbf{l}_{n-1} - \mathbf{0\!\uparrow}_{n-1}$, where 
$\mathbf{0\!\uparrow}_n \coloneq (1,2,\ldots,n)$, 
the former recursion is equivalent to the following
\begin{equation}
\label{GTrec}
\GT_h(\mathbf{k}_n-\mathbf{0\!\uparrow}_n) = \sum_{\mathbf{l}_{n-1} \prec  \mathbf{k}_{n}}
\GT_{h-1}(\mathbf{l}_{n-1}-\mathbf{0\!\uparrow}_{n-1}),
\end{equation} 
where $\mathbf{l}_{n-1} \prec  \mathbf{k}_{n}$ stands for 
$$
k_1 \le l_1 < k_2 \le l_2 < k_3 \le \ldots \le l_{n-1} < k_n.
$$
It follows that it suffices to prove the following lemma.

\begin{lemma} 
\label{gt-step}
Let $h$ and $n$ be positive integers with $n > h$. Then
$$ 
\sum_{\mathbf{l}_{n-1} \prec \mathbf{k}_n}
  \pf \left( 
\begin{array}{c|c}  \TGT_{h-1}(\mathbf{l}_{n-1}) &  \RGT_{h-1,h-1}(\mathbf{l}_{n-1})  \\ \hline 
& \T(0)_{h-1} \end{array} \right)  = 
\pf \left( 
\begin{array}{c|c}  \TGT_{h}(\mathbf{k}_{n}) & \RGT_{h,h}(\mathbf{k}_{n}) \\ \hline 
&  \T(0)_{h} \end{array} \right). 
$$ 
\end{lemma} 

We use the following two properties of $\gt_h(x,y)$ in the proof: First, for $h>0$, we have 
\begin{equation}
\label{diffgt} 
\Delta_x \Delta_y \gt_h(x,y) = \gt_{h-1}(x,y),  
\end{equation} 
which follows by applying 
\begin{equation} 
\label{diffbinom} 
\Delta_z \binom{z}{m} = \binom{z}{m-1} 
\end{equation} 
several times. Second, if we symmetrize $\gt_h(x,y)$, then we have  
\begin{multline} 
\label{sym} 
\frac{\gt_h(x,y) + \gt_h(y,x)}{2} \\ 
=\sum_{p=y}^{x-1} \left( \binom{x-p+h-1}{h} \binom{y-p+h-2}{h} - \binom{x-p+h-2}{h} \binom{y-p+h-1}{h}  \right) \\ = 
(-1)^h \frac{1}{(2h)!} (x-y-h+1)_{2h-1} (x-y), 
\end{multline} 
using Pochhammer notation 
$$
(a)_n \coloneq a(a+1) \dots (a+n-1)
$$
and 
where the last step in \eqref{sym} follows essentially from the Chu--Vandermonde summation. We set 
\begin{equation}
\label{sgtdef}  
\sgt_h(x,y) \coloneq (-1)^h \frac{1}{(2h)!} (x-y-h+1)_{2h-1} (x-y),
\end{equation} 
where $\sgt$ indicates that we have extracted the symmetric part of $\gt_h(x,y)$.

\begin{proof}[Proof of Lemma~\ref{gt-step}]
To simplify the notation, we denote the summand by  
\begin{equation} 
\label{GTsummand} 
s_{h}(\mathbf{l}_{n-1}) \coloneq   \pf \left( 
\begin{array}{c|c}  \TGT_{h-1}(\mathbf{l}_{n-1}) &  \RGT_{h-1,h-1}(\mathbf{l}_{n-1})  \\ \hline 
& \T(0)_{h-1} \end{array} \right)
\end{equation} 
and set 
\begin{equation} 
\label{GTintegral} 
S_{h}(\mathbf{l}_{n-1}) \coloneq   \pf \left( 
\begin{array}{c|c}  \TGT_{h}(\mathbf{l}_{n-1}) &  \RGT_{h,h-1}(\mathbf{l}_{n-1})  \\ \hline 
& \T(0)_{h-1} \end{array} \right).
\end{equation}
By \eqref{diffgt}, \eqref{diffbinom} and the definition of the Pfaffian, we have 
$
\Delta_{\mathbf{l}_{n-1}} S_{h}(\mathbf{l}_{n-1}) = s_{h}(\mathbf{l}_{n-1}).
$
By telescoping, that is, we use 
\begin{equation}
\label{telescoping} 
\sum_{l_i=k_i}^{k_{i+1}-1} \Delta_{l_i} F(l_i) = F(k_{i+1}) - F(k_i),
\end{equation}
we can conclude that 
 \begin{equation}
 \label{gtsum}
 \sum_{\mathbf{l}_{n-1} \prec \mathbf{k}_n} \Delta_{\mathbf{l}_{n-1}} S_{h}(\mathbf{l}_{n-1})
\end{equation} 
can be written as $2^{n-1}$ expressions where for each $l_i$, $i=1,2,\ldots,n-1$, we choose either the lower or upper bound 
in the summation $\sum_{l_i=k_i}^{k_{i+1}-1}$ and for every choice of the lower bound we have a contribution of $-1$ to the sign. 
After showing that 
\begin{equation}
\label{zero1}  
S_{h}(\mathbf{l}_{n-1})=0,
\end{equation} 
if $l_{p}=l_{p+1}$ for some $p<n-1$, 
we can conclude that \eqref{gtsum} is equal to 
\begin{equation}
\label{reduced} 
\sum_{r=1}^n (-1)^{r-1} S_{h}(\mathbf{k}_{n}^r), 
\end{equation} 
where
\begin{equation} 
\label{ignore}
\mathbf{k}^r_{n} \coloneq (k_1,\ldots,k_{r-1},k_{r+1},\ldots,k_n).
\end{equation} 

In order to show \eqref{zero1}, we use a formula that will be employed later on as well. First note that, by definition, the Pfaffian 
of a triangular array $(A_{i,j})_{1 \le i < j \le 2m}$ is linear in the vector that is obtained by concatenating the $p$-th column and the $p$-th row 
of the triangular array, that is, 
$$
(A_{1,p},A_{2,p},\ldots,A_{p-1,p},A_{p,p+1},\ldots,A_{p,2m}).
$$ 
To show \eqref{zero1}, we decompose 
$$
(A_{1,p},A_{2,p},\ldots,A_{p-1,p},A_{p,p+1},A_{p,p+2},\ldots,A_{p,n+h-2})
$$
of the triangular array underlying the Pfaffian $S_{h}(\mathbf{l}_{n-1})$ as follows: It holds $A_{p,p+1}=\gt_h(l_p,l_{p+1})$ and let the first vector be
$$
(A_{1,p},A_{2,p},\ldots,A_{p-1,p},\frac{\gt_h(l_p,l_{p+1})-\gt_h(l_{p+1},l_p)}{2},A_{p,p+2},\ldots,A_{p,n+h-2})
$$
and the second vector be 
$$
(0,0,\ldots,0,\frac{\gt_h(l_p,l_{p+1})+\gt_h(l_{p+1},l_p)}{2},0,\ldots,0)=(0,0,\ldots,0,\sgt_h(l_p,l_{p+1}),0,\ldots,0).
$$
We let 
$S^{p,\asym}_{h}(\mathbf{l}_{n-1})$ and $S^{p,\sym}_{h}(\mathbf{l}_{n-1})$ denote the two Pfaffians obtained this way, and we can conclude 
\begin{equation}\label{decompose}
S_{h}(\mathbf{l}_{n-1}) = S^{p,\asym}_{h}(\mathbf{l}_{n-1})+S^{p,\sym}_{h}(\mathbf{l}_{n-1}).
\end{equation}

Now we argue that $S^{p,\asym}_{h}(\mathbf{l}_{n-1})$ is antisymmetric in $l_{p}$ and $l_{p+1}$. Indeed, when expanding $S^{p,\asym}_{h}(\mathbf{l}_{n-1})$ as a sum over all matchings, and when considering a summand that corresponds to a matching where $p$ is not matched to $p+1$, we can pair this summand off with the one obtained by interchanging the partners of $p$ and $p+1$ in the matching to obtain an antisymmetric expression; 
if, on the other hand, we consider a summand where $p$ is matched to $p+1$, then $\frac{1}{2}(\gt_h(l_p,l_{p+1})-\gt_h(l_{p+1},l_p))$ appears as a factor and this factor is the only contribution of $l_p,l_{p+1}$, and thus this summand is antisymmetric.

The antisymmetry of $S^{p,\asym}_{h}(\mathbf{l}_{n-1})$ implies its vanishing whenever $l_p=l_{p+1}$. On the other hand, 
when expanding $S^{p,\sym}_{h}(\mathbf{l}_{n-1})$ as a sum over matchings, we only have a non-zero contribution if $p$ is matched to $p+1$, and this pair contributes the factor $\sgt_h(l_p,l_{p+1})$, which vanishes whenever $l_p=l_{p+1}$ by \eqref{sym}.

The fact that \eqref{reduced} is equal to the right-hand side in the statement follows from the Laplace expansion formula for Pfaffians, that is for 
$p \in [2m]$ we have 
\begin{equation}
\label{expansion}  
\pf_{1 \le i < j \le 2m} \left( A_{i,j} \right) = 
\sum_{q=1}^{p-1} (-1)^{p+q+1} A_{q,p} \pf_{1 \le i < j \le 2m \atop i,j \not= p,q} \left( A_{i,j} \right) \\
+ \sum_{q=p+1}^{2n} (-1)^{p+q+1} A_{p,q} \pf_{1 \le i < j \le 2m \atop i,j \not= p,q} \left( A_{i,j} \right),
\end{equation} 
when applying the formula to the last column of the triangular array on the right-hand side of the statement.
\end{proof} 

Finally we consider the case $n+h$ is odd of Theorem~\ref{GT}. Again we use induction with respect to $h$ and for $h=0$ there is nothing to prove since 
$$
\GT_0(\mathbf{k}_n) = 1 = \GT_0(\mathbf{k}_{n+1}) = \Delta^0_{k_{n+1}} \GT_0(\mathbf{k}_{n+1}).
$$
For the induction step note that we can also use the recursion as it holds for this parity too.
$$
\GT_h(\mathbf{k}_n) = \sum_{\mathbf{l}_{n-1} \preceq \mathbf{k}_{n}} \GT_{h-1}(\mathbf{l}_{n-1})
$$
By the induction hypothesis, this is further equal to 
$$
\sum_{\mathbf{l}_{n-1} \preceq \mathbf{k}_{n}} \Delta^{h-1}_{l_n} \GT_{h-1}(\mathbf{l}_{n}).
$$
This expression is independent of $l_n$, and thus we may in particular replace $l_n$ by $k_{n+1}$ and obtain 
$$
\sum_{\mathbf{l}_{n-1} \preceq \mathbf{k}_{n}}  \Delta^{h-1}_{k_{n+1}} \GT_{h-1}(l_1,\ldots,l_{n-1},k_{n+1}).
$$
As 
\begin{equation}
\label{eat1} 
\GT_{h-1}(l_1,\ldots,l_{n-1},k_{n+1}) = \e^{-1}_{k_{n+1}}\Delta_{k_{n+1}} \sum_{l_n=k_n}^{k_{n+1}}  \GT_{h-1}(\mathbf{l}_n),
\end{equation} 
we see that this is further equal to 
$$
\e^{-1}_{k_{n+1}} \Delta^{h}_{k_{n+1}} \sum_{\mathbf{l}_{n} \preceq \mathbf{k}_{n+1}} 
\GT_{h-1}(\mathbf{l}_n) = \e^{-1}_{k_{n+1}} \Delta^{h}_{k_{n+1}} \GT_h(\mathbf{k}_{n+1})
$$
Since this expression is equal to $\GT_{h}(\mathbf{k}_n)$, it is in particular
independent of $k_{n+1}$, and thus we can also apply $\e_{k_{n+1}}$ without causing any change and we obtain the claimed expression.

In \eqref{eat1}, we applied a special case of Lemma~\ref{eat}, stated below. For this purpose, we extend the notation $\mathbf{k}_n$ as follows: for two integers $i,j$, we set 
$$
\mathbf{k}_{i,j} \coloneq (k_i,k_{i+1},\ldots,k_j).
$$
If $j<i$, then $\mathbf{k}_{i,j}$ is defined to be the empty vector.

\begin{lemma} 
\label{eat}
Let $n$ be a positive integer, $a(\mathbf{l}_{n-1})$ a function, and $i$ and $j$ non-negative integers with $0 \le i < j$ and $j \le n+1$. Then 
$$
(-1)^{i} \Delta_{\mathbf{k}_{1,i}} \Delta_{\mathbf{k}_{j,n}} \left[ \sum_{\mathbf{l}_{n-1} \prec \mathbf{k}_n} a(\mathbf{l}_{n-1}) \right] =
\sum_{\mathbf{l}_{i+1,j-2} \prec \mathbf{k}_{i+1,j-1}} a(\mathbf{k}_{1,i},\mathbf{l}_{i+1,j-2},\mathbf{k}_{j,n}).
$$
\end{lemma}  

\begin{proof} 
The proof is by a simple inductive argument with respect to $i+(n+1-j)$, using 
$$
-\Delta_{k_i} \left[ \sum_{l_i=k_i}^{k_{i+1}-1} a(l_i)  \right]= a(k_i) \quad \text{and} \quad 
\Delta_{k_i} \left[ \sum_{l_{i-1}=k_{i-1}}^{k_i-1} a(l_{i-1}) \right]  = a(k_i).
$$
\end{proof}

\section{Sketch of the proof of Theorem~\ref{main}}
\label{sketch}

The purpose of this section is to provide a sketch of the proof of Theorem~\ref{main}. The details are then worked out in Sections~\ref{neutralizing}--\ref{sec:second}.

\medskip

{\bf Setting up a recursion.} The proof of the formula for $\MT_h(\mathbf{k}_n)$ is by induction with respect to 
$h$, which is one of the many parallels with the proof of Theorem~\ref{GT}. The induction is based on a recursion, and, in the case of Gelfand--Tsetlin trapezoids, this recursion is provided in \eqref{GTrec}. For monotone trapezoids it is stated in \eqref{MTrec} and its derivation follows the same principle although it is more involved. Then, in order to prove Theorem~\ref{main}, it suffices to show that the formula in Theorem~\ref{main} satisfies the recursion.

\medskip

{\bf Telescoping.} The recursion is expressed in terms of sums over integer intervals as it is also the case for Gelfand--Tsetlin trapezoids. Our starting point is the right hand side of \eqref{rectoshow1}.
We use telescoping to compute such sums, see \eqref{telescoping} where this is explained for the case of Gelfand--Tsetlin trapezoid and \eqref{inter} where this is applied to the case of monotone trapezoids. Our summands are polynomials in the summation parameters. To sum such polynomials, it is convenient to express them in terms of a binomial coefficient basis because then it is easy to determine the ``discrete integral''. Concretely, we repeatedly apply the following identity:
$$
\sum_{x=a}^b \binom{x}{n} = \sum_{x=a}^b \left( \binom{x+1}{n+1} - \binom{x}{n+1} \right) = 
\binom{b+1}{n+1} - \binom{a}{n+1}.
$$

\medskip

{\bf Decomposition into the symmetric and the antisymmetric part.} After these first two steps, the proof of Theorem~\ref{main} deviates from the proof of Theorem~\ref{GT} for the reason given next. 

In both cases, the telescoping step leads in principle to $2^{n-1}$ terms, see \eqref{gtsum} and \eqref{inter}.  However, in the case of Gelfand--Tsetlin patterns, one can almost immediately observe that all but $n$ of these terms vanish. This is accomplished by showing that the discrete integral $S_h(\mathbf{l}_{n-1})$ vanishes if $l_{p}=l_{p+1}$ for some $p < n-1$. For this purpose, we consider $S_h(\mathbf{l}_{n-1})$ as polynomial in $l_p,l_{p+1}$ and decompose it into its symmetric and its antisymmetric part using 
$$
s(l_p,l_{p+1}) = \frac{s(l_p,l_{p+1}) + s(l_{p+1},l_p)}{2} + \frac{s(l_p,l_{p+1}) - s(l_{p+1},l_p)}{2},
$$
where the first term on the right hand side is symmetric in $l_p,l_{p+1}$ and the second term is antisymmetric.
As for the antisymmetric part, it is obvious that it vanishes if $l_{p}=l_{p+1}$. The miracle happens for the symmetric 
part: It is equal to $\sgt_h(l_p,l_{p+1})$ (see \eqref{sgtdef} for the definition) up to a factor independent of $l_p,l_{p+1}$ and 
$\sgt_h(l_p,l_{p+1})$ obviously vanishes for $l_{p}=l_{p+1}$.

The summand as well as the discrete integral differs in the case of monotone trapezoids from those for Gelfand--Tsetlin  trapezoids by the application of the operator $\prod_{1 \le i < j \le n-1} U_{l_i,l_{i+1}}$, compare \eqref{MTsummand} and \eqref{MTintegral} with \eqref{GTsummand} and \eqref{GTintegral}, respectively. In the case of monotone trapezoids, we also decompose the discrete integral $T_h(\mathbf{l}_{n-1})$ defined in \eqref{MTintegral} into its symmetric and its antisymmetric part with respect to $l_p,l_{p+1}$. This is based on the decomposition for $S_h(\mathbf{l}_{n-1})$, and, for the antisymmetric part, Lemma~\ref{asym} \eqref{asym:1} shows that the antisymmetry is not destroyed by the operator and we can remove this part easily in the case of monotone trapezoids as well. 

As for the symmetric part, Lemma~\ref{asym} \eqref{asym:2} shows that applying the operator $\prod_{1 \le i < j \le n-1} U_{l_i,l_{i+1}}$ to the symmetric part results in a linear combination of $\sgt_{g}(l_p,l_{p+1})$'s for $0 \le g \le h$. Clearly, we have 
$\sgt_{g}(l_{p},l_{p+1})=0$ if $l_{p} = l_{p+1}$ for $0 < g \le h$, however, for $g=0$ this is not anymore the case, and this causes the problem.

The removal of the antisymmetric part is performed in Lemma~\ref{pq}, and the result is a simplified formula for each of the $2^{n-1}$ expressions which are obtained after telescoping. These expressions are parameterized by the interlacing sequences $\mathbf{p}_m$ and $\mathbf{q}_{m-1}$. 

\medskip

{\bf The hidden formal power series.} 
At the end of the proof of Lemma~\ref{pq}, the role of the hidden formal power series becomes apparent: 
It reduces the second line of \eqref{penultimate}, that is, 
$$  
\prod_{1 \le i < j \le m-1} U_{k^{L}_{q_i},k^{L}_{q_j}} 
U_{k^{L}_{q_i},k^{R}_{q_j}}
U_{k^{R}_{q_i},k^{L}_{q_j}}
U_{k^{R}_{q_i},k^{R}_{q_j}}
\prod_{r=1}^{m-1}  A_{k^{L}_{q_r},k^{R}_{q_r}} \sgt_{h}(k^{L}_{q_r},k^{R}_{q_r})   
$$
to 
$$
\prod_{r=1}^{m-1} \sgt_{h}(k^{L}_{q_r},k^{R}_{q_r}),
$$
compare with the expression in the statement of Lemma~\ref{pq}. Small empirical examples (see also the discussion in Example~\ref{weareinbusiness}) show that it makes sense to require 
\begin{equation} 
\label{symop1}
A_{k^{L}_{q_r},k^{R}_{q_r}} \sgt_{h}(k^{L}_{q_r},k^{R}_{q_r}) = \sgt_{h}(k^{L}_{q_r},k^{R}_{q_r})
\end{equation} 
as well as 
\begin{equation} 
\label{symop21} 
U_{k^{L}_{q_i},k^{L}_{q_j}} 
U_{k^{L}_{q_i},k^{R}_{q_j}}
U_{k^{R}_{q_i},k^{L}_{q_j}}
U_{k^{R}_{q_i},k^{R}_{q_j}} \sgt_{h}(k^{L}_{q_i},k^{R}_{q_i}) \sgt_{h}(k^{L}_{q_j},k^{R}_{q_j}) = 
\sgt_{h}(k^{L}_{q_i},k^{R}_{q_i}) \sgt_{h}(k^{L}_{q_j},k^{R}_{q_j}), 
\end{equation} 
as we do in \eqref{1} and \eqref{2} in Definition~\ref{def:annihilating} at the beginning of Section~\ref{neutralizing}.

\medskip

{\bf Applying symmetric operators to $\sgt_h(x,y)$.} Several steps of the proof make it necessary to  understand the action of symmetric operators in $\Delta_x,\Delta_y$ to $\sgt_h(x,y)$. Examples are given in \eqref{symop1} and \eqref{symop21}, 
where the operators are 
$$
A_{x,y}, \quad U_{x,z} U_{y,z}, \quad U_{z,x} U_{z,y}.
$$
For the operator $W^{-1}$, which is one factor of $U$ (see \eqref{AU}), this is accomplished in Lemma~\ref{W12}, and then applied in Corollary~\ref{corsecond} to deal further with \eqref{symop1} and \eqref{symop21}. To obtain a useful characterization of the formal power series $A(x,y)$ that satisfy the two identities, we need also Lemma~\ref{characterize} and finally Proposition~\ref{hiddencharacterize}.

A further symmetric operator in $\Delta_x,\Delta_y$ the action of which on $\sgt_h(x,y)$ needs to be understood is given in 
\eqref{remainingoperator}. 

The strategy to deal with these operators is as follows: Since they are symmetric, they can be expressed in terms of $\Delta_x \Delta_y$ and $-\Delta_x - \Delta_y$. Now the latter two operators have the same affect on $\sgt_h(x,y)$, see \eqref{Bsym1} and \eqref{Bsym2}: They decrease the parameter $h$ by $1$. Thus it would be useful to replace both $\Delta_x \Delta_y$ and $-\Delta_x - \Delta_y$ in the expansion of the operator by $\e^{-1}_{h}$. For instance, this can be accomplished as follows: by setting $X \coloneq \Delta_x$, $Y \coloneq \Delta_y$, and $H \coloneq X Y = - X - Y$, we can use this system of equations to express $X$ and $Y$ in terms of $H$. We use this to eliminate $X,Y$ from the expansion and eventually replace $H$ by $\e_h^{-1}$. Also note that $X+Y=-X-Y$ implies $Y= - \frac{X}{1+X} = \iota(X)$, which is the reason why the involution~$\iota$ appears in several places. All of this is applied in Section~\ref{neutralizing} and after \eqref{remainingoperator}. 

In the application after \eqref{remainingoperator}, we also need to take the identification $k^L_{q_r}=k^R_{q_r}$ into account. Since the operator is applied to 
$\sgt_h$,  this means that we are only interested in the coefficient of $H^h$ in the operator as $\sgt_{g}(x,x) \not= 0$ only for $g=0$. This non-vanishing of $\sgt_{0}(x,x)$ was already identified above
as the reason why the symmetric part of $T_h(\mathbf{l}_{n-1})$ with respect to $l_p,l_{p+1}$ causes problems if compared to the case of Gelfand--Tsetlin trapezoids.

\medskip 

{\bf From $2^{n-1}$ terms to $n$ terms.} The result of all these efforts is that we can eventually show that, even in the case of monotone trapezoids, only $n$ terms remain after the telescoping step. However, these $2^{n-1}-n$ terms that vanish individually for Gelfand--Tsetlin trapezoids do not necessarily vanish individually for monotone trapezoids. They must be grouped to show that the sum of terms within each group vanishes.
For the grouping, it is necessary to rearrange the operators, which is accomplished in Lemma~\ref{regrouping}. 

The grouping is done as follows: we fix a sequence $\mathbf{q}_{m-1}$, and all $\mathbf{p}_{m}$ that are interlaced by $\mathbf{q}_{m-1}$ in a strict sense (see the statement of Lemma~\ref{pq}) belong to the same group. Using the equivalent of the Laplace expansion \eqref{expansion} for Pfaffians, we then sum over all $\mathbf{p}_{m}$, see \eqref{laplace}. This is possible because at this stage the operator no longer depends on the $p_i$ (in a certain sense).

\medskip

{\bf Dealing with the $n$ non-vanishing terms.} 
One more non-trivial step is required to show that the $n$ terms\footnote{At this point they have already been unified to one term by \eqref{expansion}.} remaining after telescoping yield the formula in Theorem~\ref{main}. The $n$ terms correspond to $m=1$ in Lemma~\ref{pq}, in which case we may choose $p_1 \in [n]$, resulting in $n$ terms. 

The non-trivial step is performed in Lemma~\ref{one}, the proof of which is the purpose of the first half of Section~\ref{sec:second}.  The key to this lemma is that the recursion underlying Gelfand--Tsetlin trapezoids  \eqref{GTrec} commutes with the application of a certain operator expressed via elementary symmetric polynomials, see Lemma~\ref{urbanrenewal} \eqref{urbanrenewal:1} for the details. This commutation rule was obtained in an earlier paper, where it is shown to be an extension of \Dfn{Urban renewal}. It implies in particular that 
\begin{equation}
\label{annidiff} 
S(\Delta_{\mathbf{k}_n}) \overline \GT_h(\mathbf{k}_n) = S(\iota(\Delta_{\mathbf{k}_n}))  \overline \GT_h(\mathbf{k}_n)
\end{equation} 
for any symmetric polynomial $S(X_1,\ldots,X_n)$, see Lemma~\ref{fund}. Here an important observation is also that 
$\iota(\Delta_{x}) = \delta_{x}$. 

\medskip

{\bf Dealing with the $2^{n-1}-n$ terms that do vanish.} In order to show that if $m \ge 2$ then, for a fixed sequence $\mathbf{q}_{m-1}$, the sum of all terms where $\mathbf{p}_{m}$ are interlaced by $\mathbf{q}_{m-1}$ in a strict sense vanishes, we first argue that it suffices to consider the case $m=2$. This is done at the end of Section~\ref{sec:first} and reduces the result to proving Lemma~\ref{zero}. A nice reformulation of  Lemma~\ref{zero} is provided in Section~\ref{towards}: There we argue that it suffices to consider 
\begin{equation}
\label{z}
\prod_{i> q} \frac{1+ p(\delta_{k_i}) z}{1+ p(\Delta_{k_i}) z} \overline \GT_h(\mathbf{k}_{n}) \eqqcolon Q(z,\mathbf{k}_n),
\end{equation} 
where $q$ is a half-integer with $1 \le q \le n$ and $p(x)$ is a polynomial without constant term. 
We show in Lemma~\ref{signreversing} that $Q(z,\mathbf{k}_n)$
is a polynomial in $z$ of degree no greater than $h$, and that the degree of the alternating sum $\sum_{i>q} (-1)^i Q(z,\mathbf{k}^i_n)$ in $z$ is less than $h$\footnote{It can be seen that \eqref{annidiff} almost immediately implies that \eqref{z} is independent of $z$ if $q \le 1$.}, which is the key result. For this purpose, a certain additive decomposition of the polynomial $\overline \GT_h(\mathbf{k}_{n})$ is provided that cuts the Gelfand--Tsetlin trapezoid ``along'' the half integer $q$ into two pieces. This makes it possible to deal with the operator $\prod_{i> q} \frac{1+ p(\delta_{k_i}) z}{1+ p(\Delta_{k_i}) z}$.

\section{The hidden formal power series} 
\label{neutralizing}

We will work with the following shifted version of the operator 
$\st_{x,y}$:
\begin{equation}
\label{W} 
W_{x,y} \coloneq \e_x \st_{x,y} = \id + \e_x \e_y -  \e_y = \id + \Delta_x + \Delta_x \Delta_y.
\end{equation}

\begin{definition}\label{def:annihilating}
For a formal power series $A(x,y)$ in $x$ and $y$, we let 
\begin{equation}
	\label{AU}  
	A_{x,y} \coloneq A(\Delta_x,\Delta_y) \quad \text{and} \quad U_{x,y} = W^{-1}_{x,y} A_{x,y}.
\end{equation} 
We say that $A(x,y)$ is \Dfn{annihilating} if $A(x,y)$ is symmetric and the following two conditions are satisfied:
\begin{enumerate}
	\item\label{1} For all $h \ge 0$, we have $A_{x,y} \sgt_h(x,y) = \sgt_h(x,y)$.
	\item\label{2} For all $h_1,h_2 \ge 0$, we have 
	$$
	U_{x_1,x_2} U_{y_1,x_2} U_{x_1,y_2} U_{y_1,y_2} \sgt_{h_1}(x_1,y_1) \sgt_{h_2}(x_2,y_2)
	= \sgt_{h_1}(x_1,y_1) \sgt_{h_2}(x_2,y_2).
	$$
\end{enumerate} 
\end{definition}
In the second property, $A_{x,y}$ annihilates $W_{x,y}$, which explains where the term comes from. The purpose of this section is to show 
that the hidden formal power series as described in Section~\ref{sec:main} are all annihilating in the above sense and that, in a sense, they exhaust this class of formal power series. 

We start with a lemma on some useful computational facts on the operator $W_{x,y}$. We set $\sgt_{h}(x,y)=0$ if $h<0$.

\begin{lemma}
\label{W12}
The following is true:
\begin{enumerate} 
\item\label{W12:1} For any  integer $h$, we have 
$$W^{-1}_{x,z} W^{-1}_{y,z} \sgt_h(x,y) =  (1 + \e_{h}^{-1} \Delta_z (1 + \Delta_z))^{-1}  \sgt_{h}(x,y).$$
\item\label{W12:2} For any integer $h$, we have 
$$
W^{-1}_{z,x} W^{-1}_{z,y} \sgt_h(x,y) =  ((\Delta_z+1)^2 - \e_{h}^{-1} \Delta_z)^{-1}  \sgt_{h}(x,y).
$$
\item\label{W12:3} For any integers $h_1,h_2$, we have 
\begin{multline*} 
W^{-1}_{x_1,x_2} W^{-1}_{y_1,x_2} W^{-1}_{x_1,y_2} W^{-1}_{y_1,y_2} \sgt_{h_1}(x_1,y_1) \sgt_{h_2}(x_2,y_2) \\
=  (1 - 3 \e^{-1}_{h_1} \e^{-1}_{h_2} + \e^{-1}_{h_1} \e^{-2}_{h_2} + \e^{-2}_{h_1} \e^{-1}_{h_2})^{-1} \sgt_{h_1}(x_1,y_1) \sgt_{h_2}(x_2,y_2).
\end{multline*}
\end{enumerate}
\end{lemma}

Note that the right-hand sides of the identities need to be expanded using geometric series expansion and all sums are effectively finite.

The next two identities are crucial for the proof of the lemma
and follow from simple computations:
\begin{align}
\label{Bsym1} 
\Delta_x \Delta_y \sgt_h(x,y) &= \e_{h}^{-1} \sgt_{h}(x,y), \\
\label{Bsym2}
(\Delta_x + \Delta_y) \sgt_h(x,y) & = - \e_{h}^{-1} \sgt_{h}(x,y).
\end{align} 

\begin{proof}[Proof of Lemma~\ref{W12}]
As for the first assertion observe that the operator $W^{-1}_{x,z} W^{-1}_{y,z}$ is symmetric in $\Delta_x,\Delta_y$ and thus expressible in terms of the basis 
$$
(\Delta_x \Delta_y)^u (-\Delta_x-\Delta_y)^v,
$$
where $u,v \ge 0$. In this expansion, we want to replace this basis element by $(\e_h^{-1})^{u + v}$ as, by \eqref{Bsym1} and \eqref{Bsym2}, the basis element 
decreases $h$ by $u+v$ when applying it to $\sgt_h(x,y)$. Using the following system of equations for 
$\Delta_x,\Delta_y$ 
\begin{equation} 
\label{solutions} 
\Delta_x \Delta_y = \e_{h}^{-1} \qquad \text{and} \qquad 
-\Delta_x-\Delta_y = \e_{h}^{-1},
\end{equation} 
we observe that 
\begin{multline*} 
W^{-1}_{x,z} W^{-1}_{y,z} = (1+\Delta_x + \Delta_x \Delta_z)^{-1}(1+\Delta_y + \Delta_y \Delta_z)^{-1}  \\ = 
(1 + \Delta_x + \Delta_y + \Delta_x \Delta_y + (\Delta_x + \Delta_y + 2 \Delta_x \Delta_y) \Delta_z + \Delta_x \Delta_y \Delta^2_z)^{-1} = 
(1 + \e_{h}^{-1} \Delta_z (1 + \Delta_z))^{-1}, 
\end{multline*} 
and the first assertion follows.

We use the same strategy for the second assertion. By \eqref{solutions}, we have 
\begin{multline*}
W^{-1}_{z,x} W^{-1}_{z,y} = 
 (1+\Delta_z + \Delta_x \Delta_z)^{-1}(1+\Delta_z + \Delta_y \Delta_z)^{-1}  \\
 = (1 + (\Delta_x + \Delta_y +2) \Delta_z + (1+\Delta_x + \Delta_y + \Delta_x \Delta_y) \Delta_z^2)^{-1}  = 
 ((\Delta_z+1)^2 - \e_{h}^{-1} \Delta_z)^{-1}
\end{multline*} 
and the second assertion follows.

As for the third assertion, we use the following 
system of equations
$$
\Delta_{x_1} \Delta_{y_1} = \e_{h_1}^{-1}, 
-\Delta_{x_1}-\Delta_{y_1}  = \e_{h_1}^{-1}, 
\Delta_{x_2} \Delta_{y_2} = \e_{h_2}^{-1}, 
-\Delta_{x_2}-\Delta_{y_2}  = \e_{h_2}^{-1}
$$
as we are interested in the action of the following operator on $\sgt_{h_1}(x_1,y_1) \sgt_{h_2}(x_2,y_2)$
\begin{multline*} 
W^{-1}_{x_1,x_2} W^{-1}_{y_1,x_2} W^{-1}_{x_1,y_2} W^{-1}_{y_1,y_2}  
= (1+ \e^{-1}_{h_1} \Delta_{x_2} (1+ \Delta_{x_2}))^{-1} (1+ \e^{-1}_{h_1} \Delta_{y_2} (1+ \Delta_{y_2}))^{-1} \\
= (1 + \e^{-1}_{h_1} (\Delta_{x_2} (1+ \Delta_{x_2})+\Delta_{y_2} (1+ \Delta_{y_2})) + \e^{-2}_{h_1} 
\Delta_{x_2} (1+ \Delta_{x_2}) \Delta_{y_2} (1+ \Delta_{y_2}))^{-1},
\end{multline*} 
where in the first step we used the first assertion.
Since 
$$
\Delta_{x_2} (1+ \Delta_{x_2})+\Delta_{y_2} (1+ \Delta_{y_2}) = 
\Delta_{x_2} + \Delta_{y_2} + (\Delta_{x_2} + \Delta_{y_2})^2 - 2 \Delta_{x_2} \Delta_{y_2} 
= - 3 \e^{-1}_{h_2} + \e^{-2}_{h_2}
$$
and
$$
 \Delta_{x_2} (1+ \Delta_{x_2}) \Delta_{y_2} (1+ \Delta_{y_2}) = 
 \Delta_{x_2} \Delta_{y_2} (1 + \Delta_{x_2} + \Delta_{y_2} + \Delta_{x_2} \Delta_{y_2}) = \e^{-1}_{h_2}, 
$$
the third assertion follows.
\end{proof}

A simple consequence of Lemma~\ref{W12} \eqref{W12:3} is the following.

\begin{cor}
\label{corsecond}
We have 
$$
U_{x_1,x_2} U_{y_1,x_2} U_{x_1,y_2} U_{y_1,y_2} \sgt_{h_1}(x_1,y_1) \sgt_{h_2}(x_2,y_2)
= \sgt_{h_1}(x_1,y_1) \sgt_{h_2}(x_2,y_2)
$$
for all integers $h_1,h_2$ if and only if 
\begin{multline*} 
A_{x_1,x_2} A_{y_1,x_2} A_{x_1,y_2} A_{y_1,y_2} \sgt_{h_1}(x_1,y_1) \sgt_{h_2}(x_2,y_2) \\
= (1 - 3 \e^{-1}_{h_1} \e^{-1}_{h_2} + \e^{-1}_{h_1} \e^{-2}_{h_2} + \e^{-2}_{h_1} \e^{-1}_{h_2}) 
\sgt_{h_1}(x_1,y_1) \sgt_{h_2}(x_2,y_2). 
\end{multline*} 
\end{cor} 

Now we can easily characterize formal power series $A(x_1,x_2)$ that satisfy Conditions \eqref{1} and \eqref{2} for annihilating formal power series. Note that Condition \eqref{1} is trivially satisfied if $h<0$ and Condition \eqref{2} is trivially satisfied if at least one of $h_1,h_2$ is negative.
\begin{lemma}
\label{characterize}
Let $A(x,y)$ be formal power series. Then the following is true:
\begin{enumerate}
\item\label{characterize:1}
For all $h \ge 0$, it holds $A_{x,y} \sgt_h(x,y)=\sgt_h(x,y)$
if and only if $A(x,\iota(x))=1$.
\item\label{characterize:2} For all $h_1,h_2 \ge 0$, it holds  
\begin{multline*}  
A_{x_1,x_2} A_{y_1,x_2} A_{x_1,y_2}  A_{y_1,y_2} \sgt_{h_1}(x_1,y_1) \sgt_{h_2}(x_2,y_2) \\
= (1 - 3 \e^{-1}_{h_1} \e^{-1}_{h_2} + \e^{-1}_{h_1} \e^{-2}_{h_2} + \e^{-2}_{h_1} \e^{-1}_{h_2}) 
\sgt_{h_1}(x_1,y_1) \sgt_{h_2}(x_2,y_2)
\end{multline*} 
if and only if 
\begin{multline} 
\label{equation}
A(x_1,x_2) A(x_1,\iota(x_2)) A(\iota(x_1),x_2) A(\iota(x_1),\iota(x_2)) \\ = \frac{(1+x_1+x_2)(1-x_1 x_2)(1+x_1+x_1 x_2)(1+x_2+x_1 x_2)}{(1+x_1)^2 (1+x_2)^2}.
\end{multline}
\end{enumerate} 
\end{lemma} 

\begin{proof}
We reformulate Conditions~\eqref{1} and \eqref{2}, where we use Corollary~\ref{corsecond} for the second 
condition. Since $A_{x,y}$ is a symmetric operator in $x,y$ and 
$A_{x_1,x_2} A_{y_1,x_2} A_{x_1,y_2} A_{y_1,y_2}$ is a symmetric operator in $x_1,y_1$ and 
$x_2,y_2$, the first operator can be expanded in terms of the basis 
$$
(\Delta_x \Delta_y)^u (-\Delta_x-\Delta_y)^v
$$
and the second in terms of the basis 
$$
(\Delta_{x_1} \Delta_{y_1})^{u_1} (-\Delta_{x_1}-\Delta_{y_1})^{v_1} 
(\Delta_{x_2} \Delta_{y_2})^{u_2} (-\Delta_{x_2}-\Delta_{y_2})^{v_2}. 
$$
As the action of $\Delta_x \Delta_y$ on $\sgt_h(x,y)$ is the same as the action 
of $-\Delta_x-\Delta_y$, we can identify $\Delta_x \Delta_y$ with $-\Delta_x-\Delta_y$ in the first case and 
$\Delta_{x_1} \Delta_{y_1}$ with $-\Delta_{x_1} -\Delta_{y_1}$ as well as 
$\Delta_{x_2} \Delta_{y_2}$ with $-\Delta_{x_2} -\Delta_{y_2}$ in the second case, and this is equivalent to 
$$
\Delta_y = \iota(\Delta_x), \Delta_{y_1} = \iota(\Delta_{x_1}),\Delta_{y_2} = \iota(\Delta_{x_2})
$$
since solving  
$$
\Delta_{x} \Delta_{y} = -\Delta_{x} -\Delta_{y} 
$$
for $\Delta_y$ gives $\Delta_y = \iota(\Delta_x)$. This leads immediately to the first assertion.

As for the second assertion, we use Corollary~\ref{corsecond} to translate the condition into 
\begin{multline}\label{cond2}
A(\Delta_{x_1},\Delta_{x_2}) A(\iota(\Delta_{x_1}),\Delta_{x_2}) A(\Delta_{x_1},\iota(\Delta_{x_2})) A(\iota(\Delta_{x_1}),\iota(\Delta_{x_2}))\\
= 1 - 3 \e_{h_1}^{-1} \e_{h_2}^{-1} + 
\e_{h_1}^{-1} \e_{h_2}^{-2} + \e_{h_1}^{-2} \e_{h_2}^{-1}.
\end{multline}
Using the following identities
\begin{align*}
\e_{h_1}^{-1}&= \Delta_{x_1} \iota(\Delta_{x_1}) = - \Delta_{x_1} - \iota(\Delta_{x_1})= - \frac{\Delta_{x_1}^2}{1+\Delta_{x_1}}, \\
\e_{h_2}^{-1}&=\Delta_{x_2} \iota(\Delta_{x_2}) = -\Delta_{x_2} -\iota(\Delta_{x_2}) = - \frac{\Delta_{x_2}^2}{1+\Delta_{x_2}},
\end{align*}
the right-hand side of \eqref{cond2} can be transformed into 
$$\frac{(1+\Delta_{x_1}+\Delta_{x_2})(1-\Delta_{x_1} \Delta_{x_2})(1+\Delta_{x_1}+\Delta_{x_1} \Delta_{x_2})(1+\Delta_{x_2}+\Delta_{x_1} \Delta_{x_2})}{(1+\Delta_{x_1})^2 (1+\Delta_{x_2})^2}.$$
\end{proof} 

Finally we have gathered enough information to determine the annihilating series.

\begin{prop}
\label{hiddencharacterize}
The following holds:
\begin{enumerate}
\item\label{hiddencharacterize:1}  A formal power series $B(x,y)$ is a solution of 
\begin{equation} 
\label{secondgen} 
B(x_1,x_2) B(x_1,\iota(x_2)) B(\iota(x_1),x_2) B(\iota(x_1),\iota(x_2))=1, 
\end{equation}
 if and only if there is a formal power series 
$P(x,y)$ and a fourth root of unity $\zeta$ with  
\begin{equation}
\label{T}  
B(x,y)=\zeta \frac{P(x,y) P(y,x)}{P(\iota(x),y) P(\iota(y),x)}.
\end{equation} 
\item\label{hiddencharacterize:2} A formal power series $A(x,y)$ is annihilating if and only if it can be written as 
$$
A(x,y) = (1+x+y)  \frac{P_0(x,y) P_0(y,x)}{P_0(\iota(x),y) P_0(\iota(y),x)} \frac{P(x,y)P(y,x)}{P(\iota(x),y)P(\iota(y),x)},
$$ 
where $P_0(x,y) = 1-(x+1)(y+1) - \rho (x+2)$ and $\rho$ is a primitive sixth root of unity, and $P(x,y)$ is a formal power series that satisfies 
\begin{equation}
\label{singleequation}  
\frac{P(x,\iota(x)) P(\iota(x),x)}{P(x,x)P(\iota(x),\iota(x))}=1.
\end{equation}
Another choice for the factor $1+x+y$ of $A(x,y)$ is $1- x y$ and other choices for $P_0(x,y)$ are as follows.
\begin{align*} 
P_0(x,y) &=
1-(x+1)(y+1) + \rho (x+1) (x+2) \\
P_0(x,y)&= \sqrt{(1- (x+1) (y+1) - \rho (2+x))(1- (x+1) (y+1) - \rho^5(x+2))} \\ 
 &= \sqrt{4+ 6 x + 3 x^2+2 y +y^2 + 5 x y + 3 x^2 y  + 2 x y^2 +x^2 y^2} \\
P_0(x,y)&= 
\sqrt{(1-(x+1) (y+1) + \rho (x+1) (x+2))(1-(x+1) (y+1) + \rho^5 (x+1) (x+2))} \\
&=\sqrt{4+10 x + 11 x^2+5 x^3+x^4- 2 y+y^2-3 x y-2 x^2 y-x^3 y+2 x y^2+x^2 y^2}
\end{align*} 
\end{enumerate} 
\end{prop} 

\begin{proof} We prove the first part. A simple calculation shows that any $B(x,y)$ as given in \eqref{T} is a solution of \eqref{secondgen}.

For the converse note that we may assume without loss of generality that the constant term of  $B(x_1,x_2)$ is 
$1$ as the constant term of a solution of 
\eqref{secondgen} is obviously a fourth root of unity and thus every solution can be obtained from those with constant term $1$ by multiplying with 
the fourth root of unity.  A simple calculation using the symmetry of $B(x_1,x_2)$ as well as \eqref{secondgen} shows that 
\eqref{T} is satisfied with $\zeta=1$
when setting
$$
P(x_1,x_2) = B(x_1,x_2)^{3/8} B(x_1,\iota(x_2))^{1 / 4} B(\iota(x_1),\iota(x_2))^{1 / 8}.
$$

To prove the second assertion, note that two solutions of \eqref{equation} are the symmetric factors in the numerator of the right-hand side of \eqref{equation}, namely
$$
A_1(x,y)=x+y+1 \quad \text{and} \quad A_2(x,y)=1-x y, 
$$
however, in both cases we have 
$$
A_i(x,\iota(x))=\frac{1+x+x^2}{1+x}, 
$$
and so they are not annihilating series as they do not satisfy the first condition. Now suppose $A(x,y)$ is any solution of \eqref{equation}, then $A(x,y) A_1(x,y)^{-1} \eqcolon B(x,y)$ is a solution of 
\eqref{secondgen}. Phrased differently, any 
solution $A(x,y)$ of \eqref{equation} can be written as $A_1(x,y) B(x,y)$ for an appropriate solution $B(x,y)$ of \eqref{secondgen}.
This applies in particular to $A(x,y)=A_2(x,y)$; therefore, we work with $A_1(x,y)$ in the following.

Let 
$
P_0(x,y) = 1-(x+1)(y+1) - \rho (x+2)  
$
where $\rho$ is primitive sixth root of unity  and set
$$
B_0(x,y)=\frac{P_0(x,y) P_0(y,x)}{P_0(\iota(x),y) P_0(\iota(y),x)}
$$
then 
$$
B_0(x,\iota(x)) = \frac{x+1}{1+x+x^2} = A_1(x,\iota(x))^{-1},
$$
and so it follows that $A_1(x,y) B_0(x,y)$ is a solution that satisfies 
both conditions for annihilating formal power series, where we have also employed the first assertion.
Note that our choice for $P_0(x,y)$ is the polynomial solution of smallest degree satisfying 
$$
\frac{P_0(x,\iota(x)) P_0(\iota(x),x)}{P_0(\iota(x),\iota(x)) P_0(x,x)} = \frac{x+1}{1+x+x^2}.
$$
Other solutions of this equation are given in the assertion.

To summarize, all annihilating series are of the form $A_1(x,y) B(x,y)$, where 
$B(x,y)$ is a solution of \eqref{secondgen} and 
$$
B(x,\iota(x)) =  \frac{x+1}{1+x+x^2}.
$$
Now all such $B(x,y)$ can be written as $B_0(x,y) C(x,y)$, where $C(x,y)$ is a solution of \eqref{secondgen} and 
$C(x,\iota(x))=1$. By the first assertion, these $C(x,y)$ are precisely those of the following form 
$$
C(x,y) = \zeta \frac{P(x,y)P(y,x)}{P(\iota(x),y)P(\iota(y),x)}
$$
where 
$$
\zeta \frac{P(x,\iota(x))P(\iota(x),x)}{P(\iota(x),\iota(x))P(x,x)}=1
$$
 and $\zeta$ is a fourth root of unity. Comparing constant terms, we can conclude that this equation only has a solution if 
 $\zeta=1$.
\end{proof} 

To find solutions of \eqref{singleequation}, it is useful to apply a transformation of the variables $x,y$, such that the involution $z \to \iota(z)$ is transformed into the 
involution $z \to \frac{1}{z}$. Namely, if we are applying the simple transformation 
$x \to x-1$ and $y \to y-1$, then $y \to \iota(x)$ is transformed to $y-1 \to - \frac{x-1}{x}$, i.e. $y \to \frac{1}{x}$. 
Then \eqref{singleequation} is transformed into 
$$
\frac{Q(x,x^{-1}) Q(x^{-1},x)}{Q(x,x)Q(x^{-1},x^{-1})}=1.
$$
Such $Q(x,y)$ are for instance given by  
$$
\frac{\sum\limits_{i,j=-n}^{n} u_{i,|j|} x^i y^j}{\sum\limits_{i,j=-n}^{n} v_{i,|j|} x^i y^j}. 
$$
A solution of \eqref{singleequation} is then given by $P(x,y)=Q(x+1,y+1)$. 

\section{First part of the proof of Theorem~\texorpdfstring{\ref{main}}{\ref*{main}}}
\label{sec:first}

The proof of Theorem~\ref{main} is by induction with respect to $h$. Observe that the case $h=0$ follows from the fact that 
$\MT_0(\mathbf{k}_n) = \GT_0(\mathbf{k}_n) = 1$,
and that both $\st^{-1}_{x,y}$ and $A_{x,y}$ act as the identity on functions independent of $x$ and $y$: for $\st^{-1}_{x,y}$, this holds because it is 
obviously true for  $\st_{x,y}$; for $A_{x,y}$, this follows from the fact that the constant term of $A(x,y)$ is always $1$ (see Lemma~\ref{characterize}~\eqref{characterize:1}).

The induction step it is based on the following recursion
$$
\MT_h(\mathbf{k}_n) = \sum_{\mathbf{l}_{n-1} \preceq  \mathbf{k}_{n} \atop \mathbf{l}_{n-1} \text{ strictly increasing} }
\MT_{h-1}(\mathbf{l}_{n-1}).
$$
The recursion can also be written as 
\begin{equation}
\label{MTrec} 
\MT_h(\mathbf{k}_n)  
=  \left. \left[ \prod_{i=1}^n \st_{k^L_i,k^R_i} 
\sum_{\mathbf{l}_{n-1} \in \cart\limits_{i=1}^{n-1} [k^{R}_i,k^{L}_{i+1}]}
\MT_{h-1}(\mathbf{l}_{n-1}) \right] \right|_{k^L_i=k^R_i=k_i},
\end{equation} 
which follows from the explanation concerning the operator $\st_{x,y}$ in Section~\ref{sec:main}, see in particular \eqref{strictop}. Consequently, it suffices to show
\begin{multline} 
\label{rectoshow} 
\prod_{1 \le i < j \le n} \st^{-1}_{k_i,k_j} A_{k_i,k_j} \GT_{h}(\mathbf{k}_n) \\=
\left. \left[ \prod_{i=1}^n \st_{k^L_i,k^R_i}
\sum_{\mathbf{l}_{n-1} \in \cart\limits_{i=1}^{n-1} [k^{R}_i,k^{L}_{i+1}]}  
\prod_{1 \le i < j \le n-1} \st^{-1}_{l_i,l_j} A_{l_i,l_j} \GT_{h-1}(\mathbf{l}_{n-1}) \right] \right|_{k^L_i=k^R_i=k_i}
\end{multline} 
for $h \ge 1$. Again, let 
$W_{x,y} = \e_x \st_{x,y}$. We aim to replace all occurrences of $\st_{x,y}$ in \eqref{rectoshow} by $W_{x,y}$. As 
$\st^{-1}_{x,y} = \e_x W^{-1}_{x,y}$, we see that 
\begin{multline*} 
\prod_{1 \le i < j \le n} \st^{-1}_{k_i,k_j} A_{k_i,k_j} \GT_h(\mathbf{k}_n) 
= \prod_{1 \le i < j \le n} \e_{k_i} W^{-1}_{k_i,k_j} A_{k_i,k_j} \GT_h(\mathbf{k}_n) \\  
= \prod_{1 \le i < j \le n}  W^{-1}_{k_i,k_j} A_{k_i,k_j} \GT_h(\mathbf{k}_n- \mathbf{0\!\uparrow}_n + \const{n}{n}) \\
=  \prod_{1 \le i < j \le n}  W^{-1}_{k_i,k_j} A_{k_i,k_j} \GT_h(\mathbf{k}_n- \mathbf{0\!\uparrow}_n),
\end{multline*}
where we have used the property
\begin{equation}
	\label{constant}
	\GT_h(\mathbf{k}_n) = \GT_h(\mathbf{k}_n+\const{c}{n})
\end{equation} 
with $\const{c}{n}\coloneq\underbrace{(c,\ldots,c)}_{n}$ for any integer $c$ and again $\mathbf{0\!\uparrow}_n=(1,2,\ldots,n)$. Likewise, 
$$
\prod_{1 \le i < j \le n-1} \st^{-1}_{l_i,l_j} A_{l_i,l_j} \GT_{h-1}(\mathbf{l}_{n-1}) =
\prod_{1 \le i < j \le n-1} W^{-1}_{l_i,l_j} A_{l_i,l_j} \GT_{h-1}(\mathbf{l}_{n-1}-\mathbf{0\!\uparrow}_{n-1}).
$$
Thus, it is convenient to consider the following transformation of the counting functions $\GT_h(\mathbf{k}_n)$: 
$$
\overline \GT_h(\mathbf{k}_n) \coloneq \GT_h(\mathbf{k}_n-\mathbf{0\!\uparrow}_n).
$$
The transformed quantity satisfies the recursion
$$
\overline \GT_h(\mathbf{k}_n) = \sum_{\mathbf{l}_{n-1} \prec \mathbf{k}_n}  \overline 
\GT_{h-1}(\mathbf{l}_{n-1}),
$$
see also \eqref{GTrec}.
Now \eqref{rectoshow} is equivalent to 
\begin{multline} 
\label{rectoshow1}
\prod_{1 \le i < j \le n} W^{-1}_{k_i,k_j} A_{k_i,k_j}  \overline \GT_h(\mathbf{k}_n)\\
=\left. \left[ \prod_{i=1}^n W_{k^L_i,k^R_i} 
\sum_{\mathbf{l}_{n-1} \in \cart\limits_{i=1}^{n-1} [k^{R}_i,k^{L}_{i+1}-1]}  
\prod_{1 \le i < j \le n-1} W^{-1}_{l_i,l_j} A_{l_i,l_j}  \overline \GT_{h-1}(\mathbf{l}_{n-1}) \right] \right|_{k^L_i=k^R_i=k_i}.
\end{multline} 

Our first goal in this section is to prove Lemma~\ref{pq}, where we perform the telescoping step that happened in the proof of the simpler Theorem~\ref{GT} in \eqref{reduced}. There it was possible to conclude immediately that only $n$ of the $2^{n-1}$ terms which emerged from telescoping remain. Eventually, this will also hold here, however, this reduction will only be possible in the final step of the proof of Theorem~\ref{main}, based on Lemma~\ref{signreversing}.

From now on we assume that 
$n+h$ is even until Section~\ref{odd}, where we consider the case $n+h$ is odd. We set 
\begin{equation} 
\label{U} 
U_{x,y} \coloneq W^{-1}_{x,y} A_{x,y} 
\end{equation} 
as well as  
\begin{multline} 
\label{MTsummand}
t_h(\mathbf{l}_{n-1}) \coloneq \prod_{1 \le i < j \le n-1} U_{l_i,l_j}  \overline \GT_{h-1}(\mathbf{l}_{n-1})  \\
=
\prod_{1 \le i < j \le n-1} 
U_{l_i,l_j} \pf_{1 \le i < j \le n+h-2} \left(  
\begin{array}{c|c}  \TGT_{h-1}(\mathbf{l}_{n-1}) &  \RGT_{h-1,h-1}(\mathbf{l}_{n-1})  \\ \hline 
& \T(0)_{h-1} \end{array} \right),
\end{multline} 
which is just the summand in \eqref{rectoshow1}, and also 
\begin{equation}
\label{MTintegral} 
T_{h}(\mathbf{l}_{n-1}) \coloneq   \prod_{1 \le i < j \le n-1} 
U_{l_i,l_j} \pf_{1 \le i < j \le n+h-2} \left(  
\begin{array}{c|c}  \TGT_{h}(\mathbf{l}_{n-1}) &  \RGT_{h,h-1}(\mathbf{l}_{n-1})  \\ \hline 
& \T(0)_{h-1} \end{array} \right),
\end{equation}
which agrees with $t_h(\mathbf{l}_{n-1})$, except for two places where $h-1$ has been replaced by $h$. These definitions are analogous 
to $s_h(\mathbf{l}_{n-1})$ and $S_h(\mathbf{l}_{n-1})$ in the proof of Theorem~\ref{GT} (see \eqref{GTsummand} and \eqref{GTintegral}, respectively) and only differ by the operator $\prod_{1 \le i < j \le n-1} 
U_{l_i,l_j}$. They are also related via $\Delta_{\mathbf{l}_{n-1}} T_h(\mathbf{l}_{n-1}) = t_h(\mathbf{l}_{n-1})$. 

In Lemma~\ref{pq}, we also use an extension of the notation in \eqref{ignore}: We denote by $\mathbf{k}_n^{\mathbf{p}_m,\mathbf{q}_{m-1}}$ the vector $\mathbf{k}_n=(k_1,\ldots,k_n)$ excluding all $k_i$ such that $i$ appears in $\mathbf{p}_m$ or $\mathbf{q}_{m-1}$.

\begin{lemma} 
\label{pq} 
Assuming that $n+h$ is even, the right-hand side of \eqref{rectoshow1} is equal to the sum of expressions of the following form
\begin{multline}
\label{interpret}
\left. 
(-1)^{\sum_{i=1}^m p_i + \sum_{i=1}^{m-1} q_i +1}
\prod_{r=1}^{m-1} 
\prod_{1 \le i < q_r \atop i \not= p_s,q_s} 
U_{k_i,k^{L}_{q_r}} U_{k_i,k^{R}_{q_r}}
\prod_{q_r < j \le n \atop j \not= p_s,q_s} 
U_{k^{L}_{q_r},k_j} U_{k^{R}_{q_r},k_j}  
\sgt_{h}(k^{L}_{q_r},k^{R}_{q_r})    \right|_{k^{L}_{q_r}=k^{R}_{q_r}}  \\
\times \prod_{j \in \bigcup_{i=1}^m [p_i+1,q_i-1]}   \e_{k_j} 
T_{h}(\mathbf{k}_n^{\mathbf{p}_m,\mathbf{q}_{m-1}}), 
\end{multline} 
where the sum is over all $m \ge 1$ and all strictly increasing integers sequence $\mathbf{p}_m$ and 
$\mathbf{q}_{m-1}$ of length $m$ and $m-1$, respectively, such that $\mathbf{q}_{m-1}$ interlaces $\mathbf{p}_m$  in the following 
stricter sense 
$$1 \le p_1 < q_1 < p_2 < q_2 < \ldots < q_{m-1} < p_m \le n,$$
setting $q_m=n+1$.
\end{lemma}

For the proof, we need the following simple auxiliary lemma.

\begin{lemma} 
\label{asym}
Let $p(x,y)$ be a symmetric polynomial in $x$ and $y$. Then the following holds:
\begin{enumerate} 
\item\label{asym:1} 
If $q(x,y)$ is an antisymmetric function in $x$ and $y$, then 
$p(\e_x,\e_y) q(x,y)$ is also antisymmetric in $x$ and $y$.
\item\label{asym:2}
The function $p(\Delta_x,\Delta_y) \sgt_h(x,y)$ can be written as a linear combination of 
$\sgt_g(x,y)$'s for $0 \le g \le h$. 
\end{enumerate} 
\end{lemma} 

\begin{proof} 
As for the first assertion, we may assume without loss of generality that $p(x,y) = x^u y^v + x^v y^u$ for some integers $u,v$. Then 
$p(\e_x,\e_y) q(x,y) = q(x+u,y+v) + q(x+v,y+u)$. Now 
$$
q(x+u,y+v) + q(x+v,y+u) = -q(y+v,x+u) - q(y+u,x+v)
= - q(y+u,x+v) -q(y+v,x+u).
$$

As for the second assertion, it suffices to show the assertion for $p(x,y)= x y$ and $p(x,y)=x+y$ as every symmetric polynomial in $x,y$ 
is expressible as a polynomial in these two polynomials. The assertion now follows from 
\eqref{Bsym1}  and \eqref{Bsym2}.
\end{proof} 

\begin{proof}[Proof of Lemma~\ref{pq}]
From the right-hand side of \eqref{rectoshow1}, we first consider 
$$
\sum_{\mathbf{l}_{n-1} \in \cart\limits_{i=1}^{n-1} [k^{R}_i,k^{L}_{i+1}-1]}  t_h(\mathbf{l}_{n-1}) = 
\sum_{\mathbf{l}_{n-1} \in \cart\limits_{i=1}^{n-1} [k^{R}_i,k^{L}_{i+1}-1]} \Delta_{\mathbf{l}_{n-1}} T_h(\mathbf{l}_{n-1}).
$$ 
By telescoping, this is equal to 
\begin{equation}
\label{inter}
\sum_{(s_1,\ldots,s_{n-1}) \in \{L,R\}^{n-1}} (-1)^{\{\# i: s_i=R\}} 
T_{h}(k^{s_1}_{1+[s_1=L]},k^{s_2}_{2+[s_2=L]},\ldots,k^{s_{n-1}}_{{n-1}+[s_{n-1}=L]}), 
\end{equation}
where the \Dfn{Iverson bracket} is defined as $[\text{statement}]=1$ if the statement is true and $0$ otherwise.

We encode $(s_1,\ldots,s_{n-1}) \in \{L,R\}^{n-1}$ as follows:
\begin{itemize}
	\item Let $1\le p_1 < p_2 < \ldots < p_m \le n$ denote all indices
	$p$ such that $s_{p-1}=R$ and $s_{p}= L$. Additionally, include $p=1$ if $s_1=L$, and $p=n$ if $s_{n-1}=R$. These indices indicate that both $k^{L}_p$ and $k^{R}_p$ are missing from the argument of $T_{h}$.
	\item Similarly, let $2 \le q_1 < q_2 < \ldots < q_{m'} \le n-1$ denote all indices $q$ such that $s_{q-1}=L$ and $s_q=R$. These indices correspond to positions where both $k^{L}_q$ and $k^{R}_q$ appear in the argument of $T_{h}$.
\end{itemize}
Note that $m'=m-1$ and 
$$p_1 < q_1 < p_2 < q_2 < \ldots < q_{m-1} < p_m.$$

We recast the sign $(-1)^{\{\# i: s_i=R\}}$ due to the following observation: The number of $R$'s among $s_1,\ldots,s_{n-1}$ is given by
$$p_1-1+\left[ (p_2-1)-(q_1-1) \right]+ \left[ (p_3-1)-(q_2-1) \right]+\ldots+ \left[ (p_{m}-1)-(q_{m-1}-1) \right]=\sum_{i=1}^m p_i - \sum_{i=1}^{m-1} q_i -1.$$
Thus, the sign can be expressed as $(-1)^{\sum_{i=1}^m p_i + \sum_{i=1}^{m-1} q_i +1}$. The expression in \eqref{inter} can now be rewritten as 
\begin{multline}
\label{integral} 
(-1)^{\sum_{i=1}^m p_i + \sum_{i=1}^{m-1} q_i +1}\\
\times T_h(k_1^R,\ldots,k_{p_1-1}^R,k_{p_1+1}^L,\ldots,k_{q_1}^L,k_{q_1}^R,\ldots,k_{q_{m-1}}^L,k_{q_{m-1}}^R,\ldots,k_{p_{m}-1}^R,k_{p_{m}+1}^L,\ldots,k_n^L).
\end{multline}

Regarding the right-hand side of \eqref{rectoshow1}, the operator $W_{k^L_{q_i},k^R_{q_i}}$ has not yet been taken into account. Therefore, we next simplify $\left. W_{l_k,l_{k+1}} T_{h}(\mathbf{l}_{n-1}) \right\vert_{l_{k}=l_{k+1}}$. The result will be applied to all $(l_k,l_{k+1})=(k^L_{q_i},k^R_{q_i})$ for $i=1,\ldots,m-1$. By the definition of $T_{h}(\mathbf{l}_{n-1})$, we have 
$$
W_{l_k,l_{k+1}} T_{h}(\mathbf{l}_{n-1}) = A_{l_k,l_{k+1}} \prod_{1 \le i < j \le n-1 \atop (i,j) \not= (k,k+1)}  U_{l_i,l_j} S_{h}(\mathbf{l}_{n-1}).
$$
We know from \eqref{decompose} that we can decompose $S_{h}(\mathbf{l}_{n-1})$ into the sum of an antisymmetric function in $l_k$ and $l_{k+1}$, and 
\begin{equation}
\label{symF} 
\sgt_{h}(l_{k},l_{k+1}) S_{h}(\mathbf{l}^{k,k+1}_{n-1}).
\end{equation} 
As the operator 
$$
A_{l_k,l_{k+1}} \prod_{1 \le i < j \le n-1 \atop (i,j) \not= (k,k+1)} U_{l_i,l_j} 
$$
is symmetric in $l_k$ and $l_{k+1}$, the antisymmetric part of $S_{h}(\mathbf{l}_{n-1})$ vanishes by Lemma~\ref{asym} after applying the operator and letting $l_{k}=l_{k+1}$. 

Now we include $\prod_{i=1}^n W_{k_i^L,k_i^R}$ as well as the identification $k_i^L=k_i^R=k_i$ from the right-hand side of \eqref{rectoshow1} in \eqref{integral}, and then apply the conclusion from the previous paragraph to all pairs $(l_k,l_{k+1})=(k^L_{q_i},k^R_{q_i})$. In \eqref{penultimate} below, we additionally rearrange the operators that come from $T_h$ in \eqref{integral}, namely
$$
\left. \prod_{1 \le i < j \le n-1} U_{l_i,l_j} \right|_{\mathbf{l}_{n-1}=(k_1,\ldots,k_{p_1-1},k_{p_1+1},\ldots,k_{q_1-1},k_{q_1}^L,k_{q_1}^R,k_{q_1+1},\ldots,k_{q_{m-1}-1},k_{q_{m-1}}^L,k_{q_{m-1}}^R,k_{q_{m-1}+1},\ldots,k_{p_{m}-1},k_{p_{m}+1},\ldots,k_n)}
$$
in the following way: in the first line, we list all operators involving precisely one of $k_{q_r}^L$ or $k_{q_r}^R$; in the second line, we have all operators that involve two of these; the remaining operators feature in $T_{h}(\mathbf{k}_n^{\mathbf{p}_m,\mathbf{q}_{m-1}})$ in the third line. We still need to keep track of $k_{q_i}^R$ and $k_{q_i}^L$ because both parameters appear simultaneously. For all other pairs $(k_i^R,k_i^L)$, we can apply the identification $k_i^L=k_i^R=k_i$ everywhere since at most one of these parameters appears at a time.
More precisely, we obtain  
\begin{multline} 
\label{penultimate} 
(-1)^{\sum_{i=1}^m p_i + \sum_{i=1}^{m-1} q_i +1}
\prod_{r=1}^{m-1} 
\prod_{1 \le i < q_r \atop i \not= p_s,q_s} 
U_{k_i,k^{L}_{q_r}} U_{k_i,k^{R}_{q_r}}
\prod_{q_r < j \le n \atop j \not= p_s,q_s} 
U_{k^{L}_{q_r},k_j} U_{k^{R}_{q_r},k_j} \\
\left.  
\prod_{1 \le i < j \le m-1} U_{k^{L}_{q_i},k^{L}_{q_j}} 
U_{k^{L}_{q_i},k^{R}_{q_j}}
U_{k^{R}_{q_i},k^{L}_{q_j}}
U_{k^{R}_{q_i},k^{R}_{q_j}}
\prod_{r=1}^{m-1}  A_{k^{L}_{q_r},k^{R}_{q_r}} \sgt_{h}(k^{L}_{q_r},k^{R}_{q_r})    \right|_{k^{L}_{q_r}=k^{R}_{q_r}=k_{q_r}} \\
\times \prod_{j \in \bigcup_{i=1}^m [p_i+1,q_i-1]}   \e_{k_j} 
T_{h}(\mathbf{k}_n^{\mathbf{p}_m,\mathbf{q}_{m-1}}).
\end{multline} 
Note that for $\prod_{i \not= q_s} W_{k^L_i,k^R_i}$, we use the fact that $W_{x,y}$ acts as the identity on functions independent of $x$ and as the  
shift operator $\e_x$ on functions independent of $y$; we encode this by $\prod_{j \in \bigcup_{i=1}^m [p_i+1,q_i-1]}   \e_{k_j}$.

Clearly, Conditions~\eqref{1} and \eqref{2} for annihilating formal power series in Definition~\ref{def:annihilating} are  ``made'' for simplifying \eqref{penultimate} so that we obtain the one given in the statement of the lemma.
 
To explain the subtlety why we  omit $k_{q_r}$ when specializing $k^{L}_{q_r}=k^{R}_{q_r}=k_{q_r}$ in the statement, we explain why the last displayed expression is actually independent of $k_{q_r}$. As the operator that is applied to the product 
$\prod_{r=1}^{m-1}  \sgt_{h}(k^{L}_{q_r},k^{R}_{q_r})$
is symmetric in $k^L_{q_r}$ and $k^R_{q_r}$ for every $r$, the application of the operator leads to an expression that is expressible as linear 
combinations of products of the form 
$\prod_{r=1}^{m-1}  \sgt_{h_r}(k^{L}_{q_r},k^{R}_{q_r})$
by Lemma~\ref{asym}.
Our claim follows after observing that $\sgt_{h_r}(k,k)=0$ for $h_r \not= 0$ and $\sgt_{0}(k,k)=1$.
\end{proof}

We will eventually show that for a fixed sequence $\mathbf{q}_{m-1}$, the sum of expressions in \eqref{interpret} over all $\mathbf{p}_{m}$ that are interlaced by $\mathbf{q}_{m-1}$ in the strict sense vanishes if $m \ge 2$, while for $m=1$ this sum is just the left hand side of \eqref{rectoshow1}. In order to perform the sum over all $\mathbf{p}_{m}$, it is necessary to get rid of a certain type of dependence of the operators on $\mathbf{p}_{m}$, which is the purpose of the following lemma.

\begin{lemma} 
\label{regrouping}
For $n \ge 1$, it holds
\begin{multline*} 
\prod_{j \in \bigcup_{i=1}^m [p_i+1,q_i-1]}   \e_{k_j} \left. \prod_{1 \le i < j \le n-2m+1} U_{l_i,l_j} \right|_{\mathbf{l}_{n-2m+1}=(k_1,\ldots,\widehat{k_{p_1}},\ldots,\widehat{k_{q_1}},\ldots,\widehat{k_{q_{m-1}}},\ldots,\widehat{k_{p_{m}}},\ldots,k_n)}  \\
=  
\prod_{r=1}^{m-1}  \prod_{i<q_r \atop i \not=p_s,q_s}  \e_{k_i}^2  
\prod_{i=1}^n A^{-2m+1}_{k_i,\bullet}
\prod_{1 \le i < j \le n} U_{k_i,k_j}
\end{multline*}
if applied to a function independent of $k_{p_s}$ and $k_{q_s}$ for all $s$ and
where $\bullet$ stands for any parameter which also does not appear in this function.
\end{lemma}

\begin{proof}
The left-hand side of the identity in the lemma can be written as 
%
	$$
	\prod_{j \in \bigcup_{i=1}^m [p_i+1,q_i-1]}   \e_{k_j}
	\prod_{r=1}^m \prod_{i<p_r \atop i \not=p_s,q_s} U^{-1}_{k_i,k_{p_r}} 
	\prod_{p_r<j \atop j \not=p_s,q_s} U^{-1}_{k_j,k_{p_r}} 
	\prod_{r=1}^{m-1}  \prod_{i<q_r \atop i \not=p_s,q_s} U^{-1}_{k_i,k_{q_r}}   
	\prod_{q_r<j \atop j \not=p_s,q_s} U^{-1}_{k_j,k_{q_r}}  
	\prod_{1 \le i < j \le n} U_{k_i,k_j} 
	$$
taking into account the fact that $U_{x,y}$ acts like the identity on functions independent of $x$ and $y$. 
As $W_{x,y}$ acts on functions independent of $y$ as $\e_x$ and 
on functions independent of $x$ as the identity, this simplifies as follows:
$$
\prod_{j \in \bigcup_{i=1}^m [p_i+1,q_i-1]}   \e_{k_j} 
\prod_{r=1}^m \prod_{i<p_r \atop i \not=p_s,q_s}  \e_{k_i} A^{-1}_{k_i,k_{p_r}}
\prod_{p_r<j \atop j \not=p_s,q_s}  A^{-1}_{k_{p_r},k_{j}}
\prod_{r=1}^{m-1}  \prod_{i<q_r \atop i \not=p_s,q_s}  \e_{k_i} A^{-1}_{k_i,k_{q_r}} 
\prod_{q_r<j \atop j \not=p_s,q_s}  A^{-1}_{k_{q_r},k_{j}} 
\prod_{1 \le i < j \le n} U_{k_i,k_j} 
$$
Due to the symmetry of $A(x,y)$ and by redistributing the shifts in  $\prod_{j \in \bigcup_{i=1}^m [p_i+1,q_i-1]} \e_{k_j}$, this can alternatively be written as 
$$
\prod_{r=1}^m
\prod_{i \not=p_s,q_s}  A^{-1}_{k_i,k_{p_r}}
\prod_{r=1}^{m-1}  \prod_{i<q_r \atop i \not=p_s,q_s}  \e_{k_i}^2 
\prod_{i \not=p_s,q_s}  A^{-1}_{k_i,k_{q_r}}  
\prod_{1 \le i < j \le n} U_{k_i,k_j}. 
$$
Finally note that, under the assumption in the lemma, we have 
$$
 \prod_{r=1}^m
\prod_{i \not=p_s,q_s}  A^{-1}_{k_i,k_{p_r}}
\prod_{r=1}^{m-1} 
\prod_{i \not=p_s,q_s}  A^{-1}_{k_i,k_{q_r}} = \prod_{i=1}^n A^{-2m+1}_{k_i,\bullet},  
$$
and the assertion follows.
\end{proof} 

Applying the previous lemma to the expression in the statement of Lemma~\ref{pq}, we obtain 
\begin{multline}
\label{applied}
\left. (-1)^{\sum_{i=1}^m p_i + \sum_{i=1}^{m-1} q_i +1}
\prod_{r=1}^{m-1} 
\prod_{1 \le i < q_r \atop i \not= p_s,q_s} 
\e_{k_i}^2 
U_{k_i,k^{L}_{q_r}} U_{k_i,k^{R}_{q_r}}
\prod_{q_r < j \le n \atop j \not= p_s,q_s} 
U_{k^{L}_{q_r},k_j} U_{k^{R}_{q_r},k_j}  
\sgt_{h}(k^{L}_{q_r},k^{R}_{q_r})    \right|_{k^{L}_{q_r}=k^{R}_{q_r}}  \\
\times \prod_{i=1}^n A^{-2m+1}_{k_i,\bullet} \prod_{1 \le i < j \le n} U_{k_i,k_j}
S_{h}(\mathbf{k}_n^{\mathbf{p}_m,\mathbf{q}_{m-1}}).
\end{multline} 

Our next task is to understand 
\begin{equation} 
\label{remainingoperator}
\left.
\prod_{r=1}^{m-1} 
\prod_{1 \le i < q_r \atop i \not= p_s,q_s} 
\e_{k_i}^2 
U_{k_i,k^{L}_{q_r}} U_{k_i,k^{R}_{q_r}}
\prod_{q_r < j \le n \atop j \not= p_s,q_s} 
U_{k^{L}_{q_r},k_j} U_{k^{R}_{q_r},k_j}  
\sgt_{h}(k^{L}_{q_r},k^{R}_{q_r})    \right|_{k^{L}_{q_r}=k^{R}_{q_r}}
\end{equation} 
better. Note that it is an operator that acts on the variables $k_i$ for $i \not= p_s,q_s$. For this purpose, we now discuss, for fixed $i$, the action of the operator
\begin{equation} 
\label{op} 
V_{k_i} \coloneq \prod_{r: i < q_r} \e_{k_i}^2 U_{k_i,k^{L}_{q_r}} U_{k_i,k^{R}_{q_r}} \prod_{r: q_r < i} U_{k^{L}_{q_r},k_i} U_{k^{R}_{q_r},k_i}
\end{equation} 
on the product $\prod_{r=1}^{m-1} \sgt_{h}(k^{L}_{q_r},k^{R}_{q_r})$. For every $r \in \{ 1,\dots,m-1\}$, the operator is symmetric in 
$k^L_{q_r}$ and $k^R_{q_r}$; it can therefore be expanded in the basis 
$$
(\Delta_{k^L_{q_r}} \Delta_{k^R_{q_r}})^u (-\Delta_{k^L_{q_r}}-\Delta_{k^R_{q_r}})^v
$$
for $u,v \ge 0$. By \eqref{Bsym1} and \eqref{Bsym2}, the effect of this basis element on $\sgt_{h}(k^{L}_{q_j},k^{R}_{q_j})$ is to decrease $h$ by $u+v$; we will see that for this reason it is useful to replace this basis element in the expansion of $V_{k_i}$ by $z_{j}^{u+v}$. By letting 
$$
U(x,y) \coloneq (1+ x + x y)^{-1} A(x,y),
$$
so that $U_{x,y} = U(\Delta_x,\Delta_y)$, 
compare also with $\eqref{U}, \eqref{W}, \eqref{A}$, 
this replacement in 
$V_{k_i}$ can also be written formally as
\begin{equation}
\label{defv}
 \prod_{r: i < q_r}  \left[ \e_{k_i}^2 U(\Delta_{k_i},y) U(\Delta_{k_i},\iota(y)) \right]_{-y-\iota(y)=y \cdot \iota(y) = z_{q_r}}  \prod_{r: q_r < i} 
 \left[ U(y,\Delta_{k_i}) U(\iota(y),\Delta_{k_i}) \right]_{-y-\iota(y)=y \cdot \iota(y) = z_{q_r}},
\end{equation} 
see also the proof of Lemma~\ref{characterize}. 
We let $V_{k_i}(z_{q_1},\ldots,z_{q_{m-1}})$ denote the polynomial in $z_{q_1}$,\ldots,\allowbreak$z_{q_{m-1}}$ that we obtain by interpreting 
$V_{k_i}$ this way.  
As $\sgt_g(k,k)= [g=0]$, we can replace 
$$
\left. 
\prod_{i \not= p_s,q_s} V_{k_i} \prod_{i=1}^{m-1}  \sgt_{h}(k^{L}_{q_i},k^{R}_{q_i}) \right|_{k^{L}_{q_r}=k^{R}_{q_r}}
$$
by 
$$
\langle z_{q_1}^{h} z_{q_2}^{h} \dots z_{q_{m-1}}^{h} \rangle \prod_{i \not= p_s,q_s} V_{k_i}(z_{q_1},\ldots,z_{q_{m-1}}),
$$
where $\langle z_{q_1}^{h} z_{q_2}^{h} \dots z_{q_{m-1}}^{h} \rangle$ extracts the coefficient of $z_{q_1}^{h} z_{q_2}^{h} \dots z_{q_{m-1}}^{h}$ of 
the subsequent expression.
Thus, \eqref{applied} simplifies as follows: 
\begin{multline*}
(-1)^{\sum_{i=1}^m p_i + \sum_{i=1}^{m-1} q_i +1} \langle z_{q_1}^{h} z_{q_2}^{h} \dots z_{q_{m-1}}^{h} \rangle \prod_{i \not= p_s,q_s}  V_{k_i}(z_{q_1},\ldots,z_{q_{m-1}})  \\
\times \prod_{i=1}^n A^{-2m+1}_{k_i,\bullet}
\prod_{1 \le i < j \le n} U_{k_i,k_j} 
S_{h}(\mathbf{k}_n^{\mathbf{p}_m,\mathbf{q}_{m-1}}). 
\end{multline*}
 
It will be convenient to extend the operator product $\prod_{i \not= p_s,q_s}  V_{k_i}(z_{q_1},\ldots,z_{q_{m-1}})$ also over all $i=p_s$ and for this purpose, we need the following lemma.
 
\begin{lemma} 
\label{lem:independent}
 The actions of 
 $$
 \left. U_{k_i,k^L_{q_r}} U_{k_i,k^R_{q_r}} \sgt_{h}(k^L_{k_{q_r}},k^R_{k_{q_r}}) \right|_{k^L_{q_r}=k^R_{q_r}} 
 \quad  \text{and} \quad 
 \left. U_{k^L_{q_r},k_i} U_{k^R_{q_r},k_i} \sgt_{h}(k^L_{k_{q_r}},k^R_{k_{q_r}}) \right|_{k^L_{q_r}=k^R_{q_r}} 
 $$
 as operators in $k_i$ are the identity if applied to functions independent of $k_i$.
\end{lemma}

\begin{proof} 
 We prove it for the first operator, since the proof is analogous for the second operator.
 Let again 
 $$
 U(x,y) = (1+x+x y)^{-1} A(x,y).
  $$
 Then we need to show that the constant term of  
 $$
 u(x,z)\coloneq \left. U(x,y) U(x,\iota(y)) \right|_{-y-\iota(y)=y \cdot \iota(y) = z}
 $$
as a formal power series in $x$ and $z$ is $1$. 
By the second condition \eqref{2} for annihilating formal power series, we have 
$u(x,z) u(\iota(x),z)=1$. It is not difficult to see that a formal power series $v(x)$ with 
$v(x) v(\iota(x))=1$ has constant term $\pm 1$, and since the constant term of $A(x,y)$ and 
$(1+x+x y)^{-1}$ is $1$ as well, the lemma is proved. 
 \end{proof} 
 
In total, we obtain
\begin{equation*}
(-1)^{\sum_{i=1}^m p_i + \sum_{i=1}^{m-1} q_i +1} \langle z_{q_1}^{h} z_{q_2}^{h} \dots z_{q_{m-1}}^{h} \rangle 
\prod_{i \not= q_s}  V_{k_i}(z_{q_1},\ldots,z_{q_{m-1}})  
\prod_{i=1}^n A^{-2m+1}_{k_i,\bullet}
\prod_{1 \le i < j \le n} U_{k_i,k_j} 
 S_{h}(\mathbf{k}_n^{\mathbf{p}_m,\mathbf{q}_{m-1}})
\end{equation*} 
for \eqref{interpret} in Lemma~\ref{pq}. 
Now we sum this expression for given $q_1,q_2,\ldots,q_{m-1}$ with $1 < q_1 < q_2 < \ldots < q_{m-1} < n$ over all 
 $p_1,\ldots,p_m$ with $1 \le p_1 < q_1 < p_2 < q_2 < \ldots < q_{m-1} < p_m \le n$. 
  
 We introduce the following useful notation:  Let $1 \le a \le b \le n$ be integers and $Q \subseteq [n]$, then we define 
 a $\{0,1\}$-vector of length $n-|Q|$ as
 $$\left\{ a \atop b\right\}_{n,Q} \coloneq ([a \le i \le b])_{i \in [n] \setminus Q},$$
 where we have used the Iverson bracket in ``$[a \le i \le b]$''. We set $q_0=0$ and $q_m=n+1$. Then, by the Laplace expansion for Pfaffians \eqref{expansion}, the sum over the $p_i$ can be written as 
 \begin{multline} 
 \label{laplace}
 (-1)^{\sum_{i=1}^{m-1} q_i +m+1} 
 \langle z_{q_1}^{h} z_{q_2}^{h} \dots z_{q_{m-1}}^{h} \rangle 
 \prod_{i \not= q_s}  V_{k_i}(z_{q_1},\ldots,z_{q_{m-1}})  
\prod_{i=1}^n A^{-2m+1}_{k_i,\bullet} \\
\prod_{1 \le i < j \le n} U_{k_i,k_j}
\pf \left( 
\begin{array}{c|c}  \TGT_{h}(\mathbf{k}_n^{\mathbf{q}_{m-1}}) & 
\begin{array}{c|c}
\RGT_{h,h-1}(\mathbf{k}_n^{\mathbf{q}_{m-1}}) 
& \left( \left\{q_{j-1}+1 \atop q_j-1 \right\}_{n,Q} \right)_{1 \le j \le m} 
\end{array}  \\ \hline 
& \T(0)_{h-1+m}  \end{array} \right)
\end{multline}
with $Q=\{q_1,\ldots,q_{m-1}\}$.  

Observe that, for $m=1$, the Pfaffian is just $\overline \GT_h(\mathbf{k}_n)$. Thus, Theorem~\ref{main} is proven as soon as we have shown the following two things: (i) the expression vanishes for fixed $q_1,\ldots,q_{m-1}$ when $m>1$ and (ii) Lemma~\ref{one}.

\begin{lemma} 
\label{one}
Suppose $h$ is a non-negative integer and $n$ is a positive integer. Then 
$$
\prod_{i=1}^n A^{-1}_{k_i,\bullet} \overline \GT_h(\mathbf{k}_n) = \overline \GT_h(\mathbf{k}_n).
$$
\end{lemma} 

As for (i), it suffices to show the vanishing of 
\begin{multline*} 
\langle z_{q_1}^{h} z_{q_2}^{h} \dots z_{q_{m-1}}^{h} \rangle \prod_{i \not=q_s}  V_{k_i}(z_{q_1},\ldots,z_{q_{m-1}})  \\
\pf \left( 
\begin{array}{c|c}  \TGT_{h}(\mathbf{k}_n^{\mathbf{q}_{m-1}}) & 
\begin{array}{c|c}
\RGT_{h,h-1}(\mathbf{k}_n^{\mathbf{q}_{m-1}}) 
& \left( \left\{q_{j-1}+1 \atop q_j-1 \right\}_{n,Q} \right)_{1 \le j \le m} 
\end{array}  \\ \hline 
& \T(0)_{h-1+m}  \end{array} \right)
\end{multline*}
if $m > 1$, since the application of $\prod_{i=1}^n A^{-2m+1}_{k_i,\bullet} 
\prod_{1 \le i < j \le n} U_{k_i,k_j}$ to the zero polynomial is the zero polynomial.

We argue why it suffices to show this assertion for $m=2$. First, it clearly suffices to show the vanishing when only one $q_i$ ``appears'' in  $V_{k_i}(z_{q_1},\ldots,z_{q_{m-1}})$, so it suffices to show the vanishing of 
$$
\langle z_{q_1}^{h}  \rangle \prod_{i \not= q_s}  V_{k_i}(z_{q_1})  
\pf_{1 \le i < j \le n+h} \left( 
\begin{array}{c|c}  \TGT_{h}(\mathbf{k}_n^{\mathbf{q}_{m-1}}) & 
\begin{array}{c|c}
\RGT_{h,h-1}(\mathbf{k}_n^{\mathbf{q}_{m-1}}) 
& \left( \left\{q_{j-1}+1 \atop q_j-1 \right\}_{n,Q} \right)_{1 \le j \le m} 
\end{array}  \\ \hline 
& \T(0)_{h-1+m}  \end{array} \right).
$$
Now adding the last $m-2$ columns to the $(m-1)$-st column from the right and then expanding along the last $m-2$ columns, one after the other, gives for each summand an instance of a case when $m=2$, with those $k_i$ removed that are considered in the expansion. Thus, we need to show the following lemma.

\begin{lemma} 
\label{zero}
Suppose $h,n$ are positive integers with $n \ge h$ such that $n+h$ is even and $1 \le q \le n$. Then 
$$ 
\langle {z^h_q} \rangle \prod_{i \not= q}  V_{k_i}(z_q)  
\pf_{1 \le i < j \le n+h} \left( 
\begin{array}{c|c}  \TGT_{h}(\mathbf{k}_n^{q}) & 
\begin{array}{c|c|c}
\RGT_{h,h-1}(\mathbf{k}_n^{q}) & \left\{1 \atop q-1 \right\}_{n,\{q\}}  & \left\{q+1 \atop n \right\}_{n,\{q\}}
\end{array} 
   \\ \hline 
&  \T(0)_{h+1} \end{array} \right)=0.
$$
\end{lemma}

\section{Proofs of Lemmas~\texorpdfstring{\ref{one}}{\ref*{one}} and \texorpdfstring{\ref{zero}}{\ref*{zero}}} 
\label{sec:second}

\subsection{Proof of Lemma~\ref{one}}

We prove the more general Lemma~\ref{fund} first. To this end, we use Problems~6 and 8 from \cite{dominolozenge}, which are summarized in Lemma~\ref{urbanrenewal} below. We need a few preparations to state it.

The statement of the lemma involves the elementary symmetric polynomials, which are defined as 
$$
e_p(X_1,\ldots,X_n) \coloneq \sum_{1 \le i_1 < i_2 < \ldots < i_p \le n} X_{i_1} X_{i_2} \dots X_{i_p}.
$$
In accordance with our conventions, we use the following abbreviations:
$$
e_p(\mathbf{X}_n) \coloneq e_p(X_1,\ldots,X_n), e_p(\e_{\mathbf{k}_n}) \coloneq e_p(\e_{k_1},\ldots,\e_{k_n}), 
e_p(\Delta_{\mathbf{k}_n}) \coloneq e_p(\Delta_{k_1},\ldots,\Delta_{k_n}).
$$

In the second part of the lemma we provide an alternative description of 
\begin{equation}
\label{eCart}  
e_p(\e_{k_1},\ldots,\e_{k_n}) [k_1,k_2-1] \times \cdots \times [k_{n-1},k_n-1],
\end{equation} 
which is understood as a multiset of elements in $\mathbb{Z}^{n-1}$.  A \Dfn{multiset} is 
a set together with an integer-valued function on the set that indicates the multiplicity of each element. In \eqref{eCart},  
$[k_1,k_2-1] \times \cdots \times [k_{n-1},k_n-1]$ stands for the Cartesian product of integer intervals, where each element has multiplicity $1$. When $e_p(\e_{k_1},\ldots,\e_{k_n})$ is applied to the Cartesian product, the ``$+$'' is interpreted as a disjoint union $\bigcupdot$, resulting, in general, in a multiset.

\begin{lemma} 
\label{urbanrenewal}
Let $n \ge 2$ and $p$ be non-negative integers.
\begin{enumerate}
\item\label{urbanrenewal:1} Suppose $ a(\mathbf{l}_{n-1})$ is a function on $\mathbb{Z}^{n-1}$ which vanishes for all $\mathbf{l}_{n-1} \in \mathbb{Z}^{n-1}$ for 
which there exists an $i \in [2,n-1]$
with $l_{i-1}=l_{i}$.  Then 
\begin{equation}
\label{identity}  
e_p(\e_{\mathbf{k}_n}) \sum_{\mathbf{l}_{n-1} \prec \mathbf{k}_{n}}
 a(\mathbf{l}_{n-1}) \\
= 
\sum_{\mathbf{l}_{n-1} \prec \mathbf{k}_{n}} e_p(\e_{\mathbf{l}_n}) 
a(\mathbf{l}_{n-1}).
\end{equation} 
\item\label{urbanrenewal:2} We have 
\begin{equation} 
\label{Icart}
e_p(\e_{\mathbf{k}_n}) \cart_{i=1}^{n-1} [k_i,k_{i+1}-1]  
= 
\bigcupdot_{I \subseteq [2,n-1] \atop I \cap (I-1) = \emptyset} \bigcupdot_{{\mathbf e} \in \binom{[n] \setminus (I \cup (I-1))}{p-|I|}}
\cart_{i=1}^{n-1} 
\begin{cases}  
\{k_i\}, & i \in I, \\
\{k_{i+1}\}, & i+1 \in I, \\
[k_i, k_{i+1}-1], & i \notin I \cup (I-1), i \notin {\mathbf e}, \\
[k_i+1, k_{i+1}], & i \notin I \cup (I-1), i \in {\mathbf e}, 
\end{cases}
\end{equation} 
where $I-1$ is obtained from $I$ by subtracting $1$ from each element in $I$, and, for any set $S$, $\binom{S}{p}$ is the set of all (ordinary) subsets of $S$ with cardinality $p$.
\end{enumerate} 
\end{lemma}

Note that in the first assertion, we can also replace $e_p(\e_{\mathbf{l}_n})$ with 
$e_{p-1}(\e_{\mathbf{l}_{n-1}})+e_p(\e_{\mathbf{l}_{n-1}})$ since $a(\mathbf{l}_{n-1})$ is independent of $l_n$.

\begin{proof}
Note that the first part is just \cite[Problem~6]{dominolozenge}, while the second part is \cite[Problem~8]{dominolozenge}. However, note also that 
the first assertion is an immediate consequence of the second as 
$$
\sum_{\mathbf{l}_{n-1} \prec \mathbf{k}_{n}}
 a(\mathbf{l}_{n-1})  = 
\sum_{\mathbf{l}_{n-1} \in \cart_{i=1}^{n-1} [k_i,k_{i+1}-1]} a(\mathbf{l}_{n-1})
$$
and the assumption on $a(\mathbf{l}_{n-1})$ implies that we can assume that $I$ is empty since, for non-empty sets $I$, the Cartesian product on the 
right-hand side of \eqref{Icart} contains only elements $\mathbf{l}_{n-1}$ for which there exists an $i \in [2,n-1]$
with $l_{i-1}=l_{i}$ and thus $a(\mathbf{l}_{n-1})=0$ by assumption.
\end{proof} 

Now we are in the position to state the lemma that will imply Lemma~\ref{one}.

\begin{lemma} 
\label{fund}
Suppose $h,n$ are non-negative integers with $h \le n$ and $S(\mathbf{X}_n)$ is a symmetric polynomial in $X_1,\ldots,X_n$.
Then 
\begin{equation}
\label{S} 
S(\Delta_{\mathbf{k}_n}) \overline \GT_h(\mathbf{k}_n) = S(\iota(\Delta_{\mathbf{k}_n}))  \overline \GT_h(\mathbf{k}_n).
\end{equation}
\end{lemma} 

\begin{proof}
Recalling the definition of $\iota(x)$ in \eqref{iota}, we note that 
\begin{equation}
\label{Deltadelta}  
\iota(\Delta_x)=- \e_{x}^{-1} \Delta_x = \e^{-1}_{x} - \id, 
\end{equation} 
which is just the backward difference operator~$\delta_x$. 
We extend this notation also to vectors as we did for the (forward) difference operator and the shift operator in \eqref{vectors}.

Next we argue why it suffices to show the assertion when $S(\mathbf{X}_n)$ is an elementary symmetric polynomial. As every symmetric polynomial is a polynomial in the elementary symmetric polynomials, it suffices to show it for a product of elementary symmetric polynomials by a 
linearity argument. Proving it for elementary symmetric polynomials just means that the difference 
\begin{equation}
\label{annh}  
e_p(\Delta_{\mathbf{k}_n}) - e_p(\delta_{\mathbf{k}_n}) 
\end{equation} 
is in the annihilator of $\overline \GT_h(\mathbf{k}_n)$ for any $p \ge 0$. Using the fact that   
$$e_p(\Delta_{\mathbf{k}_n}),e_q(\Delta_{\mathbf{k}_n}), 
e_p(\delta_{\mathbf{k}_n}),e_q(\delta_{\mathbf{k}_n})$$
all commute, we will show that 
$$
\prod_{i=1}^m e_{p_i}(\Delta_{\mathbf{k}_n}) - \prod_{i=1}^m e_{p_i}(\delta_{\mathbf{k}_n})
$$
is in the ideal generated by \eqref{annh} for any sequence of non-negative integers $p_1,\ldots,p_m$. Indeed, we may gradually change the arguments of individual $e_{p_i}$'s from $\Delta_{\mathbf{k}_n}$ to $\delta_{\mathbf{k}_n}$. More concretely, we observe that 
\begin{multline*} 
\prod_{i=1}^{l-1} e_{p_i}(\delta_{\mathbf{k}_n})
\prod_{i=l}^m e_{p_i}(\Delta_{\mathbf{k}_n}) -
\prod_{i=1}^{l} e_{p_i}(\delta_{\mathbf{k}_n})
\prod_{i=l+1}^m e_{p_i}(\Delta_{\mathbf{k}_n}) \\
= \left( e_{p_l}(\Delta_{\mathbf{k}_n}) - e_{p_l}(\delta_{\mathbf{k}_n}) \right) 
\prod_{i=1}^{l-1} e_{p_i}(\delta_{\mathbf{k}_n}) \prod_{i=l+1}^m e_{p_i}(\Delta_{\mathbf{k}_n}) 
\end{multline*} 
is in the ideal generated by \eqref{annh} for all $1 \le l \le m$. Thus, 
$$
\sum_{l=1}^m \left[ \prod_{i=1}^{l-1} e_{p_i}(\delta_{\mathbf{k}_n})
\prod_{i=l}^m e_{p_i}(\Delta_{\mathbf{k}_n}) -
\prod_{i=1}^{l} e_{p_i}(\delta_{\mathbf{k}_n})
\prod_{i=l+1}^m e_{p_i}(\Delta_{\mathbf{k}_n})  \right] \\
= \prod_{i=1}^m e_{p_i}(\Delta_{\mathbf{k}_n}) - \prod_{i=1}^m e_{p_i}(\delta_{\mathbf{k}_n})
$$
is in this ideal as well.

To show that \eqref{annh} is in the annihilator of $\overline \GT_h(\mathbf{k}_n)$, first note that 
 $e_p(\Delta_{\mathbf{k}_n})$ can be written as linear combinations of 
$e_q(\e_{\mathbf{k}_n})$'s. Thus, by \eqref{identity},
\begin{equation} 
\label{forward} 
e_p(\Delta_{\mathbf{k}_n}) \sum_{\mathbf{l}_{n-1} \prec \mathbf{k}_{n}}
 a(\mathbf{l}_{n-1}) \\
= 
\sum_{\mathbf{l}_{n-1} \prec \mathbf{k}_{n}} e_p(\Delta_{\mathbf{l}_n}) 
a(\mathbf{l}_{n-1})
\end{equation}
as long as $a(\mathbf{l}_{n-1})$ satisfies the assumption from Lemma~\ref{urbanrenewal} \eqref{urbanrenewal:1}.
We can also conclude that \eqref{identity}  holds true when  $e_p(\e_{\mathbf{k}_n})$ is replaced by $e_p(\e^{-1}_{\mathbf{k}_n})$
and $e_p(\e_{\mathbf{l}_n})$ is replaced by $e_p(\e^{-1}_{\mathbf{l}_n})$. This follows from 
$$
e_p(\mathbf{X}^{-1}_n) = (X_1 \dots X_n)^{-1} e_{n-p}(\mathbf{X}_n)
$$
and 
\begin{equation}
\label{trans} 
\e_{\mathbf{k}_n}^{-1}  \sum_{\mathbf{l}_{n-1} \prec \mathbf{k}_{n}}  a(\mathbf{l}_{n-1}) = 
 \sum_{\mathbf{l}_{n-1} \prec \mathbf{k}_{n}}    \e^{-1}_{\mathbf{l}_{n}} a(\mathbf{l}_{n-1}).
\end{equation} 
Since $e_p(\delta_{\mathbf{k}_n})$ can be expressed as linear combinations of $e_q(\e^{-1}_{\mathbf{k}_n})$'s, it also follows that
\begin{equation} 
\label{backward} 
e_p(\delta_{\mathbf{k}_n}) \sum_{\mathbf{l}_{n-1} \prec \mathbf{k}_{n}}
 a(\mathbf{l}_{n-1}) \\
= 
\sum_{\mathbf{l}_{n-1} \prec \mathbf{k}_{n}} e_p(\delta_{\mathbf{l}_n}) 
a(\mathbf{l}_{n-1}).
\end{equation}

The recursive relation 
\begin{equation}
\label{strictrec}
\overline \GT_h(\mathbf{k}_n) = \sum_{\mathbf{l}_{n-1} \prec \mathbf{k}_{n}} \overline \GT_{h-1}(\mathbf{l}_{n-1}), 
\end{equation}
implies that
$\overline \GT_h(\mathbf{k}_n)$
vanishes if $k_{i-1}=k_i$ for some $i \in \{2,\ldots,n\}$, provided $h \ge 1$. 

We can now prove 
\begin{equation}
\label{ideal} 
e_p(\Delta_{\mathbf{k}_n}) \overline \GT_h(\mathbf{k}_n) = e_p(\delta_{\mathbf{k}_n}) \overline \GT_h(\mathbf{k}_n)
\end{equation} 
by induction with respect to $h$. For $h \ge 2$, we have 
$$ 
e_p(\Delta_{\mathbf{k}_n}) \overline \GT_h(\mathbf{k}_n) = 
\sum_{\mathbf{l}_{n-1} \prec \mathbf{k}_{n}} e_p(\Delta_{\mathbf{l}_n}) \overline \GT_{h-1}(\mathbf{l}_{n-1}) 
= \sum_{\mathbf{l}_{n-1} \prec \mathbf{k}_{n}} e_p(\delta_{\mathbf{l}_n}) \overline \GT_{h-1}(\mathbf{l}_{n-1}) 
= e_p(\delta_{\mathbf{k}_n}) \overline \GT_h(\mathbf{k}_n) 
$$
by \eqref{forward} and \eqref{backward} and the induction hypothesis, which is used in the second step.
Thus we have reduced the problem to proving \eqref{ideal} for $h=0,1$.

For $h=0$, the result is trivial. We consider the case $h=1$. As there are $c_i$ satisfying
$$
e_p(\Delta_{\mathbf{k}_n}) = \sum_{q=0}^p c_q e_q(\e_{\mathbf{k}_n}) 
\quad \text{and} \quad 
e_p(\delta_{\mathbf{k}_n}) = \sum_{q=0}^p c_q e_q(\e^{-1}_{\mathbf{k}_n}), 
$$
it suffices to show 
$$
e_p(\e_{\mathbf{k}_n}) \overline \GT_1(\mathbf{k}_n) = e_p(\e^{-1}_{\mathbf{k}_n}) \overline \GT_1(\mathbf{k}_n)
$$
and thus 
\begin{equation}
\label{np} 
e_p(\e_{\mathbf{k}_n}) \overline \GT_1(\mathbf{k}_n) = e_{n-p}(\e_{\mathbf{k}_n})  \overline \GT_1(\mathbf{k}_n),
\end{equation} 
where we have also used \eqref{trans}. Now \eqref{np} can be written as
$$
\sum_{\mathbf{l}_{n-1} \in e_p(\e_{\mathbf{k}_n}) \cart\limits_{i=1}^{n-1} [k_i,k_{i+1}-1]} 1 = 
\sum_{\mathbf{l}_{n-1} \in e_{n-p}(\e_{\mathbf{k}_n}) \cart\limits_{i=1}^{n-1} [k_i,k_{i+1}-1]} 1.
$$
We apply Lemma~\ref{urbanrenewal} \eqref{urbanrenewal:2} to both sides. In fact, we claim that we even have the same polynomials in $\mathbf{k}_n$ when fixing $I \subseteq [2,n-1]$
on both sides. 
Indeed, it does not matter whether or not we sum over the interval $[k_i,k_{i+1}-1]$ or $[k_i+1,k_{i+1}]$ in the third or fourth case. Therefore, it suffices to show 
that the cardinalities of $\binom{[n] \setminus (I \cup (I-1))}{p-|I|}$ and  $\binom{[n] \setminus (I \cup (I-1))}{n-p-|I|}$ agree. As $I$ and $I-1$ are disjoint, we have 
$\binom{n-2|I|}{p-|I|}$ in the first case and $\binom{n-2|I|}{n-p-|I|}$ in the second case, which is the same by the symmetry of the binomial coefficient.
\end{proof}

We introduce the following notation: Suppose $C(x)$ is a univariate function, then we write $C(\mathbf{k}_n)$ for $\prod_{i=1}^n C(k_i)$. In the subsequent proof, this is for instance applied to $B(\Delta_x)$, where $B(x)$ is a formal power series in $x$. 

\begin{proof}[Proof of Lemma~\ref{one}]
Note that the operator $A_{k_i,\bullet}$  from the statement of Lemma~\ref{one} is just $A(\Delta_{k_i},0)$. Setting $x_2=0$ in \eqref{equation}, we see that 
$$
A(x,0)^2A(\iota(x),0)^2=1, 
$$
which implies that $A(x,0) A(\iota(x),0)=\pm 1$. Since the constant term of $A(x,y)$ is $1$ by $A(x,\iota(x))=1$, it follows that 
$A(x,0) A(\iota(x),0)=1$. Thus, it suffices to show 
\begin{equation}
\label{B}  
 B(\Delta_{\mathbf{k}_n}) \overline \GT_h(\mathbf{k}_n) = \overline \GT_h(\mathbf{k}_n)
\end{equation} 
for all formal power series with $B(x) B(\iota(x))= 1$ and constant term $1$. We can take the square root $C(x)$ of $B(x)$ as follows:
$$
C(x) = \sum_{i \ge 0} \binom{1/2}{i} (B(x)-1)^i.
$$
Note that also $C(x) C(\iota(x))= 1$. 
Thus, showing 
$$
B(\Delta_{\mathbf{k}_n}) \overline \GT_h(\mathbf{k}_n) = \overline \GT_h(\mathbf{k}_n)
$$
is equivalent to showing 
$$
C(\Delta_{\mathbf{k}_n}) \overline \GT_h(\mathbf{k}_n) = C(\iota(\Delta_{\mathbf{k}_n}))  \overline \GT_h(\mathbf{k}_n)
$$ as 
$B(x) = C(x)^2 = C(x) C(\iota(x))^{-1}$. The assertion follows from Lemma~\ref{fund} by choosing 
$
S(\mathbf{X}_n) =C(\mathbf{X}_n)
$
there.
\end{proof} 

\subsection{Towards proving Lemma~\ref{zero}}
\label{towards} 

Lemma~\ref{zero} can be restated in terms of $\overline \GT_h(\mathbf{k}_n)$ as follows: In the triangular array underlying the Pfaffian in the statement, 
we add the last column to the penultimate column and then expand the Pfaffian with respect to the last column, resulting in 
$
\sum_{i=q+1}^n (-1)^{i+n+h} \overline \GT_h(\mathbf{k}^{q,i}_n).
$
Setting 
\begin{equation} 
\label{pdef}
\overline{Q}(z_q,\mathbf{k}^q_{n-1}) \coloneq \prod_{i \not=q} V_{k_i}(z_q)  \overline{\GT}_h(\mathbf{k}^{q}_{n-1}), 
\end{equation}
it suffices to prove that the degree of 
$
 \sum_{i=q+1}^n (-1)^{i+n+h} \overline{Q}(z_q,\mathbf{k}^{q,i}_{n})
$
as a polynomial in $z_q$ is less than $h$.

In the following, we replace $q$ by the half integer $q-\frac{1}{2}$, which allows us to work with $\overline \GT_h(\mathbf{k}_n)$ instead of $\overline \GT_h(\mathbf{k}^q_{n})$. The operator $V_{k_i}(z_q)$ is adapted accordingly, and we also replace $z_q$ by $z$.

Recall that the operator $V_{k_i}(z)$ is defined through a formal power series $U(x,y)$ as follows:
$$
V_{k_i}(z) = 
\begin{cases} 
 \left. \e_{k_i}^2 U(\Delta_{k_i},y) U(\Delta_{k_i},\iota(y)) \right|_{-y-\iota(y)=y \cdot \iota(y)=z}, & i < q, \\
 \left. U(y,\Delta_{k_i}) U(\iota(y),\Delta_{k_i}) \right|_{-y-\iota(y)=y \cdot \iota(y) = z}, & i>q, 
\end{cases}
$$
where $U(x,y) = U_1(x,y) U_2(x,y)$ with
$$
U_1(x,y) = \frac{1+x+y}{1+x+x y}
\quad \text{and} \quad
U_2(x,y)= \frac{P(x,y)P(y,x)}{P(\iota(x),y)P(\iota(y),x)}
$$
satisfying
$$
\frac{P(x,\iota(x))P(\iota(x),x)}{P(x,x)P(\iota(x),\iota(x))} = (1+x+\iota(x))^{-1},
$$
compare with \eqref{defv} and the proof of Proposition~\ref{hiddencharacterize} \eqref{hiddencharacterize:2}. We decompose $V_{k_i}(z)$ into 
$$
V_{1,k_i}(z) = 
\begin{cases} 
 \left. \e_{k_i}^2 U_1(\Delta_{k_i},y) U_1(\Delta_{k_i},\iota(y)) \right|_{-y-\iota(y)=y \cdot \iota(y)=z}, & i < q, \\
 \left. U_1(y,\Delta_{k_i}) U_1(\iota(y),\Delta_{k_i}) \right|_{-y-\iota(y)=y \cdot \iota(y) = z}, & i>q, 
\end{cases}
$$
and 
$$
V_{2,k_i}(z) = 
\begin{cases} 
 \left. U_2(\Delta_{k_i},y) U_2(\Delta_{k_i},\iota(y)) \right|_{-y-\iota(y)=y \cdot \iota(y)=z}, & i < q, \\
 \left. U_2(y,\Delta_{k_i}) U_2(\iota(y),\Delta_{k_i}) \right|_{-y-\iota(y)=y \cdot \iota(y) = z}, & i>q, 
\end{cases} 
$$
so that $V_{k_i}(z)=V_{1,k_i}(z) V_{2,k_i}(z)$. By the symmetry of $U_2(x,y)$, we have 
$$
V_{2,k_i}(z) = \left. U_2(\Delta_{k_i},y) U_2(\Delta_{k_i},\iota(y)) \right|_{-y-\iota(y)=y \cdot \iota(y)=z}.
$$
Moreover, since 
$$
U_2(x,y) U_2(\iota(x),y) U_2(x,\iota(y)) U_2(\iota(x),\iota(y))=1,
$$
it follows that 
\begin{multline*} 
\left.
U_2(\Delta_{k_i},y) U(\Delta_{k_i},\iota(y)) \right|_{-y-\iota(y)=y \cdot \iota(y)=z} 
=
\left. U_2(\iota(\Delta_{k_i}),y)^{-1} U_2(\iota(\Delta_{k_i}),\iota(y))^{-1}  \right|_{-y-\iota(y)=y \cdot \iota(y)=z}  \\
= \left. U_2(\delta_{k_i},y)^{-1} U_2(\delta_{k_i},\iota(y))^{-1} \right|_{-y-\iota(y)=y \cdot \iota(y)=z},
\end{multline*} 
where we have used \eqref{Deltadelta} in the last step.
Using \eqref{B}, we can conclude that 
$$
V_{2,k_i}(z) \overline \GT_h(\mathbf{k}_n) = \overline \GT_h(\mathbf{k}_n),
$$
and, therefore, 
$$
\prod_{i=1}^n V_{k_i}(z)  \overline \GT_h(\mathbf{k}_{n}) = \prod_{i=1}^n V_{1,k_i}(z)  \overline \GT_h(\mathbf{k}_{n}).
$$
As for $V_{1,k_i}(z)$, note that 
\begin{align*} 
\left. \e_{k_i}^2 U_1(\Delta_{k_i},y) U_1(\Delta_{k_i},\iota(y)) \right|_{-y-\iota(y)=y \cdot \iota(y)=z} &=\e_{k_i}^2, \\ 
\left. U_1(y,\Delta_{k_i}) U_1(\iota(y),\Delta_{k_i}) \right|_{-y-\iota(y)=y \cdot \iota(y) = z} &= \e_{k_i}^2 
\frac{1 + \e_{k_i}^{-1} \delta_{k_i} z}{1+\e_{k_i} \Delta_{k_i} z}.
\end{align*} 
Using $\prod_{i=1}^n \e^2_{k_i} \overline{\GT}_h(\mathbf{k}_{n}) = \overline{\GT}_h(\mathbf{k}_{n})$, and letting 
$$
\widehat{V}_{k_i}(z) \coloneq
\begin{cases} 
1, & i<q, \\
\frac{1+ p(\delta_{k_i}) z}{1+ p(\Delta_{k_i}) z}, & i>q.
\end{cases},
$$
where $p(x)=x(x+1)$, we may now consider 
$$
Q(z,\mathbf{k}_{n}) \coloneq \prod_{i=1}^n \widehat{V}_{k_i}(z)  \overline \GT_h(\mathbf{k}_{n})
$$
instead of \eqref{pdef}. We need to prove  that the degree of $z$ in $\sum_{i>q} (-1)^{i} Q(z,\mathbf{k}^i_{n+1})$ is less than $h$. In Lemma~\ref{signreversing}, we accomplish this for any polynomial $p(x)$ without constant term.

\subsection{An additive decomposition \texorpdfstring{of $\overline \GT_h(\mathbf{k}_n)$}{}}
\label{additive}

To prepare for the proof of Lemma~\ref{signreversing}, 
we need to consider an additive decomposition of $\overline \GT_h(\mathbf{k}_n)$ that depends on the half integer $q$. Each summand of the additive decomposition
factors into two polynomials, where one factor depends only on $k_1,k_2,\ldots,k_{q-1/2}$, while the other factor depends only on  $k_{q+1/2},\ldots,k_n$. This comes in handy as the definition of the operator $\widehat{V}_{k_i}(z)$ is homogeneous for $i<q$ and also homogeneous (but of a different nature) for $i>q$.

In order to present the additive decomposition of $\overline \GT_h(\mathbf{k}_n)$, it is useful to keep the following conceptual understanding of why $\overline \GT_h(\mathbf{k}_n)$ is a polynomial 
in $k_1,\ldots,k_n$ in mind: For any polynomial $p(l)$ in $l$ and any integer $a$, 
the sum $\sum_{l=a}^k p(l)$ can be represented by a unique polynomial $q(k)$ in $k$ which coincides with the sum for all $k \ge a-1$ and where 
\begin{equation}\label{inc}
	\deg_{k} q(k) = \deg_l p(l) + 1.
\end{equation} 
One can see that this holds true for all integers~$k$ if we use the following extended summation operator:
$$
\sum_{l=a}^b p(l) = \begin{cases} \sum\limits_{l=a}^b p(l), & a \le b, \\ - \sum\limits_{l=b+1}^{a-1} p(l), & b+1 \le a-1, \\ 0,  & b = a-1. \end{cases}
$$
The analogous assertion is true for $\sum_{l=k}^a p(l)$.

Using \eqref{strictrec}, this implies by induction with respect to $h$ that the number $\overline \GT_h(\mathbf{k}_n)$ of $(h,n)$-Gelfand--Tsetlin patterns with bottom row $\mathbf{k}_n$ with strictly increasing $\searrow$-diagonals is expressible by a polynomial in $\mathbf{k}_n$. Interestingly, the degree bound we obtain this way for 
$k_i$ using \eqref{inc} grows exponentially with $h$, while the actual degree is $2 h$, see Theorem~\ref{GT}.

We introduce the above-mentioned additive decomposition of $\overline \GT_h(\mathbf{k}_n)$: It depends on a half-integer $q$ with $1 \le q \le n$ and involves a new set of variables $\mathbf{m}_h=(m_1,\ldots,m_h)$. The definition is recursively and the $2^h$ polynomials are encoded by a $\{0,1\}$-sequence ${\mathbf b}_h=(b_1,\ldots,b_h)$ of length~$h$. The polynomial associated with $\mathbf{b}_h$ will be denoted by $\overline \GT_{h,q,\mathbf{b}_h}(\mathbf{k}_n;\mathbf{m}_h)$. 

The rationale behind the definition of these polynomials is as follows: Let $i \in [n-1]$ such that $i < q < i+1$. By definition, the integer $l_i$ of the Gelfand--Tsetlin array that is situated between $k_i$ and $k_{i+1}$ in the row above has to be in the interval $[k_i,k_{i+1}-1]$. We dissect the interval into $[k_i,m_{h}-1]$ and $[m_{h},k_{i+1}-1]$ (and thus it makes sense to assume $k_i \le m_h \le k_{i+1}$ although all of this also works eventually if 
this is not the case by using the extended summation operator) and we choose $l_i \in [k_i,m_{h}-1]$ if $b_h=0$, and $l_i \in [m_{h},k_{i+1}-1]$ otherwise. 

The definition is now recursively: When deleting the bottom row of the $(h,n)$-GT trapezoid and we are in the first case ($b_h=0$), we obtain an $(h-1,n-1)$-GT trapezoid with strictly increasing $\searrow$-diagonals that is associated with 
$q,{\mathbf b}_{h-1}, {\mathbf m}_{h-1}$; in the second case ($b_h=1$), we obtain an $(h-1,n-1)$-GT trapezoid with strictly increasing $\searrow$-diagonals that is associated with 
$q-1,{\mathbf b}_{h-1}, {\mathbf m}_{h-1}$.

Formulated in terms of recursions, this means the following: If $b_h=0$, we have  
\begin{equation}
\label{dec1}
\overline \GT_{h,q,\mathbf{b}_h}(\mathbf{k}_n;\mathbf{m}_h) = 
\sum_{l_i=k_i}^{m_h-1} \sum_{\mathbf{l}_{1,i-1} \prec \mathbf{k}_{1,i} \atop \mathbf{l}_{i+1,n-1} \prec \mathbf{k}_{i+1,n}}
\overline \GT_{h-1,q,\mathbf{b}_{h-1}}(\mathbf{l}_{n-1};\mathbf{m}_{h-1}),
\end{equation} 
while if $b_h=1$, we have 
\begin{equation}
\label{dec2}
\overline \GT_{h,q,\mathbf{b}_h}(\mathbf{k}_n;\mathbf{m}_h) = 
\sum_{l_i=m_h}^{k_{i+1}-1} \sum_{\mathbf{l}_{1,i-1} \prec \mathbf{k}_{1,i} \atop \mathbf{l}_{i+1,n-1} \prec \mathbf{k}_{i+1,n}}
\overline \GT_{h-1,q-1,\mathbf{b}_{h-1}}(\mathbf{l}_{n-1};\mathbf{m}_{h-1}).
\end{equation}
For a fixed half-integer $q$ with $1 \le q \le n$, it follows by induction with respect to $h$ that 
\begin{equation} 
\label{decomposition} 
\sum_{{\mathbf b}_h \in \{0,1\}^h} \overline \GT_{h,q,\mathbf{b}_h}(\mathbf{k}_n;\mathbf{m}_h) =  
\overline \GT_h(\mathbf{k}_n).
\end{equation}

In the next proposition, we show that 
$\overline \GT_{h,q,\mathbf{b}_h}(\mathbf{k}_n;\mathbf{m}_h)$ can 
be expressed in terms of the polynomials $\overline \GT_h(\mathbf{k}^\prime_{n^\prime})$ for certain choices of $n^\prime$ and $\mathbf{k}^\prime_{n^\prime}$. In particular, it is shown that $\overline \GT_{h,q,\mathbf{b}_h}(\mathbf{k}_n;\mathbf{m}_h)$ factorizes into two polynomials, one of which is a polynomial in $k_1,\ldots,k_i$ and $m_{x_1},\ldots,m_{x_r}$ and the other is a polynomial in $k_{i+1},\ldots,k_n$ and $m_{y_1}, \ldots,m_{y_{s}}$, for some $x_j,y_j$. 

We extend the notation of $\mathbf{k}_{a,b}$ as follows: Let $X=\{x_1 < \ldots < x_m\}$ be a set of integers, then 
$\mathbf{k}_X \coloneq (k_{x_1},k_{x_2},\ldots,k_{x_m})$.

\begin{prop} 
\label{factor}
Let $h,n$ be non-negative integers with $n \ge h$ and $q$ a half-integer with $1 \le q \le n$. Setting 
$i = \lfloor q \rfloor$ and letting $h \ge x_1 > x_2 > \ldots > x_r \ge 1$ and $1 \le y_1 < y_2 < \ldots < y_s \le h$ such that, for all $1 \le l \le h$, 
$l \in X=\{x_1,\ldots,x_r\} $ if and only if $b_l=0$, and $l \in Y=\{y_1,\ldots,y_s\}$ if and only if $b_l=1$, we have  
$$
\overline \GT_{h,q,\mathbf{b}_h}(\mathbf{k}_n;\mathbf{m}_h) = 
\prod_{j=1}^r \Delta_{m_{x_j}}^{h-x_j} \overline \GT_h(\mathbf{k}_i,\mathbf{m}_{X}) \prod_{j=1}^s (-\Delta_{m_{y_j}})^{h-y_j} \overline \GT_h(\mathbf{m}_{Y},\mathbf{k}_{i+1,n}).
$$
\end{prop}

\begin{proof} 
The proof follows directly from the definition of $\overline \GT_{h,q,\mathbf{b}_h}(\mathbf{k}_n;\mathbf{m}_h)$ by induction with respect to $h$. The case $h=0$ is easy as both sides are equal to $1$. 

For $h \ge 1$, we distinguish between the cases $b_h=0$ and $b_h=1$. First we assume that  $b_h=0$ and thus $x_1=h$. By induction, we may assume that $\overline \GT_{h-1,q,\mathbf{b}_{h-1}}(\mathbf{l}_{n-1};\mathbf{m}_{h-1})$ equals
\begin{equation*} 
\prod_{j=2}^r \Delta_{m_{x_j}}^{h-1-x_j} \overline \GT_{h-1}(\mathbf{l}_i,\mathbf{m}_{X \setminus \{x_1\}}) 
\prod_{j=1}^s (-\Delta_{m_{y_j}})^{h-1-y_j} \overline \GT_{h-1}(\mathbf{m}_{Y},\mathbf{l}_{i+1,n-1}).
\end{equation*} 
Now we use \eqref{dec1} as well as Lemma~\ref{eat}.
The case $b_h=1$ is similar.
\end{proof}

\subsection{Proof of Lemma~\ref{zero}}
We need one more proposition to prepare for Lemma~\ref{signreversing}. The latter lemma will immediately imply Lemma~\ref{zero} as was discussed in Section~\ref{towards}.

\begin{prop}
\label{sumtozero} 
Let $n$ be a positive integer, let $a(\mathbf{l}_{n-1})$ be a polynomial in $\mathbf{l}_{n-1}$ and define 
$
b(\mathbf{k}_n) \coloneq \sum\limits_{\mathbf{l}_{n-1} \prec \mathbf{k}_n} a(\mathbf{l}_{n-1}).
$
Then 
$
\sum\limits_{i=1}^{n+1} (-1)^{i} b(\mathbf{k}^i_{n+1}) = 0.
$
\end{prop} 

\begin{proof}
We set 
$
A(\mathbf{l}_{n-1}) \coloneq \sum_{\mathbf{x}_{n-1} \prec (0,\mathbf{l}_{n-1})} a(\mathbf{x}_{n-1}).
$ 
Then, by Lemma~\ref{eat},  
$\Delta_{\mathbf{l}_{n-1}} A(\mathbf{l}_{n-1}) = a(\mathbf{l}_{n-1})$
and $A(\mathbf{l}_{n-1})$ vanishes whenever $l_{i-1}=l_i$ for some $i \in \{2,\ldots,n-1\}$. It follows by telescoping that 
$$ 
\sum_{\mathbf{l}_{n-1} \prec \mathbf{k}_n} a(\mathbf{l}_{n-1}) 
= \sum_{\mathbf{l}_{n-1} \prec \mathbf{k}_n}  \Delta_{\mathbf{l}_{n-1}} A(\mathbf{l}_{n-1}) \\
= \sum_{i=1}^n (-1)^{i+1} A(\mathbf{k}^i_{n}),
$$
from which one can deduce the assertion.
\end{proof} 

Assume that $p(y)$ is a univariate polynomial. 
By abuse of our notation, we use the following short hand notations 
$$
\frac{1+p(\delta_{\mathbf{k}_{i+1,n}}) z}{1+p(\Delta_{\mathbf{k}_{i+1,n}}) z} \coloneq \prod_{j=i+1}^n \frac{1+p(\delta_{k_j}) z}{1+p(\Delta_{k_j}) z} 
\quad \text{and} \quad 
\frac{1+p(\delta_{\mathbf{m}_h}) z}{1+p(\Delta_{\mathbf{m}_h}) z} \coloneq \prod_{j=1}^h \frac{1+p(\delta_{m_j}) z}{1+p(\Delta_{m_j}) z}
$$
as well as 
$$
\frac{1+p(\delta_{\mathbf{m}_X}) z}{1+p(\Delta_{\mathbf{m}_X}) z} \coloneq \prod_{j=1}^r \frac{1+p(\delta_{m_{X_j}}) z}{1+p(\Delta_{m_{X_j}}) z}
\quad \text{and} \quad
\frac{1+p(\delta_{\mathbf{m}_Y}) z}{1+p(\Delta_{\mathbf{m}_Y}) z} \coloneq \prod_{j=1}^s \frac{1+p(\delta_{m_{Y_j}}) z}{1+p(\Delta_{m_{Y_j}}) z}
$$
and also
$
V_{\mathbf{k}_n}(z_q) \coloneq \prod_{j=1}^n V_{k_j}(z_q).
$

\begin{lemma}
\label{signreversing}
Let $n,h$ be positive integers with $n \ge h$ such that $n+h$ is even, let $q$ be a half integer with $1 \le q \le n$, and let $p(y)$ be a polynomial in $y$ without constant term. Letting $i = \lfloor q \rfloor$, the degree in $z$ of 
$$
\frac{1+p(\delta_{\mathbf{k}_{i+1,n}}) z}{1+p(\Delta_{{\mathbf{k}_{i+1,n}}}) z} \overline{\GT}_{h}(\mathbf{k}_n) \eqcolon 
Q(z,\mathbf{k}_n)
$$
is no greater than $h$.
The degrees in $z$ of the following alternating sums
$$
\sum_{j>q} (-1)^j Q(z,\mathbf{k}^j_{n+1}) 
\text{ and } 
\sum_{j<q} (-1)^j Q(z,\mathbf{k}^j_{0,n}) 
$$
are less than $h$. 
\end{lemma}


\begin{proof} 

We prove the following generalization:
Suppose $\mathbf{b}_h \in \{0,1\}^h$, then the degree of 
$$
\frac{1+p(\delta_{\mathbf{k}_{i+1,n}}) z}{1+p(\Delta_{{\mathbf{k}_{i+1,n}}}) z}
\frac{1+p(\delta_{\mathbf{m}_h}) z}{1+p(\Delta_{\mathbf{m}_h}) z}  \overline \GT_{h,q,\mathbf{b}_h}(\mathbf{k}_n;\mathbf{m}_h) \eqcolon Q_{\mathbf{b}_h}(z,\mathbf{k}_n) 
$$
as a polynomial in $z$ is no greater than $h-d$, where $d$ is the number of $1$'s at the end of $\mathbf{b}_h$. 
The degrees in $z$ of the following alternating sums
$$
\sum_{j>q} (-1)^j Q_{\mathbf{b}_h}(z,\mathbf{k}^j_{n+1}) 
\text{ and } 
\sum_{j<q} (-1)^j Q_{\mathbf{b}_h}(z,\mathbf{k}^j_{0,n}) 
$$
are less than $h$. The first expression vanishes if $b_h=0$, while the second expression vanishes if $b_h=1$.

The assertion of the lemma follows immediately from this as 
\begin{multline*} 
Q(z,\mathbf{k}_n) = \frac{1+p(\delta_{\mathbf{m}_h}) z}{1+p(\Delta_{\mathbf{m}_h}) z} Q(z,\mathbf{k}_n) = 
\sum_{{\mathbf b}_h \in \{0,1\}^h} \frac{1+p(\delta_{\mathbf{m}_h}) z}{1+p(\Delta_{\mathbf{m}_h}) z}  \frac{1+p(\delta_{\mathbf{k}_{i+1,n}}) z}{1+p(\Delta_{{\mathbf{k}_{i+1,n}}}) z} \overline \GT_{h,q,\mathbf{b}_h}(\mathbf{k}_n;\mathbf{m}_h) \\ = \sum_{{\mathbf b}_h \in \{0,1\}^h} Q_{\mathbf{b}_h}(z,\mathbf{k}_n),
\end{multline*} 
where the first equality follows from the fact that $\frac{1+p(\delta_{\mathbf{m}_h}) z}{1+p(\Delta_{\mathbf{m}_h}) z}$ acts like the identity on functions independent of $\mathbf{m}_h$ and the second equality follows from \eqref{decomposition}.

We refer to the fact that the degree of $\sum_{j>q} (-1)^j Q_{\mathbf{b}_h}(z,\mathbf{k}^j_{n+1})$ in $z$ is less than $h$ as the \Dfn{upper alternating condition}; analogously, we refer to the fact that the degree of $\sum_{j<q} (-1)^j Q_{\mathbf{b}_h}(z,\mathbf{k}^j_{0,n})$ in $z$ is less than $h$ as the \Dfn{lower alternating condition}. In case of such a sum being zero, we will refer to the \Dfn{strong} upper or lower alternating condition, respectively. The strong condition will always be deduced from Lemma~\ref{sumtozero}. 

The proof of the generalization is by induction with respect to $h$. This assertion is obviously 
true for $h=0$, which is the base case of our induction.

We use Proposition~\ref{factor} as follows:
\begin{multline} 
\label{twofactors}
\frac{1+p(\delta_{\mathbf{k}_{i+1,n}}) z}{1+p(\Delta_{{\mathbf{k}_{i+1,n}}}) z}
\frac{1+p(\delta_{\mathbf{m}_h}) z}{1+p(\Delta_{\mathbf{m}_h}) z}  \overline \GT_{h,q,\mathbf{b}_h}(\mathbf{k}_n;\mathbf{m}_h) \\
= \left[ \frac{1+p(\delta_{\mathbf{m}_X}) z}{1+p(\Delta_{\mathbf{m}_X}) z}  \prod_{j=1}^r \Delta_{m_{x_j}}^{h-x_j}
\overline \GT_h(\mathbf{k}_i,\mathbf{m}_{X}) \right] \\ 
\times \left[\frac{1+p(\delta_{\mathbf{k}_{i+1,n}}) z}{1+p(\Delta_{\mathbf{k}_{i+1,n}}) z}
\frac{1+p(\delta_{\mathbf{m}_Y}) z}{1+p(\Delta_{\mathbf{m}_Y}) z} \prod_{j=1}^s (-\Delta_{m_{y_j}})^{h-y_j} \overline \GT_h(\mathbf{m}_{Y},\mathbf{k}_{i+1,n}) \right]. 
\end{multline} 
We consider each of the two factors separately and start with the second since this is easier.

We interchange the operators $\frac{1+p(\delta_{\mathbf{k}_{i+1,n}}) z}{1+p(\Delta_{\mathbf{k}_{i+1,n}}) z}
\frac{1+p(\delta_{\mathbf{m}_Y}) z}{1+p(\Delta_{\mathbf{m}_Y}) z}$  and $\prod\limits_{j=1}^s (-\Delta_{m_{y_j}})^{h-y_j}$. Now Lemma~\ref{fund} implies that 
$$ 
\frac{1+p(\delta_{\mathbf{k}_{i+1,n}}) z}{1+p(\Delta_{\mathbf{k}_{i+1,n}}) z}
\frac{1+p(\delta_{\mathbf{m}_Y}) z}{1+p(\Delta_{\mathbf{m}_Y}) z} 
\overline \GT_h(\mathbf{m}_{Y},\mathbf{k}_{i+1,n}) = \overline \GT_h(\mathbf{m}_{Y},\mathbf{k}_{i+1,n}), 
$$
see also the proof of Lemma~\ref{one}. Thus, we have actually independence of $z$ for this factor. 

If $b_h=0$, then $y_s \not= h$, so that $h-y_j >0$ for all $j$. 
Using Lemma~\ref{eat}, it follows that the second factor on the right-hand side of \eqref{twofactors} is equal to 
$$
\sum_{\mathbf{l}_{i+1,n-1} \prec \mathbf{k}_{i+1,n}}
\prod_{j=1}^s (-\Delta_{m_{y_j}})^{h-y_j-1} \overline \GT_{h-1}(\mathbf{m}_{Y},\mathbf{l}_{i+1,n-1}).
$$
Using Lemma~\ref{sumtozero}, the strong upper alternating condition follows in this case.

If, on the other hand, $b_h=1$, we will show that the first factor is of degree less than $h$ in $z$, which will also imply the upper alternating condition.

As for the first factor, we first consider the case $b_h=1$, which implies
 $x_1 \not= h$ and thus we have $h-x_j > 0$ for all $j$. Again by Lemma~\ref{eat}, we have  
\begin{equation} 
\label{step}
\prod_{j=1}^r \Delta_{m_{x_j}}^{h-x_j} \overline 
\GT_h(\mathbf{k}_i,\mathbf{m}_{X})  \\ = 
\sum_{\mathbf{l}_{i-1} \prec \mathbf{k}_{i}}
\prod_{j=1}^r \Delta_{m_{x_j}}^{h-x_j-1} 
\overline \GT_{h-1}(\mathbf{l}_{i-1},\mathbf{m}_{X}).
\end{equation} 
By induction, 
$$
\frac{1+p(\delta_{\mathbf{m}_X}) z}{1+p(\Delta_{\mathbf{m}_X}) z} \overline \GT_{h-1}(\mathbf{l}_{i-1},\mathbf{m}_{X})
$$
is a polynomial in $z$ of degree no greater than $h-1-d$, where $d$ is the number of trailing $1$'s in $\mathbf{b}_{h-1}$. By \eqref{step}, the first factor in \eqref{twofactors} 
is equal to 
$$
\sum_{\mathbf{l}_{i-1} \prec \mathbf{k}_{i}}
\prod_{j=1}^r \Delta_{m_{x_j}}^{h-x_j-1} \frac{1+p(\delta_{\mathbf{m}_X}) z}{1+p(\Delta_{\mathbf{m}_X}) z}
\overline \GT_{h-1}(\mathbf{l}_{i-1},\mathbf{m}_{X}).
$$
It follows that this is a polynomial $z$ of degree no greater than $h-d$, where $d$ is the number of trailing $1$'s in $\mathbf{b}_{h}$. The strong lower alternating condition follows from Lemma~\ref{sumtozero}. 

It remains to consider the case $b_h=0$. We have $x_1=h$. We write 
$$
\frac{1+ p(\delta_{m_h}) z }{1+ p(\Delta_{m_h}) z} = 1 + p(\delta_{m_h}) z  - p(\Delta_{m_h}) z \frac{1+ p(\delta_{m_h}) z }{1+ p(\Delta_{m_h}) z}
$$
and consider the terms 
$$1, \quad p(\delta_{m_h}) z, \quad  -p(\Delta_{m_h}) z  \frac{1+ p(\delta_{m_h}) z}{1+ p(\Delta_{m_h}) z}$$ 
separately when applied to 
$$\frac{1+p(\delta_{\mathbf{m}_{X \setminus \{h\}}}) z}{1+p(\Delta_{\mathbf{m}_{X \setminus \{h\}}}) z} \prod_{j=1}^r \Delta_{m_{x_j}}^{h-x_j} \overline 
\GT_h(\mathbf{k}_i,\mathbf{m}_{X}).$$ 

In the case of the term $1$, $m_h$ can be interpreted as $k_{i+1}$, and we can argue as in the case $b_h=1$, using that 
clearly $h-x_j >0$ for all $j \not=1$. In this 
case, the polynomial is of degree at most $h-1$, which implies the lower alternating condition for this term.

In the case of the term $p(\delta_{m_h}) z$, we can use the fact that $p(x)$ is without constant term to pull out $\delta_{m_h}$; the remaining part of 
$p(x)$ will have no effect on what follows as it does not involve $z$. Now we can use an analogous representation as in \eqref{step} to reduce to the case $h-1$ and then use induction. The analogous expression is
\begin{equation}
\label{analogous} 
\delta_{m_h} \prod_{j=2}^r \Delta_{m_{x_j}}^{h-x_j} \overline 
\GT_h(\mathbf{k}_i,\mathbf{m}_{X})  \\ = 
- \e_{m_h}^{-1} \sum_{\mathbf{l}_{i-1} \prec \mathbf{k}_{i}}
\prod_{j=2}^r \Delta_{m_{x_j}}^{h-x_j-1} 
\overline \GT_{h-1}(\mathbf{l}_{i-1},\mathbf{m}_{X}).
\end{equation} 
 Because of the extra $z$, the polynomial 
can in principal be of degree $h$ in $z$. However, for this term we have the strong lower alternating 
condition (although we are in the case $b_h=0$), which follows when applying Lemma~\ref{sumtozero} to 
\eqref{analogous}, and thus this term does not contribute to the lower alternating condition in total.

The situation is similar for the term $-p(\Delta_{m_h}) z  \frac{1+ p(\delta_{m_h}) z}{1+ p(\Delta_{m_h}) z}$, where here the $m_h$ plays a different role compared to the case we have just discussed. More precisely, we again 
use Lemma~\ref{eat} to see that 
$$
\Delta_{m_h} \prod_{j=2}^r \Delta_{m_{x_j}}^{h-x_j} \overline 
\GT_h(\mathbf{k}_i,\mathbf{m}_{X})  \\ = 
 \sum_{\mathbf{l}_{i-1} \prec \mathbf{k}_{i}}
\prod_{j=2}^r \Delta_{m_{x_j}}^{h-x_j-1} 
\overline \GT_{h-1}(\mathbf{l}_{i-1},\mathbf{m}_{X}),
$$
where the extra $\Delta_{m_h}$ is borrowed from $p(\Delta_{m_h})$. Now the assertion follows as the degree of 
$$
\frac{1+p(\delta_{\mathbf{m}_{X \setminus \{h\}}}) z}{1+p(\Delta_{\mathbf{m}_{X \setminus \{h\}}}) z} 
\frac{1+ p(\delta_{m_h}) z}{1+ p(\Delta_{m_h}) z}
\overline \GT_{h-1}(\mathbf{l}_{i-1},\mathbf{m}_{X}) = 
\frac{1+p(\delta_{\mathbf{m}_{X}}) z}{1+p(\Delta_{\mathbf{m}_{X}}) z} 
\overline \GT_{h-1}(\mathbf{l}_{i-1},\mathbf{m}_{X})
$$
in $z$ is by induction no greater than $h-1$. For this term we also have the strong lower alternating condition for the same reason as in the previous case.
\end{proof}

\subsection{The case \texorpdfstring{$n+h$}{n+h} being odd} 
\label{odd} 

In the same manner as in the proof of Theorem~\ref{GT}, one can show that 
$$
\MT(h;k_1,\ldots,k_n) = \Delta^h_{k_{n+1}} \MT(h;k_1,\ldots,k_n).
$$
The assertion follows after proving that 
$$
\prod_{i=1}^n \st_{k_i,k_{n+1}}^{-1} A_{k_i,k_{n+1}}
$$
acts like the identity on 
$$
\GT(h;k_1,\ldots,k_n).
$$
For 
$
\prod\limits_{i=1}^n A_{k_i,k_{n+1}},
$
this follows from Lemma~\ref{one}; for $\prod\limits_{i=1}^n \st_{k_i,k_{n+1}}^{-1}$, this is obvious, as for each $i \in [n]$, $\st_{k_i,k_{n+1}}^{-1}$ 
acts as the identity on functions independent of $k_{n+1}$.

\section{A conjecture on the annihilator ideal \texorpdfstring{of $\GT_h(\mathbf{k}_n)$}{}}
\label{sec:conjecture} 

In this section, we consider the following ideal
$$
\mathcal{I}_{h,n} = \{ R(\mathbf{X}_n) \in \mathbb{Q}[\mathbf{X}_n]  | R(\Delta_{\mathbf{k}_n}) \GT_h(\mathbf{k}_n)=0  \},
$$
where $h, n$ are non-negative integers with 
$n \ge h$.
This ideal generalizes the annihilator ideal $\mathcal{I}_n$ of $\prod_{1 \le i < j \le n} \frac{k_j-k_i + j -i}{j-i} = \GT_{n} (\mathbf{k}_n)$.
As discussed towards the end of Section~\ref{sec:main}, we have the following generating set of $\mathcal{I}_n$:
\begin{equation}
\label{generator}  
\mathcal{I}_n = \langle e_r(\mathbf{X}_n) | r \ge 1 \rangle.
\end{equation} 
We indicated the significance of $\mathcal{I}_n$ in the study of the coinvariant algebra and raised the question whether this study can be generalized to $\mathcal{I}_{h,n}$.

The purpose of this section is to provide a conjectural characterization of $\mathcal{I}_{h,n}$ in terms of generators. We have just presented the generators for the extreme case $h=n$, but the situation is even easier for the other extreme case $h=0$: As   
$\GT_{0} (\mathbf{k}_n)=1$, the ideal $\mathcal{I}_{0,h}$ is generated by $X_1,\ldots,X_n$. Also note that, since $\GT_{n-1} (\mathbf{k}_n) = \prod_{1 \le i < j \le n} \frac{k_j-k_i + j -i}{j-i}$, the next-to-extreme case $h=n-1$ is understood as well.

There are certain ways to obtain polynomials in $\mathcal{I}_{h,n}$ from polynomials in $\mathcal{I}_{h',n'}$ for some $h',n'$ with $h'+n' < h + n$, which are explained next.

\medskip

{\bf First closure property.}
The first possibility is a direct consequence of the following identity, which is easily deduced from Lemma~\ref{eat}:
$$
\Delta_{\mathbf{k}_{i-1}} \Delta_{\mathbf{k}_{i+1,n}} \overline{\GT}_h(\mathbf{k}_n) =
(-1)^{i-1} \overline{\GT}_h(\mathbf{k}^i_n).
$$
It implies that, for any $i \in [n]$, 
we have 
\begin{equation*}
R(\mathbf{X}_{n-1}) \in \mathcal{I}_{h-1,n-1} \Rightarrow  \left[ \prod_{1 \le j \le n \atop j \not=i} X_j  \right] R(\mathbf{X}_n^i) \in \mathcal{I}_{h,n}.
\end{equation*} 
In the case of $h=n$, since $\mathcal{I}_n = \mathcal{I}_{n,n}$, this reduces to
\begin{equation*}
	R(\mathbf{X}_{n-1}) \in \mathcal{I}_{n-1} \Rightarrow   \left[ \prod_{1 \le j \le n, j \not=i} X_j \right] R(\mathbf{X}_n^i) \in \mathcal{I}_{n}.
\end{equation*}
The following identity 
\begin{multline*} 
\left[ \prod_{j=1}^{n-1}  X_j \right]  e_r(X_1,\ldots,X_{n-1})\\
= \left[ \prod_{j=1}^{n-1} X_j \right] \sum_{s=1}^{r} (-1)^{r-s} X_n^{r-s} e_s(X_1,\ldots,X_n) + (-1)^r X_n^{r-1} e_n(X_1,\ldots,X_n) 
\end{multline*} 
provides an alternative proof of this closure property for $i=n$ in the case $h=n$, using \eqref{generator}; for general $i$, the assertion then follows by symmetry.

\medskip

{\bf Second closure property.} The second possibility also follows immediately from Lemma~\ref{eat}. Namely, we have 
$$
\Delta_{k_1}^h \overline{\GT}_h(\mathbf{k}_n) = \overline{\GT}_h(\mathbf{k}_{2,n}) 
\quad \text{and} \quad 
\Delta_{k_n}^h \overline{\GT}_h(\mathbf{k}_n) = \overline{\GT}_h(\mathbf{k}_{n-1}).
$$
From this, we conclude
\begin{equation*}
R(X_1,\ldots,X_{n-1}) \in \mathcal{I}_{h,n-1} \Rightarrow R(X_1,\ldots,X_{n-1}) X_n^h, X_1^h R(X_2,\ldots,X_{n})   \in \mathcal{I}_{h,n}.
\end{equation*}
 
 In the case $h=n-1$, since $\mathcal{I}_{n-1,n-1} = \mathcal{I}_{n-1}$ and $\mathcal{I}_{n-1,n} = \mathcal{I}_n$, this establishes 
 $$
 R(X_1,\ldots,X_{n-1}) \in \mathcal{I}_{n-1} \Rightarrow R(X_1,\ldots,X_{n-1}) X_n^{n-1}, X_1^{n-1} R(X_2,\ldots,X_{n})   \in \mathcal{I}_{n}.
 $$
Again, the following identity 
 $$
 e_r(X_1,\ldots,X_{n-1}) X_n^{n-1} = \sum_{s=r+1}^{n} (-1)^{s+r+1} X_n^{n-1+r-s} e_s(X_1,\ldots,X_n)
 $$
 provides an alternative proof of this closure property based on \eqref{generator}.
 
\medskip

{\bf Third closure property.} The third closure property is a consequence of the conjectural identity stated next: 
Computer experiments suggest
\begin{equation}
\label{first3} 
(\e_{\mathbf{k}_{i+1,n}} - \id)^{\binom{h+1}{2}-\binom{h-i+1}{2}-\binom{h+i-n+1}{2}} \overline{\GT}_h(\mathbf{k}_n) \propto \overline{\GT}_{\min(h,i)}(\mathbf{k}_i) \overline{\GT}_{\min(h,n-i)}(\mathbf{k}_{i+1,n}),
\end{equation} 
where $\propto$ indicates that the two sides are equal up to a constant that only depends on $n$ and $i$. Note that we set $\binom{m}{2}=0$ if $m<2$. 

The proof of the identity is likely to be based on identities of the following type: 
\begin{multline*}
\left( \e_{\mathbf{k}_{i+1,n}} - \id \right)^q  \sum_{\mathbf{l}_{n-1} \prec \mathbf{k}_n} a(\mathbf{l}_{n-1}) = 
\sum_{\mathbf{l}_{n-1} \prec \mathbf{k}_n} (\e_{\mathbf{l}_{i,n-1}} - \id)^{q} a(\mathbf{l}_{n-1}) \\ + 
\sum_{p=0}^{q-1} \sum_{\mathbf{l}_{i-1} \prec \mathbf{k}_{i}} 
\sum_{\mathbf{l}_{i+1,n-1} \prec \mathbf{k}_{i+1,n}}  \e_{\mathbf{l}_{i+1,n-1}} 
(\e_{\mathbf{l}_{i+1,n-1}} - \id)^{p}  (\e_{k_{i}} \e_{\mathbf{l}_{i+1,n-1}} - \id)^{q-1-p}  a(\mathbf{l}_{i-1},k_{i},\mathbf{l}_{i+1,n-1}).
\end{multline*}  
The latter identity can be proved by induction with respect to $q$. To prove \eqref{first3}, one possibly has to derive similar formulas for 
\begin{equation*}
\Delta_{k_i}^d (\e_{\mathbf{k}_{i+1,n}} - \id)^{\binom{h+2}{2}-1-d} \overline{\GT}_h(\mathbf{k}_n) \quad \text{and} \quad 
 \Delta_{k_{i+1}}^d (\e_{\mathbf{k}_{i+1,n}} - \id)^{\binom{h+2}{2}-1-d} \overline{\GT}_h(\mathbf{k}_n)
\end{equation*}
if $d \ge h$. For $d=h$, this follows immediately from \eqref{first3} and Lemma~\ref{eat}, while for $d>2h$ the expression vanishes as the degree of  $\overline{\GT}_h(\mathbf{k}_n)$ in $k_i$ is no greater than $2h$. 

Let $R_1(X_1,\ldots,X_i) \in \mathcal{I}_{\min(h,i),i}$ and 
$R_2(X_1,\ldots,X_{n-i}) \in \mathcal{I}_{\min(h,n-i),n-i}$. Then, the identity in \eqref{first3} implies that 
$$
\left( \prod_{j=i+1}^n (1+X_i) - 1 \right)^{\binom{h+1}{2}-\binom{h-i+1}{2}-\binom{h+i-n+1}{2}} R_1(X_1,\ldots,X_i)   \in \mathcal{I}_{h,n} 
$$
as well as 
$$
\quad \left( \prod_{j=i+1}^n (1+X_i) - 1 \right)^{\binom{h+1}{2}-\binom{h-i+1}{2}-\binom{h+i-n+1}{2}} R_2(X_{i+1},\ldots,X_n)  \in \mathcal{I}_{h,n}.
$$
We refer to this as the third closure property. Also note that the second closure property is the special case $i=1$ and $i=n$ of the third closure property.

\medskip

Finally, we identify the following elements of $\mathcal{I}_{h,n}$.

\begin{prop}
\label{list}
The following holds:
\begin{enumerate}
\item\label{diff} We have 
$e_i(X_1+1,\ldots,X_n+1) - e_{n-i}(X_1+1,\ldots,X_n+1) \in \mathcal{I}_{h,n}$
for $i \in \{0,1,\ldots,n\}$ as well as 
$
\sum\limits_{i=0}^n (-1)^i e_i(X_1+1,\ldots,X_n+1) \in \mathcal{I}_{h,n}.
$ 
\item\label{degreeGT} We have 
$$
X_1^{h+1},X_n^{h+1} \in \mathcal{I}_{h,n} \quad \text{and} \quad X_2^{2h+1},\ldots,X_{n-1}^{2h+1} \in \mathcal{I}_{h,n}.
$$
\end{enumerate} 
\end{prop} 

\begin{proof}
\emph{Re \eqref{diff}}: This can be deduced using Lemma~\ref{urbanrenewal}, see also the proof of Lemma~\ref{fund}.

\emph{Re \eqref{degreeGT}}: This follows as the degree of $\overline{\GT}_h(\mathbf{k}_n)$ in $k_1$ and $k_n$ is $h$ and the degree 
in $k_i$ is $2h$ for $i=2,\ldots,n-1$. 
\end{proof} 

\begin{conj}
\label{anni}
The family $\mathcal{I}_{h,n}$ of polynomial ideals in the variables $X_1,\ldots,X_n$ is the minimal family with the following properties:
\begin{enumerate} 
\item $\mathcal{I}_{0,n} = \langle X_i | i=1,\ldots,n \}$.
\item $\mathcal{I}_{n-1,n} = \mathcal{I}_{n,n} = \langle e_i(\mathbf{X}_n) | i \ge 1 \rangle$.
\item The family $\mathcal{I}_{h,n}$, $0 \le h \le n$, satisfies all three closure properties.
\item The ideal $\mathcal{I}_{h,n}$ contains all polynomials listed in Proposition~\ref{list}.
\end{enumerate} 
\end{conj}

Note that if the conjecture is true then it provides a generating set for all $\mathcal{I}_{h,n}$, inductively with respect to $h+n$.

\bibliographystyle{alpha}
\bibliography{asmpp.bib}

\begin{thebibliography}{MRR86}

\bibitem[ACJ23]{chhita}
A.~Ayyer, S.~Chhita, and K.~Johansson.
\newblock G{OE} fluctuations for the maximum of the top path in alternating
  sign matrices.
\newblock {\em Duke Math. J.}, 172(10):1961--2014, 2023.

\bibitem[Ber09]{bergeron}
F.~Bergeron.
\newblock {\em Algebraic combinatorics and coinvariant spaces}.
\newblock CMS Treatises in Mathematics. Canadian Mathematical Society, Ottawa,
  ON; A K Peters, Ltd., Wellesley, MA, 2009.

\bibitem[Fis05]{Fis05}
I.~Fischer.
\newblock A method for proving polynomial enumeration formulas.
\newblock {\em J. Combin. Theory Ser. A}, 111:37--58, 2005.

\bibitem[Fis06]{Fis06}
I.~Fischer.
\newblock The number of monotone triangles with prescribed bottom row.
\newblock {\em Adv. in Appl. Math.}, 37(2):249--267, 2006.

\bibitem[Fis07]{Fis07}
I.~Fischer.
\newblock A new proof of the refined alternating sign matrix theorem.
\newblock {\em J. Combin. Theory Ser. A}, 114(2):253--264, 2007.

\bibitem[Fis12]{Fis12}
I.~Fischer.
\newblock Sequences of labeled trees related to {Gelfand--Tsetlin} patterns.
\newblock {\em Adv. in Appl. Math.}, 49:165--195, 2012.

\bibitem[Fis18]{fischergog}
I.~Fischer.
\newblock Constant term formulas for refined enumerations of {G}og and {M}agog
  trapezoids.
\newblock {\em J. Combin. Theory Ser. A}, 158:560--604, 2018.

\bibitem[Fis23]{dominolozenge}
I.~Fischer.
\newblock Computations versus bijections for tiling enumeration.
\newblock {\em Adv. in Appl. Math.}, 142:Paper No. 102427, 49, 2023.

\bibitem[Fis24]{bounded}
I.~Fischer.
\newblock Bounded {L}ittlewood identity related to alternating sign matrices.
\newblock {\em Forum Math. Sigma}, 12:Paper No. e124, 2024.

\bibitem[FK20a]{PartI}
I.~Fischer and M.~Konvalinka.
\newblock A bijective proof of the {ASM} theorem, part {I}: the operator
  formula.
\newblock {\em Electron. J. Combin.}, 27(3):Paper 3.35, 29 pp.\ (electronic),
  2020.

\bibitem[FK20b]{cube}
I.~Fischer and M.~Konvalinka.
\newblock The mysterious story of square ice, piles of cubes, and bijections.
\newblock {\em Proc. Natl. Acad. Sci. USA}, 117 (38):23460--23466, 2020.

\bibitem[FSA23]{nASMDPP}
I.~Fischer and F.~Schreier-Aigner.
\newblock The relation between alternating sign matrices and descending plane
  partitions: {$n + 3$} pairs of equivalent statistics.
\newblock {\em Adv. Math.}, 413:Paper No. 108831, 47, 2023.

\bibitem[Gan23]{gangl}
M.. Gangl.
\newblock Alternating sign pentagons and {Magog} pentagons.
\newblock {\em arXiv:2309.15043}, 2023.

\bibitem[GT50]{gelfandUni}
I.~M. Gelfand and M.~L. Tsetlin.
\newblock Finite-dimensional representations of the group of unimodular
  matrices.
\newblock {\em Doklady Akad. Nauk SSSR (N.S.)}, 71:825--828, 1950.
\newblock (Russian).

\bibitem[GV85]{GesVie85}
I.~Gessel and G.~Viennot.
\newblock Binomial determinants, paths, and hook length formulae.
\newblock {\em Adv. Math.}, 58(3):300--321, 1985.

\bibitem[GV89]{GesVie89}
I.~Gessel and G.~Viennot.
\newblock Determinants, paths and plane partitions, 1989.

\bibitem[H{\"o}n22]{fourfold}
H.~H{\"o}ngesberg.
\newblock A fourfold refined enumeration of alternating sign trapezoids.
\newblock {\em Electron. J. Combin.}, 29(3):Paper No. 3.42, 27, 2022.

\bibitem[Kra98]{kratthook1}
C.~Krattenthaler.
\newblock An involution principle-free bijective proof of {S}tanley's
  hook-content formula.
\newblock {\em Discrete Math. Theor. Comput. Sci.}, 3(1):11--32, 1998.

\bibitem[Kra99]{kratthook2}
C.~Krattenthaler.
\newblock Another involution principle-free bijective proof of {S}tanley's
  hook-content formula.
\newblock {\em J. Combin. Theory Ser. A}, 88(1):66--92, 1999.

\bibitem[Lin73]{Lin73}
B.~Lindstr{\"o}m.
\newblock On the vector representations of induced matroids.
\newblock {\em Bull. London Math. Soc.}, 5:85--90, 1973.

\bibitem[MRR86]{MilRobRum86}
W.~Mills, D.~P. Robbins, and H.~Rumsey, Jr.
\newblock Self-complementary totally symmetric plane partitions.
\newblock {\em J. Combin. Theory Ser. A}, 42(2):277--292, 1986.

\bibitem[Oka06]{Oka06}
S.~Okada.
\newblock Enumeration of symmetry classes of alternating sign matrices and
  characters of classical groups.
\newblock {\em J. Algebraic Combin.}, 23(1):43--69, 2006.

\bibitem[RW83]{remmel1983}
J.~B. Remmel and R.~Whitney.
\newblock A bijective proof of the hook formula for the number of column strict
  tableaux with bounded entries.
\newblock {\em European J. Combin.}, 4(1):45--63, 1983.

\bibitem[Sta79]{stan79}
R.~P. Stanley.
\newblock Invariants of finite groups and their applications to combinatorics.
\newblock {\em Bull. Amer. Math. Soc. (N.S.)}, 1(3):475--511, 1979.

\bibitem[Sta99]{Sta99}
R.~Stanley.
\newblock {\em Enumerative combinatorics. {V}olume 2}.
\newblock Cambridge Studies in Advanced Mathematics~62. Cambridge University
  Press, Cambridge, 1999.

\bibitem[Ste64]{steinberg}
R.~Steinberg.
\newblock Differential equations invariant under finite reflection groups.
\newblock {\em Trans. Amer. Math. Soc.}, 112:392--400, 1964.

\bibitem[Ste90]{ste90}
J.~Stembridge.
\newblock Nonintersecting paths, {P}faffians, and plane partitions.
\newblock {\em Adv. Math.}, 83(1):96--131, 1990.

\bibitem[Wei21]{weigandt}
A.~Weigandt.
\newblock Bumpless pipe dreams and alternating sign matrices.
\newblock {\em J. Combin. Theory Ser. A}, 182:Paper No. 105470, 52, 2021.

\end{thebibliography}

\end{document}